\documentclass[10pt,a4paper]{amsart}

\usepackage[left=3.3cm,right=3.3cm,top=4cm,bottom=4cm]{geometry}
\usepackage[utf8]{inputenc}
\usepackage{amsmath}
\usepackage{amsfonts}
\usepackage{amssymb}
\usepackage{amsthm}
\usepackage[all]{xy}
\usepackage{color}
\usepackage{hyperref}
\usepackage{xypic}
\usepackage{mathrsfs}
\usepackage{centernot}
\usepackage{cite}
\usepackage{tikz-cd}

\theoremstyle{plain}
\newtheorem{theorem}{Theorem}

\newtheorem{corollary}[theorem]{Corollary}
\newtheorem{definition}[theorem]{Definition}
\newtheorem{example}[theorem]{Example}

\newtheorem{lemma}[theorem]{Lemma}

\newtheorem{proposition}[theorem]{Proposition}
\newtheorem{remark}[theorem]{Remark}

\DeclareMathOperator{\Gal}{Gal}

\DeclareMathOperator{\val}{val}

\newcommand{\Gm}{\mathbb{G}_m}

\newcommand{\BQ}{{\mathbb {Q}}} \newcommand{\BR}{{\mathbb {R}}}

 \newcommand{\BZ}{{\mathbb {Z}}}

\newcommand{\can}{{\mathrm{can}}}

\theoremstyle{plain}
\newtheorem{thm}{Theorem}[section] \newtheorem{cor}[thm]{Corollary}
\newtheorem{lem}[thm]{Lemma}

\newtheorem{ques}[thm]{Question} \newtheorem{prop}[thm]{Proposition}
 \newtheorem{defn}[thm]{Definition}
 \newtheorem{lem-defn}[thm]{Lemma-Definition}

\numberwithin{equation}{section}

\makeatletter

\newcommand{\Rmnum}[1]{\expandafter\@slowromancap\romannumeral #1@}
\makeatother

\usepackage{enumitem}
\usepackage{booktabs}
\providecommand{\bR}{\mathbb{R}}

\providecommand{\ol}{\overline}
\providecommand{\can}{\operatorname{can}}

\providecommand{\widevol}{\widehat{\operatorname{vol}}_{\chi}}
\providecommand{\bD}{\overline{D}}
\providecommand{\bL}{\overline{L}}
\providecommand{\bN}{\overline{N}}
\providecommand{\bE}{\overline{E}}
\providecommand{\bM}{\overline{M}}
\providecommand{\muess}{\mu^{\mathrm{ess}}}
\providecommand{\muabs}{\mu^{\mathrm{abs}}}

\begin{document}
\title{Equidistribution for Semiabelian Varieties over Number Fields}
\author{Wei Xue}
\date{\today}

\begin{abstract}
	K\"uhne established equidistribution for canonical adelic line bundles on
	semiabelian varieties, a setting which need not lie in the quasi-canonical
	range of Yuan--Zhang's theorem.  We study adelic line bundles on semiabelian
	compactifications whose toric part is governed by a toric metrized divisor.
	The main unconditional result is generic equidistribution when the toric metric
	is monocritical and arithmetically \(T\)-effective.  The Bogomolov conjecture
	holds unconditionally for K\"uhne's standard \((\mathbb P^1)^t\)-compactification;
	for a general fan it holds conditionally on the independently established
	curvature--Haar slice formula, including the lattice-index normalization and
	the assertion that no additional boundary mass occurs.
	The proof isolates the asymptotic estimates in K\"uhne's argument and
	reinterprets them as estimates along explicit compression paths.  This gives a
	comparison mechanism between canonical, quasi-canonical, and more general
	toric metrics.  For the Bogomolov conjecture, the quasi-canonical case is
	handled by a K\"uhne local-trivialization transport package.  After fixing a
	single theta-factor convention for the local trivializations, the proof checks
	the Picard-zero theta factors under K\"uhne's operations and obtains the
	Bogomolov theorem for the metric class treated in this paper, namely the
	monocritical and arithmetically \(T\)-effective toric metrics, through the
	quasi-canonical replacement argument.
\end{abstract}
\maketitle

\section*{Theta-Factor Convention}

We use the following convention throughout the paper.  Let
\[
0\longrightarrow T_G\longrightarrow G\xrightarrow{\rho_G}A_G\longrightarrow0
\]
be a semiabelian variety, and let \(M_G=X^*(T_G)\).  For
\(m\in M_G\), the symbol
\[
\mathcal H_{G,m}\in\operatorname{Pic}^0(A_G)
\]
denotes the rigidified Picard-zero line bundle whose associated
\(\mathbb G_m\)-torsor is obtained by pushing out the semiabelian extension
through the character \(\chi^m:T_G\to\mathbb G_m\).  In local
trivializations we use the denominator convention: if \(g_m\) is the
semiabelian character section attached to \(m\), and if \(f_{m,v}\) is the
theta representative of a local rational section of \(\mathcal H_{G,m}\),
then the toric coordinate is written
\[
\psi_{j,m}(y)=\frac{g_m(y)}{f_{m,v}(\rho_G(y))}.
\]
Some sources, depending on the contravariant convention for line bundles and
invertible sheaves, attach the same theta expression to the dual line bundle.
With the convention above this is recorded by
\[
\mathcal H_{G,-m}\simeq\mathcal H_{G,m}^{\vee}.
\]
Thus a dual appears in the notation only when the pulled-back character is
actually \(-m\).  Since all such Picard-zero factors carry their canonical
flat metrics at archimedean places, this convention does not change any
Chern-current or Haar-measure statement; it only fixes the signs in the local
formulas.

\begingroup
\small
\tableofcontents
\endgroup
\clearpage
\section{Introduction}

Heights and equidistribution form one of the main bridges between arithmetic
positivity and the geometry of special subvarieties.  Weil's theory of local
heights \cite{We} shows that global height functions are assembled from local
data, while N\'eron's local height theory \cite{Ne} gives the canonical local
functions behind the N\'eron--Tate height on abelian varieties.  These ideas
enter both finiteness results, such as the Mordell--Weil theorem, and
density statements for algebraic points of small height.

The Bogomolov conjecture is a central example.  Its basic assertion is that a
non-torsion subvariety cannot contain a Zariski-dense set of points of
arbitrarily small N\'eron--Tate height.  For curves in their Jacobians it was
proved by Ullmo \cite{Ull} and Zhang \cite{Zha}; Zhang then proved the
corresponding theorem for subvarieties of abelian varieties.  Zhang's proof introduced adelic line
bundles and used equidistribution of small points.  Yuan and Zhang later
extended this adelic framework and the equidistribution theorem to
quasi-projective varieties over global fields \cite{YZ}.

For an integrable adelic line bundle \(\overline L\) on an irreducible
variety \(Z\), its essential minimum is
\[
\mu^{\mathrm{ess}}_{\overline L}(Z)
=\sup_{U\subseteq Z\ {\rm nonempty\ open}}
\inf_{x\in U(\overline K)}h_{\overline L}(x).
\]
It is the smallest height threshold below which algebraic points can be
Zariski dense.  The formal conventions for heights and small sequences are
recalled below.

More precisely, if \(A\) is an abelian variety equipped with a symmetric
ample line bundle and its canonical metric, and \(Z\subseteq A_{\overline K}\)
is irreducible, the classical Bogomolov theorem may be written
\[
\mu^{\mathrm{ess}}_{\overline L}(Z)=0
\quad\Longleftrightarrow\quad
Z\text{ is a torsion subvariety}.
\]
Here a torsion subvariety is a translate of an abelian subvariety by a torsion
point.  In this form the conjecture exhibits geometric rigidity directly:
vanishing of an arithmetic minimum forces the subvariety, up to translation
by a torsion point, to arise from an algebraic subgroup.

The present paper is motivated by the following relative form of the
Bogomolov conjecture.
\begin{ques}
	Let \(\pi:X\to Y\) be a projective flat morphism between quasi-projective
	varieties over a number field \(K\).  Suppose that \(L\) is relatively ample
	and satisfies suitable invariance conditions, for instance that \((X,f,L)\)
	forms a polarized dynamical system over \(Y\) in the sense of \cite{YZ}.  Thus
	one obtains a nef adelic line bundle \(\overline L\) on \(X/\mathbb Z\) with
	\(\mu^{\mathrm{ess}}_{\overline L}(X)=0\).  For a closed
	\(\overline K\)-subvariety \(X_1\subseteq X\), what geometric condition on
	\(X_1\) implies that, if \(X_1\) is not special, then
	\[
	\mu^{\mathrm{ess}}_{\overline L}(X_1)>0?
	\]
\end{ques}

We list four illustrative cases of this framework.
\begin{enumerate}[label=\textup{(\arabic*)},leftmargin=2em]
	\item The classical Bogomolov conjecture for abelian varieties: \(Y\) is a
	point, \(X\) is an abelian variety, and the special subvarieties are torsion
	subvarieties.
	\item K\"uhne's equidistribution theorem in families \cite{Ku4}: the morphism
	\(\pi:X=\mathscr U_g\to Y=\mathbb A_g\) is the universal abelian variety, the
	fibers \(X_{1,y}\) are irreducible curves, and the special condition is that
	each fiber \(X_{1,y}\) is a translate of an elliptic curve.
	\item The relative Bogomolov conjecture \cite{DGH2}: the morphism
	\(\pi:X=\mathcal A\to Y=S\) is a relative abelian scheme, one assumes
	\(\overline{\mathbb Z X_1}=\mathcal A\), and the special condition is governed
	by the expected codimension inequality
	\(\operatorname{codim}_{\mathcal A}X_1\leq\dim S\).
	\item The Bogomolov conjecture for semiabelian varieties, as treated by
	K\"uhne \cite{Ku1}: \(X=G\) is a semiabelian variety, \(Y=A\) is its abelian
	quotient, and the special subvarieties are translates of algebraic subgroups
	of \(G\) by torsion points.
\end{enumerate}

The full relative Bogomolov conjecture remains open.  Gao proved a weaker
relative Manin--Mumford statement \cite{Gao}: if \(X_1\) is not special, then
its torsion points cannot be Zariski dense in \(X_1\).  Equidistribution is the
main proof mechanism in Cases (1), (2), and (4), and it is the point of
departure for the present work on semiabelian varieties.

\begin{ques}
	Let \(\pi:X\to Y\) be a projective flat morphism between quasi-projective
	varieties over a number field \(K\).  Let \(L\) be big and nef, and let
	\(\overline L\) be a nef adelic line bundle on \(X\).  Assume that the
	Deligne pairing
	\[
	\bigl(\langle\overline L,\ldots,\overline L\rangle,Y\bigr)
	\]
	and every fiber pair \((\overline L|_{X_y},X_y)\) satisfy equidistribution.
	What additional hypotheses on \(\overline L\) imply equidistribution for
	\((\overline L,X)\)?
\end{ques}

Existing proofs for Cases (1) and (2) use Yuan--Zhang's equidistribution
theorem \cite{YZ}.  In Case (4), K\"uhne \cite{Ku1} treats a canonical
semiabelian metric: \(X=G\) is semiabelian, \(Y=A\) is its abelian quotient,
and
\[
\overline{L}
=
\overline{G}(\overline{D}^{\mathrm{can}})
\otimes \overline{\pi}^{*}\overline N .
\]
Here \(T\)-effectivity of \(D\) controls the Deligne pairing and leads to
equidistribution.  The key difficulty is that the metric defining
\(\overline L\) is generally not quasi-canonical in the sense required by
Yuan--Zhang; one may have
\[
h_{\overline{L}}(\overline G)<\mu^{\mathrm{ess}}_{\overline{L}}(\overline G),
\]
so Yuan and Zhang's theorem cannot be applied directly.

The present paper studies this gap for toric metrics on semiabelian
compactifications.  Its scope is the class of semipositive, monocritical and
arithmetically \(T\)-effective toric metrics specified below; no assertion is
made for arbitrary semipositive toric metrics.  The precise comparison with
earlier work is given immediately after the main theorems.

\subsection*{Setup}

Let \(G\) be a semiabelian variety over a number field \(K\), with split toric
part \(T\) of dimension \(t\) and maximal abelian quotient
\(\pi:G\to A\) of dimension \(g\).  Let \(X_{\Sigma}\) be a normal proper
toric variety whose open dense torus is \(T\), and let
\(\overline{D}=(D,\{\Vert\cdot\Vert_v\})\) be a toric metrized
\(\mathbb R\)-divisor on \(X_{\Sigma}\).  We write \(\overline G\) for the
toric compactification of \(G\) associated to \(X_{\Sigma}\), with projection
\(\overline\pi:\overline G\to A\).

In Section 4, we construct an adelic line bundle \(\overline G(\overline D)\)
on \(\overline G/\mathbb Z\) via local trivializations valid at every place of
\(K\), extending K\"uhne's archimedean local trivialization to all valuations.
Fix a symmetric ample line bundle \(\overline N\) on \(A\), equipped with its
canonical metric, and define
$\overline{L}:=\overline{G}(\overline{D})\otimes \overline{\pi}^*\overline{N}$.
\begin{thm}[Generic equidistribution]\label{thm:intro-generic-equidistribution}
	Suppose that \(D\) is big and \(T\)-effective, and that
	\(\overline D=(D,\{\Vert\cdot\Vert_v\})\) is a semipositive, monocritical and
	arithmetically \(T\)-effective toric metrized \(\mathbb R\)-divisor in the
	sense of Definitions~\ref{def 3.12} and~\ref{def:arith-T-effective}.  Then
	every generic \(\overline L\)-small sequence satisfies \(v\)-adic
	equidistribution at every place \(v\) of \(K\).  If \(\overline D\) is
	additionally quasi-canonical, the limiting measure is
	\[
	\mu_{\overline{L},v}
	=
	\frac{
		\hat{c}_1(\overline{G}(\overline{D})_v)^t
		\wedge
		\hat{c}_1(\overline{\pi}^*\overline{N}_v)^g
	}{
		\overline{G}(D)^t(\overline{\pi}^*N)^g
	}.
	\]
\end{thm}

\begin{thm}[Bogomolov theorem for the paper's metric class]
	\label{thm:intro-bogomolov}
	Let \(X\subseteq G_{\overline K}\) be irreducible, and assume that \(D\) is ample
	and \(T\)-effective and that \(\overline D\) is semipositive, monocritical and
	arithmetically \(T\)-effective.  At every archimedean place, assume in addition
	either the standard \((\mathbb P^1)^t\)-compactification used by K\"uhne, or
	that the general-fan curvature--Haar slice formula
	\ref{input:general-fan-curvature-haar-slice}, including its lattice-index and
	boundary no-mass assertions, has been independently established for every
	subvariety occurring in the difference-morphism argument.  Put
	\[
	\overline L
	=\overline G(\overline D)\otimes\overline\pi^*\overline N .
	\]
	\begin{enumerate}[label=\textup{(\arabic*)},leftmargin=2em]
		\item If
		\[
		\mu^{\mathrm{ess}}_{\overline{L}}(X)
		=
		\mu^{\mathrm{ess}}_{\overline{L}}(\overline{G}),
		\]
		then \(X\) is a special subvariety.
		\item If \(G(\overline K)\) has \(\overline L\)-special points, then
		\[
		\mu^{\mathrm{ess}}_{\overline{L}}(X)
		=
		\mu^{\mathrm{ess}}_{\overline{L}}(\overline{G})
		\]
		if and only if \(X\) is \(\overline L\)-special.
	\end{enumerate}
\end{thm}

\subsection*{What was known and what is new}

In the notation above, the antecedents divide into four distinct inputs.
Yuan--Zhang \cite{YZ} provide the adelic height and generic-equidistribution
framework, with direct application requiring the relevant quasi-canonical
equality.  K\"uhne \cite{Ku1} proves canonical semiabelian equidistribution
and Bogomolov by local trivializations, Haar slices and difference morphisms.
BGPS \cite{BGPS19} isolate monocritical toric metrics and their critical
tropical data on pure toric varieties.  More recently, Ballay--Sombra
\cite{BS24} develop approximation beyond the quasi-canonical equality case,
whereas Hultberg \cite{Hul24} computes heights, minima and intersections for
toric bundles.  The latter height-theoretic results are not, by themselves,
the strict Bogomolov bridge used here.

The new point of Theorems~\ref{thm:intro-generic-equidistribution} and
\ref{thm:intro-bogomolov} is the interaction between a semiabelian extension
and a non-canonical toric metric.  The proof compresses the metric toward the
quasi-canonical model singled out by its critical tropical center, transports
K\"uhne's Haar-slice and difference-morphism argument from the standard
compactification to a compactification satisfying the stated Haar-slice input,
and
then combines monocritical concentration with restriction to subvarieties.
This yields the metric-dependent implication
\[
\mu^{\mathrm{ess}}_{\overline L}(X)
=\mu^{\mathrm{ess}}_{\overline L}(\overline G)
\quad\Longrightarrow\quad X\text{ is special},
\]
with the converse under the special-point hypothesis stated in
Theorem~\ref{thm:intro-bogomolov}.  Thus the contribution is neither a new
proof of the canonical theorem nor a theorem for all semipositive metrics: it
is the extension to the stated metric class, organized by its critical
tropical data.

\subsection*{Proof strategy}

The strategy is stated only after the results and their relation to the
literature have been made explicit.  The generic equidistribution theorem is
proved internally; the Bogomolov theorem then uses the K\"uhne transport
assembly established in Proposition~\ref{prop:kuhne-section4-transport}.

The proof of the generic equidistribution theorem decomposes into four
distinct stages:
\begin{enumerate}[label=\textup{(\arabic*)},leftmargin=2em]
	\item Reduce to the ample case by a birational modification of \(X_\Sigma\).
	\item Prove the canonical metric case by extending K\"uhne's asymptotic
	estimates to arbitrary ample toric divisors.
	\item Treat quasi-canonical metrics by constructing explicit compression paths
	and proving the corresponding lower-bound, measure, first-variation, and
	arithmetic-volume estimates.
	\item Pass to monocritical metrics using tropical measure theory and
	arithmetical \(T\)-effectivity, reducing the limiting measure to the
	quasi-canonical model.
\end{enumerate}

\section{Conventions and terminology}
\paragraph{Asymptotic notation.}
For two non-negative quantities \(a\) and \(b\), we write \(a\ll b\) if
there is a constant \(c>0\) such that \(a\leq cb\).  If \(c\) is allowed to
depend on auxiliary data, for instance an algebraic variety \(X\), we write
\(a\ll_X b\).  We use \(\gg\) with the analogous meaning.

\paragraph{Number fields.}
Throughout the paper, \(K\subseteq \overline{\mathbb Q}\) is a number field
with ring of integers \(O_K\), and \(S=\operatorname{Spec}(O_K)\).  We denote
by \(\Sigma(K)\) the set of places of \(K\), and by \(\Sigma_f(K)\) and
\(\Sigma_{\infty}(K)\) the non-archimedean and archimedean places,
respectively.  For \(v\in\Sigma(K)\), let \(K_v\) be the \(v\)-adic
completion of \(K\), let \(\mathbb C_v\) be the completion of an algebraic
closure of \(K_v\), let \(O_v\subseteq\mathbb C_v\) be its ring of integers,
let \(\boldsymbol p_v\) be the maximal ideal of \(O_v\), and let
\(k_v=O_v/\boldsymbol p_v\).

\paragraph{Generic sequences.}
Let \(X\) be an algebraic \(K\)-variety.  A sequence
\((x_i)_{i\geq 1}\) of closed points of \(X\) is called \(X\)-generic if no
subsequence is contained in a proper algebraic subvariety of \(X\).  When \(X\)
is clear from context, we simply say generic.

\paragraph{Small sequences.}
Let \(X\) be a quasi-projective algebraic \(K\)-variety, and let
\(\overline L\) be an integrable adelic line bundle on \(X/\mathbb Z\).  A
sequence \((x_i)_{i\geq1}\) of closed points of \(X\) is called
\(\overline L\)-small if
\[
h_{\overline{L}}(x_i)\to\mu^{\mathrm{ess}}_{\overline{L}}(X)
\qquad (i\to \infty).
\]

\section{Preliminary}
In this section, we summarize the basic facts that are essential for the proof.
\subsection{Valuation theory}
A global field \(K\) is either a number field or the function field of a
regular projective curve over a field.  We equip \(K\) with a set of places
\(\Sigma(K)\).  A place \(v\in\Sigma(K)\) consists of an absolute value
\(|\cdot|_v\) on \(K\), together with a positive weight \(n_v\).

For \(K_0=\mathbb Q\), the places are the usual archimedean and \(p\)-adic
absolute values, with all weights equal to \(1\).  For a function field
\(K_0=k(C)\), where \(C\) is a regular projective curve over \(k\), the places
are indexed by closed points \(v_0\in C\).  For
\(\alpha\in k(C)^\times\), we set
\[
|\alpha|_{v_0}=c_k^{-\operatorname{ord}_{v_0}(\alpha)},
\qquad
n_{v_0}=[k(v_0):k],
\]
where \(c_k=e\) if \(k\) is infinite and \(c_k=\#k\) if \(k\) is finite.

Let \(K/K_0\) be a finite extension.  A place \(v\) of \(K\) extending
\(v_0\in\Sigma(K_0)\) is assigned the weight
\[
N^{K}_{v}=[K_v:K_{0,v_0}]n_{v_0},
\qquad
n^{K}_{v}=\frac{N^{K}_{v}}{[K:K_0]}.
\]
We call \(N^{K}_{v}\) the raw local degree and \(n^{K}_{v}\) the normalized
weight.  Unless the superscript ``raw'' is displayed, every height, adelic
intersection and adelic measure in this paper uses normalized weights.  Raw
degrees are only intermediate bookkeeping and are divided by the global
degree before any height is formed.  From now on we suppress the field
superscript and write \(n_v=n_v^K\) for the normalized weight.
Then
\[
\sum_{v\mid v_0}N^{K}_{v}=[K:K_0]n_{v_0},
\qquad
\sum_{v\mid v_0}n^{K}_{v}=n_{v_0},
\qquad
\prod_{v\in\Sigma(K)}|\alpha|_v^{n^{K}_{v}}=1
\quad(\alpha\in K^\times).
\]
\subsection{Toric varieties}
Let \(T\cong\mathbb G_m^t\) be a split torus of dimension \(t\) over \(K\).
Let \(N=\operatorname{Hom}(\mathbb G_m,T)\) be the lattice of cocharacters of
\(T\), let \(M=\operatorname{Hom}(T,\mathbb G_m)=N^\vee\) be the lattice of
characters, and write
\[
N_{\mathbb R}=N\otimes\mathbb R,\qquad
M_{\mathbb R}=M\otimes\mathbb R .
\]

\begin{defn}
	Let \(\mathscr P\) be a non-empty collection of convex subsets of
	\(N_{\mathbb R}\).  We call \(\mathscr P\) a convex subdivision if
	\begin{enumerate}[label=\textup{(\arabic*)},leftmargin=2em]
		\item every face of an element of \(\mathscr P\) belongs to \(\mathscr P\);
		\item any two elements of \(\mathscr P\) are either disjoint or intersect in a
		common face.
	\end{enumerate}
	The support of \(\mathscr P\) is
	\[
	|\mathscr P|=\bigcup_{C\in\mathscr P}C.
	\]
	The subdivision is complete if \(|\mathscr P|=N_{\mathbb R}\).
\end{defn}
\begin{defn}
	A polyhedral complex in \(N_{\mathbb R}\) is a finite convex subdivision
	whose elements are polyhedra.  It is strongly convex if all of its polyhedra
	are strongly convex, and conic if all of its elements are cones.  A strongly
	convex conic polyhedral complex in \(N_{\mathbb R}\) is called a fan.
\end{defn}
Let \(X\) be a proper toric variety over \(K\) with torus \(T\), described by a
complete fan \(\Sigma\) on \(N_{\mathbb R}\).  To each cone
\(\sigma\in\Sigma\) there correspond an open affine subset \(X_\sigma\) and an
orbit \(O(\sigma)\).  For \(\sigma=\{0\}\), the principal open subset \(X_0\)
agrees with the orbit \(O(0)\), and is canonically isomorphic to \(T\).  The
action of \(T\) on \(X\) is denoted by \((t,p)\mapsto t\cdot p\).
\begin{defn}
	Let \(\Sigma\) be a fan in \(N_{\mathbb R}\).  A function
	\(\Psi:|\Sigma|\to\mathbb R\) is called a virtual support function on
	\(\Sigma\) if it is a conic \(H\)-lattice function in the sense of
	\cite[Definition~2.6.6]{BPS14}; equivalently, for every cone
	\(\sigma\in\Sigma\), there is \(m_\sigma\in M\) such that
	\[
	\Psi(u)=\langle m_\sigma,u\rangle
	\qquad (u\in\sigma).
	\]
	The \(m_\sigma\) are called defining vectors of \(\Psi\).  A concave virtual
	support function on a complete fan is called a support function.
	
	A toric \(\mathbb R\)-divisor on the proper toric variety \(X\) is an
	\(\mathbb R\)-divisor invariant under the action of \(T\).  Such a divisor
	\(D\) defines a function \(\Psi_D:N_{\mathbb R}\to\mathbb R\) whose restriction
	to each cone of \(\Sigma\) is linear.  We also attach to \(D\) the stability
	set
	\[
	\Delta_D=\operatorname{stab}(\Psi_D)\subseteq M_{\mathbb R}.
	\]
\end{defn}
In the sequel, toric varieties are assumed to be proper; equivalently,
\(\Sigma\) is a complete fan.
\begin{prop}
	The following is \cite[Propositions 4.7--4.9]{BMPS22}.  Let \(D\) be a toric
	\(\mathbb R\)-divisor on \(X\).  Then:
	\begin{enumerate}[label=\textup{(\arabic*)},leftmargin=2em]
		\item \(D\) is ample if and only if \(\Psi_D\) is strictly concave.
		\item \(D\) is nef if and only if \(\Psi_D\) is concave, equivalently \(D\)
		is globally generated.
		\item \(D\) is big if and only if \(\dim\Delta_D=\dim X\).
		\item \(D\) is pseudo-effective if and only if \(\Delta_D\neq\emptyset\).
		\item \(D\) is \(T\)-effective if and only if \(\Psi_D\leq0\), equivalently
		\(0\in\Delta_D\).
	\end{enumerate}
\end{prop}

Now let \(K\) be a global field.  For each place \(v\in\Sigma(K)\), the
algebraic torus \(T\) has an analytic space \(T^{an}_v\), and we denote its
compact torus by \(\mathbb S_v^{an}\).  At an archimedean place, this compact
torus is isomorphic to \((S^1)^t\).  In the non-archimedean case, it is a
compact analytic group; see \cite[\S4.2]{BPS14}.

\begin{prop}
	By \cite[Chapter 4.1]{BPS14}, there is a commutative diagram
	\begin{equation*}
		\xymatrix{
			&X^{an}_{0,v}\ar[d]^{\rho_0}\ar[ld]_{\val_v}\\
			N_{\mathbb{R}}\ar[r]^{\boldsymbol{e}}&X_0(\mathbb{R}_{\geqslant0})
		}
	\end{equation*}
	which extends to \(X^{an}_v\):
	\begin{equation*}
		\xymatrix{
			&X^{an}_v\ar[d]^{\rho_{\Sigma}}\ar[ld]_{\val_v}\\
			N_{\Sigma}\ar[r]^{\boldsymbol{e}}&X(\mathbb{R}_{\geqslant0})
		}
	\end{equation*}
	where the spaces with corners \(N_{\Sigma}\) and
	\(X(\mathbb{R}_{\geqslant0})\) are compactifications of \(N_{\mathbb R}\) and
	\(X_0(\mathbb R_{\geqslant0})\), respectively.
\end{prop}
\begin{prop}
	By \cite[Definition~4.2.12]{BPS14}, there exists a continuous proper section
	\(\theta_\Sigma\) of \(\rho_\Sigma\)
	whose image is contained in \(\mathbb S_v^{an}\).
\end{prop}
\begin{defn}
	A \(v\)-adically metrized \(\mathbb R\)-divisor
	\(\overline D=(D,\Vert\cdot\Vert)\) on \(X\) is toric if \(D\) is a toric
	\(\mathbb R\)-divisor and its Green function is invariant under the action of
	\(\mathbb S_v^{an}\), equivalently constant on the fibers of \(\rho_0\).  We
	then define
	\[
	\Psi_{\overline D,v}(u)=-g_{\overline D}(p),
	\]
	where \(p\in X^{an}_{0,v}\) is any point with \(\operatorname{val}_v(p)=u\).
	
	A quasi-algebraic metrized \(\mathbb R\)-divisor
	\(\overline D=(D,\{\Vert\cdot\Vert_v\})\) is toric if each
	\((D,\Vert\cdot\Vert_v)\) is toric.  In this case we call \(\overline D\) a
	toric metrized \(\mathbb R\)-divisor.  The canonical \(v\)-adic toric metric
	on \(D\) is characterized by \(\Psi_{\overline D,v}=\Psi_D\).
\end{defn}
\subsubsection*{Corner complexes and dual cells}
The following convex-geometric objects will be used in Chapter~9.  Let
\(\Psi:N_{\mathbb R}\to\mathbb R\) be a continuous piecewise-affine concave
function.  Its \emph{corner complex} \(\mathcal C(\Psi)\) is the polyhedral
complex formed by the maximal domains on which \(\Psi\) is affine and all
their common faces.  Its codimension-one cells are called \emph{walls}, and a
codimension-\(r\) cell is called an \emph{\(r\)-corner}.

If \(\tau\) is a cell and \(u\in\operatorname{relint}(\tau)\), define its
dual cell by
\[
\Delta_\tau:=\partial^+\Psi(u)
=\{m\in M_{\mathbb R}:\Psi(w)\leq
\Psi(u)+\langle m,w-u\rangle\ \text{for all }w\}.
\]
This set is independent of \(u\) in the relative interior of \(\tau\).  If
locally \(\Psi=\min_\alpha(\langle m_\alpha,\cdot\rangle+c_\alpha)\), then
\(\Delta_\tau\) is the convex hull of the slopes of the branches active on
\(\tau\).  Under the usual regularity hypotheses,
\(\dim\Delta_\tau=\operatorname{codim}\tau\).  Thus higher-codimension corner
contributions are governed in general by the volume, or mixed volume, of the
dual cells.  A decomposition into independent one-wall currents is available
only under the additional product-corner condition that the relevant dual
cell is a Minkowski sum of independent segments.

For \(\Psi=\Psi_{\overline D,v}\), these definitions apply to the local toric
weight introduced above.  The critical vector \((u_v)_v\) of a monocritical
metric, defined below, specifies the tropical translations used in the
quasi-canonical replacement; it should not be confused with a cell of
\(\mathcal C(\Psi_{\overline D,v})\).
\begin{prop}
	Let \(D\) be a toric \(\mathbb R\)-divisor on \(X\).
	\begin{enumerate}[label=\textup{(\arabic*)},leftmargin=2em]
		\item For a place \(v\), the correspondence
		\(\Vert\cdot\Vert\mapsto\Psi_{\overline D,v}\) is a bijection between toric
		\(v\)-adic metrics on \(D\) and functions
		\(\Psi:N_{\mathbb R}\to\mathbb R\) such that \(\Psi-\Psi_D\) extends
		continuously to \(N_\Sigma\).
		\item The correspondence
		\[
		\{\Vert\cdot\Vert_v\}\longmapsto \{\Psi_{\overline D,v}\}
		\]
		is a bijection between quasi-algebraic toric metrics on \(D\) and families
		\(\{\Psi_v\}_{v\in\Sigma(K)}\) satisfying the preceding continuity condition
		for every \(v\), and \(\Psi_v=\Psi_D\) for all but finitely many \(v\).
	\end{enumerate}
\end{prop}
\begin{defn}
	Let \(\overline D\) be a toric metrized \(\mathbb R\)-divisor on \(X\).  For
	each \(v\in\Sigma(K)\), the \(v\)-adic roof function of \(D\) is the concave
	function on \(\Delta_D\) defined by
	\[
	\upsilon_{\overline D,v}=\Psi_{\overline D,v}^{\vee}.
	\]
	The global roof function is
	\[
	\upsilon_{\overline D}
	=
	\sum_{v\in\Sigma(K)} n_v\upsilon_{\overline D,v}.
	\]
\end{defn}
\begin{prop}
	The following characterization is \cite[Theorem~6.1]{BMPS22}.  Let \(D\) be a
	toric metrized \(\mathbb R\)-divisor on \(X\), and let
	\(\upsilon_{\overline D}:\Delta_D\to\mathbb R\) be its global roof function.
	Then:
	\begin{enumerate}[label=\textup{(\arabic*)},leftmargin=2em]
		\item \(\overline D\) is ample if and only if \(\Psi_D\) is strictly concave
		on \(\Sigma\), each \(\Psi_{\overline D,v}\) is concave, and
		\(\upsilon_{\overline D}(x)>0\) for all \(x\in\Delta_D\).
		\item \(\overline D\) is nef if and only if each \(\Psi_{\overline D,v}\) is
		concave and \(\upsilon_{\overline D}(x)\geq0\) for all \(x\in\Delta_D\).
		\item \(\overline D\) is big if and only if \(\dim\Delta_D=\dim X\) and
		\(\upsilon_{\overline D}(x)>0\) for some \(x\in\Delta_D\).
		\item \(\overline D\) is pseudo-effective if and only if
		\(\upsilon_{\overline D}(x)\geq0\) for some \(x\in\Delta_D\).
		\item \(\overline D\) is effective if and only if \(0\in\Delta_D\) and
		\(\upsilon_{\overline D,v}(0)\geq0\) for every \(v\).
		\item \(\overline D\) is vertically semipositive if and only if each
		\(\Psi_{\overline D,v}\) is concave.
	\end{enumerate}
\end{prop}
\begin{prop}
	By \cite[Theorem~3.9 and Corollary~3.10]{BGPS15}, let \(\overline D\) be a
	toric metrized \(\mathbb R\)-divisor on \(X\).  Then
	\[
	\mu^{\mathrm{ess}}_{\overline D}(X)
	=
	\mu^{\mathrm{abs}}_{\overline D}(X_0)
	=
	\max_{x\in\Delta_D}\upsilon_{\overline D}(x).
	\]
\end{prop}

Let \(\Delta_{D,\max}\) be the locus where \(\upsilon_{\overline D}\) attains
its maximum.
\begin{defn}
	A vertically semipositive toric metrized \(\mathbb R\)-divisor
	\(\overline D\) is called monocritical if, for every
	\(x\in\Delta_{D,\max}\), the origin is a vertex of the subdifferential
	\(\partial\upsilon_{\overline D}(x)\).
	
	The critical point of \(\overline D\) is the family
	\((u_v)_v\in\bigoplus_v\partial\upsilon_{\overline D,v}(x)\) determined by the
	decomposition
	\[
	0=\sum_v n_vu_v
	\qquad\text{in}\qquad
	\partial\upsilon_{\overline D}(x)
	=
	\sum_v n_v\partial\upsilon_{\overline D,v}(x).
	\]
	\label{def 3.12}\label{def:monocritical}
\end{defn}
\begin{remark}
	The critical point of \(\overline D\) is independent of the choice of
	\(x\in\Delta_{D,\max}\).
\end{remark}
The monocritical condition is therefore a uniqueness condition on the
decomposition of the zero slope at the maximum locus of the global roof
function.
\subsection{Semi-abelian varieties}
\begin{defn}
	A semiabelian variety \(G\) over a field \(K\) is a connected smooth
	commutative algebraic \(K\)-group fitting into an exact sequence
	\[
	0\longrightarrow T\longrightarrow G\longrightarrow A\longrightarrow 0,
	\]
	where \(T\) is a \(K\)-torus and \(A\) is an abelian variety over \(K\).
	This exact sequence defines a Yoneda extension class
	\[
	\eta_G\in\operatorname{Ext}^1_K(A,T).
	\]
\end{defn}
\begin{prop}(The Weil-Barsotti formula)
	There is a canonical isomorphism
	\[
	\operatorname{Ext}^1_K(A,\mathbb G_m)\cong A^\vee(K).
	\]
\end{prop}
If \(T\) is split and \(T\simeq\mathbb G_m^t\), we frequently use
\[
\operatorname{Ext}^1_K(A,\mathbb G_m^t)\cong (A^\vee)^t(K).
\]
\begin{prop}
	The following is \cite[Lemma~1]{Ku2}.  Let \(G\) and \(G'\) be semiabelian
	varieties over \(K\), with toric parts
	\(T,T'\) and abelian quotients \(A,A'\).  For any homomorphism
	\(\varphi:G\to G'\), there are unique homomorphisms
	\(\varphi_{\mathrm{tor}}:T\to T'\) and
	\(\varphi_{\mathrm{ab}}:A\to A'\) such that
	\begin{equation*}
		\xymatrix{
			0\ar[r]&T\ar[r]\ar[d]_{\varphi_{\mathrm{tor}}}&G\ar[r]\ar[d]_{\varphi}&A\ar[r]\ar[d]_{\varphi_{\mathrm{ab}}}&0\\
			0\ar[r]&T'\ar[r]&G'\ar[r]&A'\ar[r]&0
		}
	\end{equation*}
	is a morphism of exact sequences.  Moreover, the induced map
	\[
	\operatorname{Hom}(G,G')
	\longrightarrow
	\operatorname{Hom}(T,T')\times\operatorname{Hom}(A,A')
	\]
	is injective, and its image consists of the pairs
	\((\varphi_{\mathrm{tor}},\varphi_{\mathrm{ab}})\) satisfying
	\[
	(\varphi_{\mathrm{tor}})_*\eta_G
	=
	(\varphi_{\mathrm{ab}})^*\eta_{G'}.
	\]
\end{prop}
\begin{prop}
	The following quotient construction is recalled from \cite[Chapter~2]{Ku2}.
	Let \(G\) be a semiabelian variety over \(K\) with split toric part
	\(T\simeq\mathbb G_m^t\) and abelian quotient \(A\).  Let \(X_\Sigma\) be a
	toric variety with principal torus \(X_0=T\).  Then the categorical quotient
	\[
	\overline G:=G_{X_\Sigma}:=(G\times_K X_\Sigma)/T
	\]
	exists in the category of \(K\)-schemes and is a proper \(K\)-variety.  If
	\((M,\rho)\) is a \(T\)-linearized line bundle on \(X_\Sigma\), then
	\[
	G(M,\rho):=(G\times_K M)/T
	\]
	is a line bundle on \(\overline G\).  Since \(T\)-linearized line bundles on
	\(X_\Sigma\) correspond to toric divisors, we write this quotient line bundle
	as \(\overline G(D)\) when \(M\) corresponds to \(D\).
\end{prop}
\begin{prop}
	If \(X_\Sigma\) is smooth, there is a split exact sequence
	\[
	0\longrightarrow \operatorname{Pic}(A)
	\longrightarrow \operatorname{Pic}(\overline G)
	\longrightarrow \operatorname{Pic}(X_\Sigma)
	\longrightarrow 0.
	\]
\end{prop}
This follows from the seesaw theorem, the vanishing
\(\operatorname{Pic}^0(X_\Sigma)=0\), and the fact that every line bundle on
\(X_\Sigma\) admits a \(T\)-linearization.  The quotient construction
\((G\times_K M)/T\) gives a section of
\(\operatorname{Pic}(\overline G)\to\operatorname{Pic}(X_\Sigma)\).
\subsection{Adelic line bundles}

We recall the adelic language used in \cite{YZ}.  The theory of adelic line
bundles first appeared in Zhang's work on small points and adelic metrics.  In
this framework the N\'eron--Tate height of an algebraic point on an abelian
variety over a number field is interpreted as the normalized arithmetic degree
of an adelic line bundle at that point.  Adelic line bundles are limits of
hermitian line bundles in a suitable topology, and their arithmetic invariants
are obtained as limits of the corresponding invariants on models.

The original theory was formulated for projective varieties over number fields
or one-variable function fields.  Yuan and Zhang extended it to
quasi-projective varieties over finitely generated fields.  More precisely, if
\(F\) is a finitely generated field over \(\mathbb Q\) or over a constant
field, and \(X\) is a quasi-projective variety over \(F\), they construct
adelic line bundles on \(X\), develop their intersection theory, study volumes
of effective sections, and define the associated heights.

Using this theory, Yuan and Zhang proved an equidistribution theorem for small points on quasi-projective varieties over number fields, generalizing the equidistribution theorems of Szpiro--Ullmo--Zhang, Chambert-Loir, and Yuan.  K\"uhne proved an independent semiabelian equidistribution theorem in the canonical metric case and used it, together with the work of Dimitrov--Gao--Habegger, to obtain uniformity results related to the Mordell conjecture.

Let \(X\) be a quasi-projective variety over the number field \(K\).
\begin{defn}
	Let \(U\) be a quasi-projective variety over \(\mathbb Z\).  An object of
	\(\widehat{\mathcal{P}ic}(U/\mathbb Z)\) is a pair
	\[
	\bigl(\mathcal L,(\mathcal X_i,\overline{\mathcal L}_i,\ell_i)_i\bigr)
	\]
	where:
	\begin{enumerate}[label=\textup{(\arabic*)},leftmargin=2em]
		\item \(\mathcal L\) is a line bundle on \(U\);
		\item \(\mathcal X_i\) is a projective model of \(U\) over \(\mathbb Z\);
		\item \(\overline{\mathcal L}_i\) is a hermitian
		\(\mathbb Q\)-line bundle on \(\mathcal X_i\);
		\item \(\ell_i:\mathcal L\to \mathcal L_i|_U\) is an isomorphism in
		\(\mathcal{P}ic(U)_{\mathbb Q}\).
	\end{enumerate}
	The sequence is required to satisfy the Cauchy condition of
	\cite[\S2.5.1]{YZ}.  Adelic line bundles on \(X/\mathbb Z\) are objects of
	\[
	\widehat{\mathcal{P}ic}(X/\mathbb Z)
	:=
	\varinjlim_U \widehat{\mathcal{P}ic}(U/\mathbb Z),
	\]
	where \(U\) ranges over quasi-projective models of \(X\) over \(\mathbb Z\).
\end{defn}
Yuan and Zhang extend arithmetic intersection theory from projective varieties
over number fields to quasi-projective varieties over finitely generated
fields.  In particular, if \(d\) denotes the absolute dimension of a model of
\(X\), they construct an arithmetic intersection map
\[
\widehat{\operatorname{Pic}}(X/\mathbb Z)^d_{\mathrm{int}}
\longrightarrow \mathbb R,
\]
where the subscript \(\mathrm{int}\) denotes integrable adelic line bundles;
see \cite[Definition~2.5.2 and Proposition~4.1.1]{YZ}.

There is also a fully faithful analytification functor
\[
\widehat{\mathcal{P}ic}(X/\mathbb Z)
\longrightarrow
\widehat{\mathcal{P}ic}(X^{r\text{-}an}),
\]
where the target is the category of metrized line bundles on the reified
analytic space, with isometries as morphisms \cite[Proposition~3.5.1]{YZ}.
For projective varieties over \(\mathbb Q\), the essential image is described
by the coherent condition \cite[Proposition~3.5.2]{YZ}.  In the
quasi-projective case we will only use the induced \(v\)-adic metrized line
bundles
\[
\widehat{\mathcal{P}ic}(X/\mathbb Z)
\longrightarrow
\widehat{\mathcal{P}ic}(X^{an}_v).
\]
\begin{defn}
	A \(v\)-metrized line bundle on \(X\) is a pair
	\((L,\Vert\cdot\Vert_v)\), where \(L\) is a line bundle on \(X\) and
	\(\Vert\cdot\Vert_v\) is a metric on \(L^{an}_{\mathbb C_v}\).  Thus, for
	each open subset \(U\subseteq X\), each section
	\(\boldsymbol s:U\to L\), and every \(f\in O_X(U)\), we require:
	\begin{enumerate}[label=\textup{(\arabic*)},leftmargin=2em]
		\item the function
		\[
		\Vert \boldsymbol s\Vert
		=
		\Vert\cdot\Vert_v\circ \boldsymbol s^{an}_{\mathbb C_v}
		:U^{an}_{\mathbb C_v}\to\mathbb R_{\geq0}
		\]
		is continuous;
		\item if \(\boldsymbol s\) vanishes nowhere on \(U\), then
		\(\Vert \boldsymbol s\Vert\) has no zeros on \(U^{an}_{\mathbb C_v}\);
		\item \(\Vert f\boldsymbol s\Vert=|f|\,\Vert\boldsymbol s\Vert\);
		\item the metric is \(\operatorname{Gal}(\mathbb C_v/K_v)\)-invariant.
	\end{enumerate}
\end{defn}
\begin{defn}
	Let \(X\) be a projective variety over \(K\).  If \(v\) is archimedean, a
	\(v\)-metrized line bundle \((L,\Vert\cdot\Vert)\) is semipositive if
	\(dd^c c_1(L,\Vert\cdot\Vert)\geq0\).  If \(v\) is non-archimedean, it is
	semipositive if its metric is the uniform limit of formally semipositive
	\(v\)-metrics.  An adelic line bundle on \(X/\mathbb Z\) is vertically
	semipositive if its \(v\)-adic realization is semipositive for every place
	\(v\) of \(K\).
\end{defn}
\begin{remark}
	Yuan and Zhang also define semipositive \(v\)-metrized line bundles in the
	quasi-projective case; see \cite[p.~105]{YZ}.  We use the same terminology for
	vertical semipositivity in that setting.
\end{remark}

We next recall the arithmetic quantities used to formulate smallness and
equidistribution.
\begin{defn}
	Let \(Z\) be a closed \(\overline K\)-subvariety of \(X\), and let
	\(\overline L\) be an integrable adelic line bundle on \(X/\mathbb Z\).  Let
	\(Z'\) be the image of \(Z\to X\), and let \(\widetilde L\) be the image of
	\(\overline L\) under the canonical map
	\[
	\widehat{\operatorname{Pic}}(X/\mathbb Z)
	\longrightarrow
	\widehat{\operatorname{Pic}}(X/K).
	\]
	The height of \(Z\) with respect to \(\overline L\) is
	\[
	h_{\overline L}(Z)
	:=
	h_{\overline L}(Z')
	:=
	\frac{(\overline L|_{Z'})^{\dim Z+1}}
	{(\dim Z+1)\deg_{\widetilde L}(Z'/K)}.
	\]
	If \(x\in X(\overline K)\) and \(x'\) is the corresponding closed point of
	\(X\), then
	\[
	h_{\overline L}(x):=h_{\overline L}(x')
	:=
	\frac{\overline L|_{x'}}{\deg(x')}.
	\]
	
	The essential minimum is
	\[
	\mu^{\mathrm{ess}}_{\overline L}(X)
	=
	\sup_{U\subseteq X}
	\inf_{x\in U(\overline K)} h_{\overline L}(x),
	\]
	where \(U\) ranges over non-empty Zariski open subschemes of \(X\).  The
	absolute minimum is
	\[
	\mu^{\mathrm{abs}}_{\overline L}(X)
	=
	\inf_{x\in X(\overline K)}h_{\overline L}(x).
	\]
	An integrable adelic line bundle \(\overline L\) is called quasi-canonical if
	\[
	h_{\overline L}(X)=\mu^{\mathrm{ess}}_{\overline L}(X)
	\qquad\text{and}\qquad
	\deg_{\widetilde L}(X/K)>0.
	\]
\end{defn}
Let \(\overline L\) be a vertically semipositive adelic line bundle on
\(X/\mathbb Z\) with \(\deg_{\widetilde L}(X/K)>0\).
\begin{defn}
	For a place \(v\) of \(K\), we say that \(\overline L\) has the
	\(v\)-adic equidistribution property if, for every generic
	\(\overline L\)-small sequence \((x_i)_{i\in\mathbb N}\) in
	\(X(\overline K)\), the probability measures
	\((\delta_{x_i,v})_{i\in\mathbb N}\) associated with the \(v\)-adic Galois
	orbits of the \(x_i\) converge.  In this case the limit measure is unique;
	it is called the canonical measure associated with \(\overline L\), and is
	denoted by \(\mu_{\overline L,v}\).
\end{defn}
This equidistribution property is closely related to Bogomolov-type
statements.  For abelian varieties, equidistribution together with the
Faltings--Zhang morphism gives Zhang's proof of Bogomolov \cite{Zha}.
Similarly, K\"uhne \cite{Ku1} uses the description of the canonical measure at
archimedean places and the Faltings--Zhang morphism to prove Bogomolov for
semiabelian varieties.  In general, equidistribution plus a Bogomolov
non-density statement yields strict equidistribution, and strict
equidistribution for semiabelian varieties implies the Manin--Mumford theorem
in this setting.

We close this section by recalling some standard equidistribution theorems.
\begin{example}
	(Szpiro--Ullmo--Zhang, Chambert-Loir, and Yuan; see \cite{SUZ,CL2,Yua1}.)  If \(X\) is projective
	over a global field \(K\), and if \(\bar L\) is vertically semipositive
	and quasi-canonical with \(\widetilde L\) ample, then \(\bar L\) has the
	\(v\)-adic equidistribution property for all places \(v\) of \(K\).  The
	limiting measure is
	\[
	\frac{1}{\deg_{\widetilde{L}}(X/K)}
	\hat{c}_1(\bar{L}_v)^{\wedge{\dim X}}.
	\]
\end{example}
\begin{example}
	(Yuan--Zhang; \cite[Theorem~5.4.3]{YZ}.)  Let \(\overline L\) be a nef and quasi-canonical adelic line bundle on
	\(X/\mathbb Z\) such that \(\deg_{\widetilde L}(X/K)>0\), where
	\(\widetilde L\) is the image of \(\overline L\) under the canonical map
	\[
	\widehat{\operatorname{Pic}}(X/\mathbb Z)\to
	\widehat{\operatorname{Pic}}(X/K).
	\]
	Then \(\overline L\) has the \(v\)-adic equidistribution property for every
	place \(v\) of \(K\), and the limiting measure is
	\[
	\frac{1}{\deg_{\widetilde{L}}(X/K)}
	\hat{c}_1(\overline{L}_v)^{\wedge{\dim X}} .
	\]
\end{example}
The preceding examples are quasi-canonical.  This condition does not hold in
general: even for projective varieties, varying the metric on \(\overline L\)
can lead to
\[
h_{\overline{L}}(X)<\mu^{\mathrm{ess}}_{\overline{L}}(X).
\]

For toric varieties, Burgos Gil, Philippon, Rivera-Letelier, and Sombra give a
necessary and sufficient condition for a toric metrized
divisor to have the \(v\)-adic equidistribution property on algebraic points of
the principal orbit \(X_0\), for every place \(v\) of \(K\).  This condition is
strictly weaker than quasi-canonicity:
\begin{example}
	By \cite{BGPS19}, let \(X\) be a proper toric variety, and let \(\overline D\) be a
	vertically semipositive toric metrized \(\mathbb R\)-divisor with \(D\)
	big.  Then \(\overline D\) is monocritical if and only if it has the
	\(v\)-adic equidistribution property on algebraic points of \(X_0\), for
	every place \(v\) of \(K\).  Moreover, BGPS describe the limiting measures
	in terms of the critical points of \(\overline D\).
\end{example}
K\"uhne gives semiabelian examples in which \(\overline L\) is not
quasi-canonical but still has the equidistribution property, with canonical
measure formally given by the normalized Monge--Amp\`ere measure:
\begin{example}
	In K\"uhne's construction \cite{Ku1}, choose \((\mathbb P^1)^t\) as a
	compactification of \(T\).  He
	constructs an adelic line bundle
	\[
	\overline L=\overline{M_{\overline G}}\otimes\overline\pi^*\overline N
	\]
	on \(\overline G\), and proves that \(\overline L\) has the \(v\)-adic
	equidistribution property for every place \(v\) of \(K\), with limiting
	measure equal to the normalized Monge--Amp\`ere measure of \(\overline L\).
\end{example}	
\begin{remark}
	The adelic line bundle constructed by K\"uhne is
	\[
	\overline{G}(\overline{D}^{\mathrm{can}})
	\otimes\overline{\pi}^*\overline{N},
	\]
	where \(D=\sum D_i\) is the sum of the divisors corresponding to the rays
	in the fan of \((\mathbb P^1)^t\).
\end{remark}
\section{Local trivialization}
In this section, we extend K\"uhne's local trivialization at archimedean places
to all places of \(K\), using \(p\)-adic theta functions and the
Raynaud--Bosch--L\"utkebohmert uniformization of abelian varieties over
complete algebraically closed non-archimedean fields.  We then use the local
trivialization to define the adelic line bundle associated with a toric
metrized \(\mathbb R\)-divisor and to compare the essential minima of this
adelic line bundle on \(G\) with the absolute minimum of the corresponding
toric metrized divisor on \(T\).
\begin{thm}[Non-archimedean local trivialization]\label{thm:local-trivialization}
	Let \(v\) be a non-archimedean place of \(K\).  There is a finite collection
	\(\{(U_j,\psi_j)\}_{j\in J}\) such that the \(U_j\) form an open covering of
	\(A^{an}_{\mathbb C_v}\), and
	\[
	\psi_j:
	\overline{G}^{an}_{\mathbb C_v}|_{U_j}
	\longrightarrow
	X_{\Sigma,\mathbb C_v}^{an}
	\]
	are analytic maps for which
	\[
	\overline{\pi}^{an}_{\mathbb C_v}\times\psi_j:
	\overline{G}^{an}_{\mathbb C_v}|_{U_j}
	\longrightarrow
	U_j\times X_{\Sigma,\mathbb C_v}^{an}
	\]
	is an isomorphism of analytic spaces.  Moreover, for every toric
	\(\mathbb R\)-Cartier divisor \(D\),
	\[
	\overline{G}(D)^{\mathrm{can}}_v|_{U_j}
	=
	\psi_j^*(\overline{D}^{\mathrm{can}}_v).
	\]
	For all \(j,j'\in J\) and all
	\[
	y\in \overline{G}^{an}_{\mathbb C_v}|_{U_j\cap U_{j'}},
	\]
	one has
	\[
	\operatorname{val}_v\circ\psi_j(y)
	=
	\operatorname{val}_v\circ\psi_{j'}(y).
	\]
\end{thm}
We now recall the analytic inputs used to construct the maps \(\psi_j\).

We first recall \(p\)-adic theta functions and Raynaud--Bosch--L\"utkebohmert
uniformization for abelian varieties.  The following form of the
uniformization theorem can be found in \cite{FRSS}.  In K\"uhne's notation,
the two uniformization statements below correspond to
\cite[Section~3.2]{Ku1}.

Let \(\mathbf K\) be a complete algebraically closed non-archimedean valued
field with valuation ring \(R\), and let \(A\) be an abelian variety over
\(\mathbf K\).

\begin{thm}[Raynaud--Bosch--L\"utkebohmert uniformization I]
	There exists a topological universal covering \(p:E^{an}\to A^{an}\).  If
	we choose a base point \(0\in E^{an}\) over the identity, then \(E^{an}\)
	has a unique structure of analytic group with identity \(0\), and \(p\) is
	a homomorphism.  Moreover, \(E^{an}\) is the analytification of a group
	scheme \(E\).
\end{thm}
\begin{thm}[Raynaud--Bosch--L\"utkebohmert uniformization II]
	There are exact sequences
	\[
	0\longrightarrow\mathbf T\longrightarrow E\xrightarrow{q}\mathbf B
	\longrightarrow0
	\]
	and
	\[
	0\longrightarrow M'\longrightarrow E^{an}\xrightarrow{p}A^{an}
	\longrightarrow0,
	\]
	where \(\mathbf T=\operatorname{Spec}(\mathbf K[M])\) is a split
	\(\mathbf K\)-torus, \(\mathbf B\) is an abelian variety with good reduction,
	and \(M'\) is a lattice in \(E^{an}\).
\end{thm}
\addtocounter{thm}{12}

\paragraph{Theta functions and descent data.}
We use the Raynaud theta formalism in the form developed in
\cite[Sections~3.3--3.6]{Ku1}.  A character \(u\in M\) pushes the Raynaud
extension out to a rigidified Picard-zero line bundle \(E_u\) on \(B\), with
canonical section \(e_u\).  The resulting tropicalization is characterized by
\[
\langle u,\operatorname{trop}(x)\rangle
=-\log\|e_u(x)\|_{E_u}.
\]
The period lattice \(M'\) maps to a full lattice in \(N_{\mathbb R}\), and the
fiber over \(0\) is the generic fiber of the formal Raynaud extension.

Rigidified line bundles on \(A\) are described by Raynaud descent triples
\((L,\lambda,c)\); see \cite[Theorem~3.6]{Ku1}.  On a formal affinoid chart
where the Raynaud extension splits, a section of \(q^*L\) has a unique Fourier
expansion
\[
f=\sum_{u\in M}a_u\otimes e_u.
\]
The theta transformation law of \cite[Proposition~3.13]{Ku1} is exactly the
condition that this section descend through \(M'\) to \(A\).  For a rational
section \(s\) of a rigidified algebraically trivial line bundle, the associated
theta function \(f_s\) satisfies
\[
\operatorname{div}(f_s)
=p^{-1}\bigl(\operatorname{div}(s)^{an}\bigr).
\]
These are the only theta-function facts used in the local-trivialization
argument.  We record them, together with the toric extension criterion, in the
next proposition.

We next recall the skeleton of \(A^{an}_{\mathbf{K}}\), constructed on
pp.~124--125 of \cite{Ber}.
\begin{lem}
	There is a skeleton \(\Delta(A)\) homeomorphic to a real torus, and it is a
	strong deformation retract of \(A^{an}\).  Moreover, \(A^{an}\) admits a
	finite cover by contractible analytic open subsets.
\end{lem}
\begin{remark}
	For any divisor \(D\) of \(A\), the lemma allows us to choose an open
	contractible subset of \(A^{an}\) disjoint from the analytic divisor
	\(D^{an}\).
\end{remark}
We use these facts and the theta functions above to construct the local
trivializations.

\begin{proposition}[Analytic inputs for the local trivialization]\label{input:local-triv-analytic}
	The proof of Theorem~\ref{thm:local-trivialization} uses the following standard
	facts in the precise forms stated here.
	\begin{enumerate}[leftmargin=2em]
		\item Raynaud--Bosch--Lutkebohmert uniformization gives
		\[
		0\to M'_v\to E_v^{an}\xrightarrow{p_v}A_{\mathbb C_v}^{an}\to0
		\]
		and the associated Raynaud extension.  Over a sufficiently small contractible
		analytic open \(U\subset A_{\mathbb C_v}^{an}\), the covering map \(p_v\) admits
		an analytic section.
		\item For a rational section \(s\) of a rigidified algebraically trivial line
		bundle on \(A\), Raynaud theta descent produces a function \(f_{s,v}\) satisfying
		\[
		\operatorname{div}(f_{s,v})=p_v^{-1}(\operatorname{div}(s)^{an}).
		\]
		\item A morphism from the principal torus to a normal toric variety extends over
		the affine chart \(X_\sigma\) precisely when all characters
		\(\chi^m\), \(m\in\sigma^\vee\cap M\), pull back to analytic functions with
		effective boundary divisor.  This is the standard affine-chart criterion for
		toric morphisms, applied in the analytic category.
	\end{enumerate}
\end{proposition}

\begin{proof}[Proof of Theorem~\ref{thm:local-trivialization}]
	Let \(\Sigma(1)=\{\rho_1,\ldots,\rho_d\}\), let \(v_\ell\) be the primitive
	generator of \(\rho_\ell\), and let \(D_\ell\) be the corresponding
	\(T\)-invariant prime divisor on \(X_\Sigma\).  Write the extension class of
	\(G\) in the chosen character basis as
	\[
	\eta_G=(H_1,\ldots,H_t)\in A^\vee(\overline K)^t .
	\]
	The cocycle description of the associated toric fibration gives
	\[
	\overline{\pi}^*H_i
	=
	\sum_{\ell=1}^{d}\langle e_i^\vee,v_\ell\rangle\,\overline G(D_\ell),
	\qquad 1\le i\le t.
	\]
	
	Choose an analytic open subset \(U\subset A_{\mathbb C_v}^{an}\) which is
	contractible and disjoint from the analytic divisors of rational sections
	\(s_i\) of \(H_i^\vee\).  By Proposition~\ref{input:local-triv-analytic}, the
	covering \(p_v:E_v^{an}\to A_{\mathbb C_v}^{an}\) has an analytic section over
	\(U\); denote the corresponding lift of \(\pi\) by
	\[
	\widetilde\pi_U:\overline G^{an}_{\mathbb C_v}|_U\longrightarrow E_v^{an}.
	\]
	The section \(s_i\) determines a theta function \(f_{s_i,v}\) with
	\[
	\operatorname{div}(f_{s_i,v})
	=
	p_v^{-1}\bigl(\operatorname{div}(s_i)^{an}_{\mathbb C_v}\bigr).
	\]
	The cocycle formula gives a rational function \(g_i\) on \(\overline G\) with
	\[
	\operatorname{div}(g_i)
	=
	\sum_{\ell=1}^{d}\langle e_i^\vee,v_\ell\rangle\,\overline G(D_\ell)
	+\overline\pi^*\operatorname{div}(s_i).
	\]
	After multiplying \(s_i\) by constants, we may assume
	\(g_i(e_G)=1\) and \(\Vert f_{s_i,v}(e_E^{an})\Vert=1\).
	
	On the open semiabelian part over \(U\), define
	\[
	\psi_U^{(i)}(y)
	=
	\frac{g_i(y)}{f_{s_i,v}(\widetilde\pi_U(y))},
	\qquad
	\psi_U=(\psi_U^{(1)},\ldots,\psi_U^{(t)}):G^{an}_{\mathbb C_v}|_U\to T^{an}_{\mathbb C_v}.
	\]
	The denominator has no zero over \(U\), and hence \(\psi_U^{(i)}\) is an
	invertible analytic function on \(G^{an}_{\mathbb C_v}|_U\).
	For each \(a\in U\), the divisor of \(\psi_U^{(i)}\) on the compactified
	fiber \(\overline G_a\simeq X_{\Sigma,\mathbb C_v}\) is exactly the divisor of
	the torus character \(\chi^{e_i^\vee}\).  Thus \(\psi_U\) identifies the
	principal torus in the fiber with \(T_{\mathbb C_v}^{an}\), up to multiplication
	by nonzero constants.
	
	We now extend over the toric boundary.  Let
	\(X_\sigma=\operatorname{Spec}\mathbb C_v[\sigma^\vee\cap M]\) be an affine
	toric chart.  For \(m=\sum_i m_ie_i^\vee\in\sigma^\vee\cap M\), the pullback
	of \(\chi^m\) is
	\[
	\prod_{i=1}^{t}\bigl(\psi_U^{(i)}\bigr)^{m_i}.
	\]
	Its boundary divisor is
	\[
	\sum_{\ell=1}^{d}\langle m,v_\ell\rangle\,\overline G(D_\ell),
	\]
	which is effective on the chart because \(m\in\sigma^\vee\cap M\).  By the
	affine-chart criterion in Proposition~\ref{input:local-triv-analytic}, the map
	\((\overline\pi,\psi_U)\) extends uniquely to an analytic morphism
	\[
	\Psi_U:
	\overline G^{an}_{\mathbb C_v}|_U
	\longrightarrow
	U\times X_{\Sigma,\mathbb C_v}^{an}.
	\]
	The same character calculation applied to the inverse torus coordinates gives
	an inverse morphism.  Hence \(\Psi_U\) is an isomorphism of analytic spaces.
	
	The finite cover is obtained by choosing finitely many such contractible opens
	and, on each of them, rational sections \(s_i\) whose divisors avoid the open.
	Compactness of \(A_{\mathbb C_v}^{an}\) gives a finite subcover
	\(\{U_j\}_{j\in J}\), and we write the second component of \(\Psi_{U_j}\) as
	\(\psi_j\).
	
	It remains to verify overlap compatibility.  On \(U_j\cap U_{j'}\), write
	\[
	s_j^{(i)}=t_{jj'}^{(i)}s_{j'}^{(i)}
	\]
	with \(t_{jj'}^{(i)}\) an analytic unit.  The corresponding rational functions
	and theta functions satisfy
	\[
	g_j^{(i)}=(t_{jj'}^{(i)}\circ\pi)g_{j'}^{(i)},\qquad
	f_{s_j^{(i)},v}=(t_{jj'}^{(i)}\circ p_v)f_{s_{j'}^{(i)},v}.
	\]
	If the chosen local lifts of \(U_j\) and \(U_{j'}\) differ on a connected
	component of the overlap by an element of the period lattice \(M'_v\), the
	theta transformation law adds the usual Raynaud automorphy factor.  The same
	factor occurs in the descent cocycle defining the semiabelian extension
	class \(\eta_G\); hence it cancels in the quotient
	\(g_i/f_{s_i,v}\).  After this cancellation the displayed formulas give the
	actual change of local coordinates.
	Since \(p_v(\widetilde\pi_U(y))=\pi(y)\), the two unit factors have the same
	absolute value after evaluation at \(y\).  Therefore
	\[
	\left|
	\frac{\psi_j^{(i)}(y)}{\psi_{j'}^{(i)}(y)}
	\right|_v=1,\qquad 1\le i\le t.
	\]
	Equivalently, the transition multiplier lies in the compact torus, so it has
	valuation zero.  The toric tropicalization map is invariant under compact
	torus multiplication.  Hence
	\[
	\val_v\circ\psi_j(y)=\val_v\circ\psi_{j'}(y)
	\]
	on overlaps.
	
	Finally, under \(\Psi_U\), the line bundle \(\overline G(D)\) is identified
	with the pullback of the toric line bundle \(O(D)\) on the second factor.  When
	\(D\) is equipped with the canonical toric metric, that metric is a function
	of the toric valuation and is invariant under compact torus multiplication.
	The overlap compatibility therefore identifies the canonical metrics:
	\[
	\overline G(D)^{can}_v|_{U_j}=\psi_j^*(\overline D^{can}_v).
	\]
\end{proof}
\begin{lemma}[Good-place model for the induced canonical metric]\label{input:good-place-model-induced}
	For all but finitely many non-archimedean places \(v\), the semiabelian
	extension \(G\), the abelian quotient \(A\), the toric variety \(X_\Sigma\), the
	fan \(\Sigma\), the toric divisor \(D\), and the local data used in
	Theorem~\ref{thm:local-trivialization} extend over \(O_{K_v}\).  At such a
	place the local trivializations are induced by formal trivializations of the
	corresponding toric fibration, their transition functions lie in the maximal
	compact torus, and the canonical toric metric on \(O(D)\) is the model metric
	coming from the toric model.
\end{lemma}

\begin{prop}[Adelicity of the induced metric]\label{prop:induced-metric-adelic}
	Let \(\bD\) be a toric metrized \(\BR\)-divisor on \(X_\Sigma\).  On
	\(\overline\pi^{-1}(U_j)\), define the induced local metric on
	\(\overline G(D)\) by
	\[
	\|\widetilde s(y)\|_{\overline G(\bD),v}
	=
	\|s(\psi_{j,v}(y))\|_{\bD,v},
	\]
	where \(s\) is a local section of \(O(D)\) and \(\widetilde s\) is the
	corresponding local section of \(\overline G(D)\).  These local formulas glue
	to a continuous \(v\)-adic metric on \(\overline G(D)\).  If \(\bD\) is adelic,
	then the induced metric is adelic.
\end{prop}
\begin{proof}
	The only point to check for gluing is independence of \(j\).  On an overlap,
	Theorem~\ref{thm:local-trivialization} shows that the two toric coordinates
	differ by multiplication by an element of the compact torus.  Toric metrics
	are invariant under this compact torus.  Hence the displayed local formula is
	independent of the chart.  Continuity is local and follows from continuity of
	the toric metric and of the analytic maps \(\psi_{j,v}\).
	
	We next prove adelicity.  First suppose that \(D\) is an integral toric
	divisor.  Write
	\[
	\bD=\bD^{can}+\varphi,\qquad \varphi=(\varphi_v)_v,
	\]
	where each \(\varphi_v\) is a compact-torus-invariant continuous function on
	\(X_{\Sigma,\mathbb C_v}^{an}\), and \(\varphi_v=0\) for all but finitely many
	places.  The induced metric is obtained from the induced canonical metric by
	the finite family of functions
	\[
	\varphi_v\circ\psi_{j,v}.
	\]
	Thus it remains to prove that the induced canonical metric is adelic.
	
	By Lemma~\ref{input:good-place-model-induced}, outside a finite set of
	non-archimedean places the local trivializations are formal trivializations of
	a toric fibration over \(O_{K_v}\), and the canonical toric metric is the model
	metric on the toric model.  Pulling back that model metric through the formal
	trivializations gives the model metric on \(\overline G(D)\).  Therefore the
	induced canonical metric is a model metric at almost all places, and hence is
	adelic.  Adding finitely many continuous changes preserves adelicity.  The
	case of toric metrized \(\BR\)-divisors follows by linearity and uniform
	approximation by rational linear combinations of integral toric divisors.
\end{proof}

\begin{remark}
	Lemma~\ref{input:good-place-model-induced} is the precise good-place fact used
	here.  No separate conjectural assertion about canonical metrics at every
	good-reduction place is needed for the adelicity proof.
\end{remark}

\begin{prop}[Semipositivity of the induced metric]\label{prop:induced-metric-semipositive}
	If \(\overline{D}\) is vertically semipositive then the adelic line bundle
	\(\overline{G}(\overline{D})\) associated with \(\overline{D}\) is vertically
	semipositive and \(\overline{G}(D)\) is nef.
\end{prop}  
\begin{proof}
	First consider the underlying line bundle.  For a toric metrized
	\(\BR\)-divisor, vertical semipositivity implies that the associated local
	metric functions are concave; in particular the virtual support function of
	\(D\) is concave, and hence \(D\) is nef on \(X_\Sigma\).
	
	Let \(N\) be an ample symmetric line bundle on \(A\).  If \(D\) is ample, then
	by K\"uhne's ampleness criterion \cite[Lemma 3, Chapter 2]{Ku1},
	\(\overline G(D)\otimes\overline\pi^*N\) is ample, and therefore
	\(\overline G(D)\) is nef.  If \(D\) is only nef, choose an ample toric divisor
	\(D'\).  Since the nef cone is the closure of the ample cone,
	\(D+\epsilon D'\) is ample for every \(\epsilon>0\).  Hence
	\[
	\overline G(D)+\epsilon\bigl(\overline G(D')\otimes\overline\pi^*N\bigr)
	\]
	is ample for every \(\epsilon>0\), and letting \(\epsilon\to0\) shows that
	\(\overline G(D)\) is nef.
	
	For the metric, on each chart \(U_j\) the induced metric is the pullback of
	the semipositive toric metric \(\bD_v\).  Pullback by an analytic morphism
	preserves semipositivity, both at archimedean and non-archimedean places.
	Since the chartwise metrics glue by Proposition~\ref{prop:induced-metric-adelic},
	\(\overline G(\bD)\) is vertically semipositive.
\end{proof}
\begin{cor}
	The line bundle
	\[
	L=\overline{G}(D)\otimes \overline{\pi}^{*}N
	\]
	is big if \(D\) is nef and big and \(N\) is an ample symmetric line bundle on
	\(A\).
\end{cor}
\begin{proof}
	By Proposition~\ref{prop:induced-metric-semipositive}, \(\overline{G}(D)\) is nef.
	Since the pullback of the ample symmetric line bundle \(N\) is nef, \(L\)
	is nef.  It therefore suffices to show \(L^{g+t}>0\).
	Let \(\eta\) be the generic point of \(A\).  Then
	\[
	L^{g+t}
	=
	\overline{G}(D)^t(\overline{\pi}^*N)^g
	=
	N^g\bigl(\overline{G}(D)|_{\eta}\bigr)
	=
	N^gD^t>0.
	\]
\end{proof}
\begin{defn}[Arithmetic \(T\)-effectivity]\label{def:arith-T-effective}
	A toric metrized \(\BR\)-divisor \(\bD\) on \(X_\Sigma\) is called
	arithmetically \(T\)-effective if there exist a finitely supported family of
	real numbers \((c_v)_v\in\bigoplus_v\mathbb R\) such that the
	distinguished toric section \(s_0=s_D\) of \(O(D)\), corresponding to
	\(0\in\Delta_D\), satisfies
	\[
	\sum_vn_vc_v=\muabs_{\bD}(T)
	\]
	and
	\[
	\|s_0\|_{v,\sup}\le e^{-c_v}
	\qquad\hbox{for every place }v.
	\]
	We call \(s_0\) a \(\muabs_{\bD}(T)\)-small distinguished section.
	Here, as in the preceding definition of \(T\)-effectivity, the term
	toric invariant section always means this distinguished section
	corresponding to \(0\in\Delta_D\), rather than an arbitrary character
	eigensection.
\end{defn}

\begin{example}
	If \(D\) is \(T\)-effective and \(\bD\) is the canonical toric metrized
	\(\BR\)-divisor, then \(\bD\) is arithmetically \(T\)-effective.  Indeed
	\(\muabs_{\bD}(T)=0\), and the canonical toric section \(s_D\) satisfies
	\(\|s_D\|_{v,\sup}\le1\) at every place.
\end{example}

\begin{remark}
	The distinguished section \(s_0\) is nowhere vanishing on the open
	torus \(T\).  This is the point of the definition that will be used below:
	after passing to the semiabelian compactification, this section gives a
	global lower bound for the height on \(G\).
\end{remark}

\subsection{Meaning and practical criteria for arithmetic \(T\)-effectivity}
\label{subsec:arith-teff-meaning}

The preceding condition is an attainment condition, not merely an auxiliary
positivity assumption.  Let \(\vartheta_{\bD,v}\) and
\(\vartheta_{\bD}=\sum_vn_v\vartheta_{\bD,v}\) be the local and global roof
functions.  The toric sup-norm identity and the successive-minima formula give
\[
-\log\|s_0\|_{v,\sup}=\vartheta_{\bD,v}(0),\qquad
\muabs_{\bD}(T)=\max_{x\in\Delta_D}\vartheta_{\bD}(x)
\]
\cite[Chapters 4--5]{BPS14}\cite[Theorem 3.9 and Corollary 3.10]{BGPS15}.
Consequently the following are equivalent:
\[
\bD\text{ is arithmetically \(T\)-effective}
\quad\Longleftrightarrow\quad
\vartheta_{\bD}(0)=\muabs_{\bD}(T)
\quad\Longleftrightarrow\quad
0\in\operatorname*{argmax}_{\Delta_D}\vartheta_{\bD}.
\]
Indeed the local inequalities imply
\(c_v\le\vartheta_{\bD,v}(0)\); after summing, both sides are trapped between
\(\muabs_{\bD}(T)\) and the maximum of the roof.  Conversely one takes
\(c_v=\vartheta_{\bD,v}(0)\).

This characterization gives several quick tests.  It is enough that
\(0\in\partial^+\vartheta_{\bD}(0)\); at an interior differentiability point
this reduces to \(\nabla\vartheta_{\bD}(0)=0\).  It also holds when
\(\Delta_D\) is centrally symmetric and the roof is even, for canonical
metrics, and for the quasi-canonical normal form used later in this paper.  In
the rational piecewise-affine case it is a finite linear check on the affine
cells of the roof.

The condition is not automatic.  On \(\mathbb P^1\), take
\(\Delta_D=[0,1]\) and a semipositive toric metric with global roof
\[
\vartheta_{\bD}(x)=C-(x-\alpha)^2
\]
for an irrational \(\alpha\in(0,1)\).  The unique maximizer is not a lattice
point; in particular it is not \(0\), so the distinguished section cannot
realize the absolute minimum.  This example exhibits the arithmetic
integrality obstruction that is invisible to semipositivity and bigness.

More generally, if arbitrary toric eigensections are allowed, the corresponding
coordinate-free condition is
\[
\operatorname*{argmax}_{\Delta_D}\vartheta_{\bD}\cap M\ne\varnothing.
\]
Such a witness fixes \(m\), not automatically the origin.  The definition
above deliberately uses \(s_0\), because the monocritical argument in
Section~9 requires the origin to lie on the maximizing face.

This also connects the condition with the general Dirichlet problem in
Arakelov geometry.  Chen--Moriwaki ask whether an adelic
\(\mathbb R\)-Cartier divisor is \(\mathbb R\)-linearly equivalent to an
effective one and study both criteria and obstructions
\cite{CMDirichlet,CMDynamics}; their adelic-curve framework relates absolute
minima to asymptotic minimal slopes \cite[Chapter 7]{CMAdelicCurves}.
Arithmetic \(T\)-effectivity is a torus-equivariant, threshold-sharp version
of this principle: the effective witness is restricted to the distinguished
toric eigenline, remains nonvanishing on the open group, and realizes exactly
the absolute minimum.  Thus it is stronger than the ordinary Dirichlet
property, but it is precisely adapted to the global height separation used
here.

What is necessary is correspondingly limited.  Arithmetic \(T\)-effectivity
is necessary for the present proof mechanism if one insists on deriving the
optimal lower bound from the single distinguished section.  It is not proved
to be necessary for the equidistribution or Bogomolov conclusion itself.

We next recall the theta-function calculation which identifies the local
metric contribution of a rigidified Picard-zero line bundle with the
corresponding N\'eron local function.  The non-archimedean argument uses the
tropical theta formalism to play the role of the hermitian form in the
archimedean case.

As in the toric case, there is a canonical continuous section of
\[
\operatorname{trop}:E^{an}\to N_{\mathbb R}.
\]
Choose a local section \(V\to E^{an}\) such that the formal Raynaud extension
splits over \(V\), so \(q^{-1}(V)\cong T\times V\).  Then
\(\operatorname{trop}^{-1}(v)=U_v\times V\), where \(U_v\) is the inverse image
of \(v\) under \(\operatorname{trop}:T^{an}\to N_{\mathbb R}\).  Every analytic
function \(h\) on \(U_v\times V\) has a unique Laurent expansion
\[
h=\sum_{u\in M}a_u\chi^u
\]
with \(a_u\in\mathscr O(V)\) and
\[
|a_u|_{\sup}e^{-\langle u,v\rangle}\to0
\]
outside finite subsets of \(M\).
\begin{defn}
	Define \(\sigma(v)\) to be the multiplicative seminorm on
	\(\mathscr O(U_v\times V)\) given by
	\[
	|h(\sigma(v))|
	=
	\max_{u\in M}\{|a_u|_{\sup}e^{-\langle u,v\rangle}\}.
	\]
	This seminorm is independent of the choice of \(V\) and of the splitting, and
	it depends continuously on \(v\); see \cite[Chapter~4.1]{TJFA}.
\end{defn}
Fix a rigidified line bundle \(L_A\) on \(A\), and let \((L,\lambda,c)\) be a
corresponding triple for \(L_A\).  Let
\(f\in H^0(E^{an},q^*L)\) be a nonzero theta function.
\begin{defn}
	The tropicalization of \(f\) is the function
	\[
	f_{\operatorname{trop}}:N_{\mathbb R}\to\mathbb R
	\]
	defined by
	\[
	f_{\operatorname{trop}}(v)
	=
	-\log\Vert f(\sigma(v))\Vert_{q^*L}.
	\]
\end{defn}
\begin{remark}
	This definition extends to the theta function \(f_D\) associated with any
	Cartier divisor \(D\): if \(D=D_1-D_2\) with \(D_1,D_2\) ample, then
	\(f_{D,\operatorname{trop}}\) is defined as the corresponding difference of
	tropicalizations.
\end{remark}
\begin{prop}
	Assume that \(L_A\) is ample, and put
	\[
	c_{\operatorname{trop}}(u')
	=
	-\log\Vert c(u')\Vert_L .
	\]
	Then, for all \(u'_1,u'_2\in M'\),
	\[
	c_{\operatorname{trop}}(u'_1+u'_2)
	-c_{\operatorname{trop}}(u'_1)
	-c_{\operatorname{trop}}(u'_2)
	=
	[u'_1,\lambda(u'_2)].
	\]
	Moreover, for \(u'\in M'\) and \(v\in N_{\mathbb R}\),
	\[
	f_{\operatorname{trop}}(v)
	=
	f_{\operatorname{trop}}(v+u')
	+c_{\operatorname{trop}}(u')
	+\langle\lambda(u'),v\rangle .
	\]
\end{prop}
\begin{prop}[Canonical Picard-zero theta metrics]
	\label{prop:canonical-picard-zero-theta-metrics}
	\label{input:theta-neron-local-heights}
	Let \(A\) be an abelian variety over a number field \(K\), let
	\(H\in\operatorname{Pic}^0(A)\) be a rigidified line bundle, and let \(s\) be a
	rational section of \(H\).  Equip \(H\) with its canonical adelic Picard-zero
	metric, characterized by the rigidification at the origin and by the requirement
	that the canonical square isomorphism, equivalently the isomorphisms
	\([n]^*H\simeq H^{\otimes n}\), be isometries.
	
	For every place \(v\), let \(f_{s,v}\) be the theta representative of \(s\) on
	the Raynaud/K\"uhne analytic uniformization, normalized by the rigidification of
	\(H\).  Then \(f_{s,v}\) computes the canonical local norm of \(s\):
	\[
	\Vert s(x)\Vert_{H,v}^{can}
	=
	\Vert f_{s,v}(\widetilde x)\Vert_v
	\]
	for every \(x\notin|\operatorname{div}(s)|\) and every lift \(\widetilde x\) of
	\(x\).  Hence
	\[
	\lambda_{s,v}(x):=-\log\Vert f_{s,v}(\widetilde x)\Vert_v
	\]
	is the normalized Neron local function attached to the canonically metrized
	pair \((H,s)\).  It is independent of the choice of lift \(\widetilde x\), and
	\[
	\sum_v n_v\lambda_{s,v}(x)=h_{\overline H}(x).
	\]
	The normalization is compatible with principal divisors: if
	\(\operatorname{div}(s)=\operatorname{div}(g)\), then
	\(\lambda_{s,v}(x)=-\log|g(x)|_v\) up to the usual adelic constant, whose global
	sum is zero by the product formula.
\end{prop}
\begin{proof}
	We spell out the normalization because this is the point at which the
	Raynaud--K\"uhne theta expression has to be identified with the canonical
	Picard-zero metric.  Choose a Cartier divisor \(D=\operatorname{div}(s)\)
	representing \(H\).  On the Raynaud analytic uniformization at a
	non-archimedean place \(v\), let \(f_{D,v}\) be the theta representative of
	\(D\), normalized by the rigidification at the origin.  Following the tropical
	theta formalism recalled above, set
	\[
	W_{D,v}(x)
	=
	-\log\Vert f_{D,v}(\widetilde x)\Vert_v
	-f_{D,v}^{\operatorname{trop}}
	\bigl(\operatorname{trop}(\widetilde x)\bigr)
	+f_{D,v}^{\operatorname{trop}}(0).
	\]
	The transformation law for theta functions attached to a triple
	\((L,\lambda,c)\), together with the corresponding transformation law for
	tropical theta functions, shows that \(W_{D,v}(x)\) is independent of the lift
	\(\widetilde x\).  Explicitly, replacing \(\widetilde x\) by
	\(\widetilde x+u'\), \(u'\in M'\), changes the first term by the automorphy
	factor \(c(u')e_{\lambda(u')}(\widetilde x)\), while the tropical term changes
	by the logarithm of the same factor.  Hence the two changes cancel.  This is
	precisely the compatibility recorded in the theta transformation formulas of
	\cite[Section~3]{FRSS} and in the tropical theta calculation of
	\cite[Section~4]{TJFA}.
	
	The same construction is additive in \(D\).  If \(D=\operatorname{div}(g)\) is
	principal, then the associated theta representative is the pullback of \(g\),
	up to the constant fixed by the rigidification, and therefore
	\[
	W_{D,v}(x)=-\log |g(x)|_v
	\]
	up to a \(v\)-constant.  Functoriality of the Raynaud extension and of the
	theta descent data gives
	\[
	W_{[2]^*D,v}(x)=W_{D,v}(2x).
	\]
	Thus \(W_{D,v}\) satisfies the standard characteristic properties of the
	normalized Neron local function attached to \(D\).  At archimedean places the
	same statement is the classical theta-function comparison with Neron
	functions, as in \cite[Chapter~13, Theorem~1.1]{Lang}; the hermitian
	quadratic correction is the archimedean analogue of the tropical correction.
	
	Now use that \(H\in\operatorname{Pic}^0(A)\).  In the Raynaud triple this means
	that the homomorphism \(\lambda\) is zero, or equivalently that the Chern form
	of the canonical metric vanishes.  Consequently the tropical correction
	\(f_{D,v}^{\operatorname{trop}}\) is constant after the above rigidified
	normalization; at archimedean places the hermitian form attached to \(D\) is
	zero.  Hence the normalized local Neron function is simply
	\[
	\lambda_{s,v}(x)
	=-\log\Vert f_{s,v}(\widetilde x)\Vert_v .
	\]
	By uniqueness of the canonical rigidified Picard-zero metric, this is
	\(-\log\Vert s(x)\Vert^{can}_{H,v}\).  Summing over all places gives the
	height of the canonically metrized line bundle \(\overline H\), and changing
	the rational section by a rational function only changes the local functions
	by \(-\log|g(x)|_v\), whose weighted adelic sum is zero by the product formula.
\end{proof}

\begin{remark}[Source status of the theta--Neron comparison]
	\label{rem:theta-neron-source-status}
	The normalization in
	Proposition~\ref{prop:canonical-picard-zero-theta-metrics} is fixed in two
	independent layers.  First, Zhang's canonical adelic metric construction
	\cite[Theorem~2.2]{ZhangAdelicMetrics}, in the Picard-zero form recalled by
	Chambert-Loir \cite[\S1 and final appendix]{CLHeight}, gives the unique
	canonical adelic metric on a rigidified algebraically trivial line bundle for
	which the square isomorphism and all multiplication isomorphisms are isometries.
	At complex places this is the unique flat hermitian metric compatible with the
	rigidification; at non-archimedean places it is the \(v\)-adic canonical metric
	obtained by Tate's limiting construction.
	
	Second, the theta functions used here are local expressions of that same
	canonical metric.  At non-archimedean places, Raynaud--Bosch--Lutkebohmert
	uniformization and theta functions are recalled in \cite[Section~3]{FRSS}; the
	line-bundle/triple correspondence is \cite[Theorem~3.10]{FRSS}, the descent
	criterion for theta functions is \cite[Proposition~3.19 and Definition~3.20]{FRSS},
	and the Fourier/model-metric compatibility is recorded around
	\cite[Remark~3.18 and Proposition~3.19]{FRSS}.  At archimedean places this is
	the classical theta/Neron-function comparison, e.g.
	\cite[Chapter~13, Theorem~1.1]{Lang}.  With this convention, the function
	\(-\log\|s(x)\|^{can}_{H,v}\) is the normalized Neron local function attached
	to \((H,s)\), and its adelic sum is by definition the height of the
	canonically metrized Picard-zero bundle \(\overline H\).
\end{remark}

\begin{prop}[Theta functions and Neron heights]\label{prop:theta-neron-height}
	Let \(A\) be an abelian variety over a number field \(K\), let
	\(H\in\operatorname{Pic}^0(A)\) be rigidified, and let \(s\) be a rational
	section of \(H\).  If \(x\in A(\overline K)\) does not lie in
	\(\lvert\operatorname{div}(s)\rvert\), then
	\[
	h_{\overline H}(x)
	=
	\sum_v -n_v\log
	\bigl\Vert f_{s,v}(\widetilde{\pi_v}(x))\bigr\Vert_v .
	\]
	Here \(f_{s,v}\) is the normalized theta representative from
	Proposition~\ref{prop:canonical-picard-zero-theta-metrics}, and
	\(\widetilde{\pi_v}(x)\) is any lift of \(x\) to the analytic universal
	cover.  The summand is independent of that lift.
\end{prop}
\begin{proof}
	This is the global form of
	Proposition~\ref{prop:canonical-picard-zero-theta-metrics}, corresponding to
	K\"uhne's theta-height comparison \cite[Theorem~4.7]{Ku1}.  Indeed the local functions
	\[
	-\log\Vert f_{s,v}(\widetilde{\pi_v}(x))\Vert
	\]
	are the normalized Neron local functions associated with \((H,s)\).  Their
	adelic sum is therefore the canonical height attached to the canonically
	metrized Picard-zero bundle \(\overline H\).  The product formula removes
	the ambiguity coming from replacing \(s\) by a rational multiple.
\end{proof}
\begin{corollary}[Picard-zero heights and the Neron pairing]
	\label{cor:picard-zero-height-pairing}
	Let \(N\) be a symmetric ample line bundle on \(A\), and let
	\(\phi_N:A\to A^\vee\) be the associated polarization, with the convention
	\[
	\phi_N(q)=T_q^*N\otimes N^{-1}.
	\]
	Let \(\langle\cdot,\cdot\rangle_N\) be the Neron--Tate bilinear form normalized
	by
	\[
	h_{\overline N}(y+z)
	=h_{\overline N}(y)+h_{\overline N}(z)+2\langle y,z\rangle_N.
	\]
	Suppose
	\(H\in\operatorname{Pic}^0(A)_\BR\), \(a>0\), and
	\(q\in A(\overline K)\otimes_\BZ\BR\) satisfy \(\phi_N(q)=aH\).  Then
	\[
	h_{\overline H}(y)=\frac2a\langle y,q\rangle_N
	\]
	for every \(y\in A(\overline K)\).
\end{corollary}
\begin{proof}
	First assume that \(H\) is integral and that \(q\in A(\overline K)\).  With the
	above convention,
	\(\phi_N(q)=T_q^*N\otimes N^{-1}\), rigidified at the origin.  The canonical
	height of this rigidified Picard-zero bundle is therefore
	\[
	h_{\overline{\phi_N(q)}}(y)
	=h_{\overline N}(y+q)-h_{\overline N}(y)-h_{\overline N}(q)
	=2\langle y,q\rangle_N.
	\]
	If \(\phi_N(q)=aH\), then
	\[
	a\,h_{\overline H}(y)=h_{\overline{\phi_N(q)}}(y),
	\]
	and the formula follows.  The real case follows by linearity of Picard-zero
	heights and of the Neron--Tate bilinear form.
\end{proof}
\begin{remark}[Neron local-height convention]\label{rem:neron-local-height-convention}
	The preceding proposition uses the convention that the theta norm associated
	with a rigidified Picard-zero bundle is the normalized Neron local function.
	Remark~\ref{rem:theta-neron-source-status} records the source chain: Zhang's
	canonical adelic metric, Chambert-Loir's Picard-zero formulation, and the
	Raynaud/K\"uhne theta expression of the corresponding local norm.
\end{remark}
\begin{cor}[Global torsion cancellation for Picard-zero theta factors]
	\label{cor:torsion-picard-global-cancellation}
	With the normalization of Proposition~\ref{prop:canonical-picard-zero-theta-metrics}, if
	\(x\in A(\overline K)\) is torsion and
	\(x\notin|\operatorname{div}(s)|\), then
	\[
	\sum_v n_v\bigl(-\log\Vert f_{s,v}(\widetilde x)\Vert_v\bigr)=0.
	\]
\end{cor}
\begin{proof}
	By Proposition~\ref{prop:theta-neron-height}, the displayed sum is the
	canonical height \(h_{\overline H}(x)\) attached to the rigidified
	Picard-zero bundle \(H\).  This height is linear on the torsion-free quotient
	and vanishes on torsion points.
\end{proof}

\begin{remark}[Local source ledger]\label{rem:local-triv-source-ledger}
	The construction in this section uses three standard analytic tools.  At non-archimedean
	places the analytic uniformization is Raynaud--Bosch--Lutkebohmert
	uniformization \cite{TJFA}, recalled in the form needed here by \cite{FRSS};
	the analytic skeleton and the contractible cover are taken from Berkovich's
	construction \cite{Ber}; and the tropicalization of theta functions follows the
	tropical theta formalism in \cite{FRSS}.  At archimedean places we use K\"uhne's
	local trivialization
	\cite[Lemma 11]{Ku1}.  The statements below record exactly the formal
	consequences of these tools used later in the proof.
	
	For the theta-descent input, the precise source chain is
	\cite[Sections~3.2--3.3 and Theorem~3.6]{Ku1} for Raynaud uniformization and
	line-bundle descent, \cite[Propositions~3.8--3.9]{Ku1} for Picard-zero
	translation, \cite[Section~3.6 and Proposition~3.13]{Ku1} for Fourier
	expansions and the theta transformation law, and \cite[Theorem~4.7]{Ku1} for
	the global theta-height comparison.  Proposition~\ref{input:local-triv-analytic}
	isolates exactly the portion used in Theorem~\ref{thm:local-trivialization};
	Proposition~\ref{prop:canonical-picard-zero-theta-metrics} fixes the metric
	normalization needed later.
\end{remark}

\begin{proposition}[Semiabelian tropicalization]
	\label{input:torsion-tropicalization}
	For every place \(v\), the Raynaud/K\"uhne construction gives a canonical
	real-valued tropicalization homomorphism
	\[
	\operatorname{trop}_{G,v}:G^{an}_{\mathbb C_v}\longrightarrow N_{\BR}
	\]
	attached to the split torus \(T\subset G\).  Equivalently, for each character
	\(u\in M\), pushing out the extension \(0\to T\to G\to A\to0\) by
	\(\chi^u:T\to\mathbb G_m\) gives a rigidified algebraically trivial line bundle
	\(H_u\) on \(A\), and
	\[
	\langle u,\operatorname{trop}_{G,v}(y)\rangle
	=-\log\|e_u(y)\|_{H_u,v}
	\]
	for the canonical Raynaud/K\"uhne metric.  This map extends the ordinary
	tropicalization of \(T^{an}_{\mathbb C_v}\) and is a group homomorphism.
	Consequently, if \(q\in G(\overline K)\) is torsion, then
	\[
	\operatorname{trop}_{G,v}(q)=0\in N_{\BR}.
	\]
\end{proposition}

\begin{lemma}[Coordinate tropicalization]
	\label{lem:local-coordinates-compute-tropicalization}
	Let \(\psi_{j,v}\) be one of the Raynaud/K\"uhne local toric coordinate systems
	constructed in Theorem~\ref{thm:local-trivialization}.  If \(y\in
	G^{an}_{\mathbb C_v}\) is in the chart, then
	\[
	\val_v(\psi_{j,v}(y))=\operatorname{trop}_{G,v}(y)\in N_{\BR}.
	\]
\end{lemma}

\begin{proof}
	It is enough to pair with every character \(u\in M\).  On a Raynaud formal
	splitting in the non-archimedean case, the map \(\psi_{j,v}\) is the toric
	coordinate in the product \(T^{an}\times U\); hence the assertion is exactly the
	local description of tropicalization as the first projection followed by the
	ordinary torus tropicalization.  At archimedean places this is the corresponding
	unitary local coordinate in K\"uhne's local trivialization.
	
	Equivalently in the notation used in the proof of
	Theorem~\ref{thm:local-trivialization}, the pullback of the character
	\(\chi^u\) through \(\psi_{j,v}\) is the local expression of the pushout vector
	\(e_u(y)\) in the rigidified Picard-zero bundle \(H_u\).  The canonical
	Raynaud/K\"uhne metric in this trivialization is the ordinary absolute value of
	that coefficient.  Therefore
	\[
	\bigl\langle u,\val_v(\psi_{j,v}(y))\bigr\rangle
	=-\log|\chi^u(\psi_{j,v}(y))|_v
	=-\log\|e_u(y)\|_{H_u,v}
	=\bigl\langle u,\operatorname{trop}_{G,v}(y)\bigr\rangle .
	\]
	The overlap calculation in the proof of
	Theorem~\ref{thm:local-trivialization} shows that transition multipliers lie in
	the compact torus, hence have valuation zero.  Thus the equality is independent
	of the chosen chart.
\end{proof}

\begin{lemma}[Torsion lifts have zero toric valuation]
	\label{lem:torsion-lift-zero-valuation}
	Let \(q\in G(\overline K)_{\operatorname{tor}}\).  For every place \(v\) and
	every local trivialization chart \(\psi_{j,v}\) defined at \(\pi(q)\), the
	toric coordinate of \(q\) has valuation zero:
	\[
	\val_v(\psi_{j,v}(q))=0\in N_{\BR}.
	\]
	Consequently, for every \(t\in T(\overline K)\),
	\[
	\val_v(\psi_{j,v}(q\,t))=\val_v(t),
	\]
	where the right hand side is the ordinary valuation of the split torus \(T\).
\end{lemma}

\begin{proof}
	By Proposition~\ref{input:torsion-tropicalization},
	\(\operatorname{trop}_{G,v}(q)=0\).  Lemma
	\ref{lem:local-coordinates-compute-tropicalization} identifies this
	tropicalization with \(\val_v(\psi_{j,v}(q))\) in every local chart, giving the
	first assertion.  Multiplication by \(t\in T\) is
	multiplication in the toric coordinate of the chart, so valuation additivity on
	the split torus gives
	\[
	\val_v(\psi_{j,v}(q\,t))
	=\val_v(\psi_{j,v}(q))+\val_v(t)
	=\val_v(t).
	\]
\end{proof}

\begin{prop}[Valuation and height package]\label{prop:section4-valuation-height-package}
	Let \(\bD\) be a toric metrized \(\BR\)-divisor on \(X_\Sigma\), and let
	\(\overline G(\bD)\) be the adelic line bundle on \(\overline G\) obtained by
	pulling back the local toric metrics through the local trivializations
	\(\psi_{j,v}\).  Let \(x\in G(\overline K)\), and let \(O_v(x)\) denote the
	finite set of its \(v\)-adic Galois conjugates in \(G(\mathbb C_v)\).  For every
	place \(v\), choose, for each \(y\in O_v(x)\), an index \(j(y)\) with
	\(\pi(y)\in U_{j(y),v}\), and put
	\[
	\nu_{x,v}
	=
	\frac1{\#O_v(x)}
	\sum_{y\in O_v(x)}
	\delta_{\val_v(\psi_{j(y),v}(y))}.
	\]
	Then:
	\[
	\nu_{x,v}\ \hbox{is independent of the choices of }j(y),
	\]
	\[
	\sum_v n_v\int_{N_\BR}u\,d\nu_{x,v}(u)
	=
	\bigl(-h_{\overline H_1}(\pi(x)),\ldots,
	-h_{\overline H_t}(\pi(x))\bigr),
	\]
	and
	\[
	h_{\overline G(\bD)}(x)
	=
	-\sum_v n_v\int_{N_\BR}\Psi_{\bD,v}\,d\nu_{x,v}.
	\]
	Moreover, if \(\pi(x)\) is torsion and \(x=q\cdot t\), where
	\(q\in G(\overline K)_{\operatorname{tor}}\) and \(t\in T(\overline K)\), then
	\[
	h_{\overline G(\bD)}(x)=h_{\bD}(t).
	\]
	Consequently, for
	\(\bL=\overline G(\bD)\otimes\overline\pi^*\bN\), with \(\bN\) symmetric ample
	and canonically metrized on \(A\), one has \(h_{\bL}(x)=h_{\bD}(t)\) on such a
	torsion fiber.
\end{prop}

\begin{proof}
	The independence of the choices follows from
	Theorem~\ref{thm:local-trivialization} at non-archimedean places and from
	K\"uhne's archimedean trivialization \cite[Lemma 11]{Ku1}: on an overlap the two
	charts have the same valuation map.  Hence the measure \(\nu_{x,v}\) is
	well-defined.
	
	Let \(e_1^\vee,\ldots,e_t^\vee\) be the character basis used in the construction
	of the functions \(g_i/f_{s_i}\).  On a chart \(U_{j,v}\), the \(i\)-th
	coordinate of \(\val_v(\psi_{j,v}(y))\) is
	\[
	\langle e_i^\vee,\val_v(\psi_{j,v}(y))\rangle
	=
	-\log |g_i(y)|_v
	+\log\Vert f_{s_i,v}(\widetilde{\pi_v}(y))\Vert .
	\]
	After averaging over \(O_v(x)\) and summing over \(v\), the terms
	\(-\log |g_i(y)|_v\) vanish by the product formula.  The remaining theta term is
	identified with \(-h_{\overline H_i}(\pi(x))\) by
	Proposition~\ref{prop:theta-neron-height}.  This gives the displayed formula
	for the center of the adelic valuation measures.
	
	For the height identity, use the canonical toric section of \(D\), which is
	nowhere vanishing on the open torus.  By construction of the metric
	\(\overline G(\bD)\),
	\[
	-\log\Vert\widetilde{s_D}(y)\Vert_{\overline G(\bD),v}
	=
	-\Psi_{\bD,v}(\val_v(\psi_{j(y),v}(y))).
	\]
	Averaging over the \(v\)-adic conjugates and summing over all places gives the
	height formula.  Linearity then gives the same identity for toric metrized
	\(\BR\)-divisors.
	
	Finally suppose \(\pi(x)\) is torsion.  Over \(\overline K\), the fiber above
	\(\pi(x)\) contains a torsion point \(q\); hence \(x=q\cdot t\) for some
	\(t\in T(\overline K)\).  By
	Lemma~\ref{lem:torsion-lift-zero-valuation}, the valuation measure of \(x\) is
	the toric valuation measure of \(t\).  Therefore the above height formula
	reduces exactly to the toric height \(h_{\bD}(t)\).  The canonical height of
	\(\pi(x)\) with respect to \(\bN\) is zero, so the same identity holds for
	\(\bL=\overline G(\bD)\otimes\overline\pi^*\bN\).
\end{proof}

\begin{corollary}[Torsion-lift height invariance]
	\label{cor:torsion-lift-height-invariance}
	Let \(q\in G(\overline K)_{\operatorname{tor}}\) and \(t\in T(\overline K)\).
	For every toric metrized \(\BR\)-divisor \(\bD\),
	\[
	h_{\overline G(\bD)}(q\,t)=h_{\bD}(t).
	\]
	If
	\[
	\bL=\overline G(\bD)\otimes\overline\pi^*\bN
	\]
	with \(\bN\) symmetric ample and canonically metrized on \(A\), then
	\[
	h_{\bL}(q\,t)=h_{\bD}(t).
	\]
\end{corollary}

\begin{proof}
	Apply Proposition~\ref{prop:section4-valuation-height-package} to
	\(x=q\,t\).  Since \(q\) is torsion, \(\pi(q\,t)=\pi(q)\) is torsion, and the
	canonical height of \(\pi(q)\) is zero.
\end{proof}

\begin{lemma}[Height lower bound from arithmetic \(T\)-effectivity]
	\label{lem:arith-teff-height-lower}
	Let \(\bD\) be arithmetically \(T\)-effective, and let
	\[
	\bL=\overline G(\bD)\otimes\overline\pi^*\bN,
	\]
	where \(\bN\) is a symmetric ample line bundle on \(A\) equipped with its
	canonical metric.  Then, for every \(x\in G(\overline K)\),
	\[
	h_{\bL}(x)\ge \muabs_{\bD}(T).
	\]
\end{lemma}

\begin{proof}
	Choose the small distinguished section \(s_0\) and constants \((c_v)_v\) from
	Definition~\ref{def:arith-T-effective}.  Since \(s_0\) is torus-invariant, it is
	nowhere vanishing on the open torus \(T\).  Hence its induced section
	\(\widetilde s_0\) of \(\overline G(D)\) is nowhere vanishing on \(G\).  For every
	\(v\)-adic conjugate \(y\) of \(x\), choose a local trivialization
	\(\psi_{j,v}\) at \(\pi(y)\).  By construction of the induced metric,
	\[
	\|\widetilde s_0(y)\|_{\overline G(\bD),v}
	=
	\|s_0(\psi_{j,v}(y))\|_{\bD,v}
	\le
	\|s_0\|_{v,\sup}
	\le e^{-c_v}.
	\]
	Therefore
	\[
	h_{\overline G(\bD)}(x)
	=
	\sum_v-\frac{n_v}{\#O_v(x)}
	\sum_{y\in O_v(x)}
	\log\|\widetilde s_0(y)\|_{\overline G(\bD),v}
	\ge
	\sum_v n_vc_v
	=
	\muabs_{\bD}(T).
	\]
	The canonical height attached to the symmetric ample bundle \(\bN\) is
	nonnegative, so
	\[
	h_{\bL}(x)=h_{\overline G(\bD)}(x)+h_{\bN}(\pi(x))
	\ge \muabs_{\bD}(T).
	\]
\end{proof}

\begin{lemma}[Torsion fibers approximate the toric minimum]
	\label{lem:torsion-fiber-toric-min}
	Assume that \(D\) is big and \(\bD\) is vertically semipositive.  Let
	\(\bL=\overline G(\bD)\otimes\overline\pi^*\bN\), with \(\bN\) as above.  For
	every nonempty open subset \(U\subseteq G\) and every \(\epsilon>0\), there is a
	point \(x\in U(\overline K)\) such that
	\[
	h_{\bL}(x)\le \muabs_{\bD}(T)+\epsilon.
	\]
\end{lemma}

\begin{proof}
	The image \(\pi(U)\) contains a nonempty open subset of \(A\), since
	\(\pi:G\to A\) is smooth and surjective.  Torsion points are Zariski dense in
	\(A\), so choose \(a\in A(\overline K)_{\operatorname{tor}}\cap \pi(U)\).  The
	torsion point \(a\) has a torsion lift
	\(q\in G(\overline K)_{\operatorname{tor}}\): if \(na=0\), choose any
	\(z\in G(\overline K)\) above \(a\), then \(nz\in T\), and divisibility of the
	torus over \(\overline K\) allows one to multiply \(z\) by a point of \(T\) to
	obtain a lift killed by \(n\).
	
	Translation by \(q^{-1}\) identifies the open set
	\[
	U\cap\pi^{-1}(a)
	\]
	with a nonempty open subset \(W\subseteq T\).  By the toric successive-minima
	formula \cite[Theorem 3.9 and Corollary 3.10]{BGPS15},
	\[
	\muess_{\bD}(X_\Sigma)=\muabs_{\bD}(T),
	\]
	and hence every nonempty open subset of \(T\) contains points whose
	\(\bD\)-height is at most \(\muabs_{\bD}(T)+\epsilon\).  Choose such a point
	\(t\in W(\overline K)\), and put \(x=qt\in U(\overline K)\).  Since
	\(\pi(x)=a\) is torsion, Corollary~\ref{cor:torsion-lift-height-invariance}
	gives
	\[
	h_{\bL}(x)=h_{\bD}(t)\le \muabs_{\bD}(T)+\epsilon.
	\]
\end{proof}

\begin{prop}[Section 4 minima formula]\label{prop:section4-minima-formula}
	If \(D\) is big and \(T\)-effective, \(\bD\) is vertically semipositive and
	arithmetically \(T\)-effective, and \(\bN\) is a symmetric ample line bundle
	with its canonical metric on \(A\), then for
	\[
	\bL=\overline G(\bD)\otimes\overline\pi^*\bN
	\]
	one has
	\[
	\muess_{\bL}(\overline G)
	=
	\muabs_{\bL}(G)
	=
	\muess_{\bD}(X_\Sigma)
	=
	\muabs_{\bD}(T).
	\]
\end{prop}

\begin{proof}
	By Lemma~\ref{lem:arith-teff-height-lower},
	\[
	\muabs_{\bL}(G)\ge \muabs_{\bD}(T).
	\]
	On the other hand, let \(V\subseteq\overline G\) be a nonempty open subset.
	Since \(G\) is dense in \(\overline G\), the intersection \(V\cap G\) is a
	nonempty open subset of \(G\).  Lemma~\ref{lem:torsion-fiber-toric-min} gives,
	for every \(\epsilon>0\), a point \(x\in (V\cap G)(\overline K)\) with
	\[
	h_{\bL}(x)\le \muabs_{\bD}(T)+\epsilon.
	\]
	Taking the infimum over \(V\), then the supremum over all such \(V\), and then
	letting \(\epsilon\to0\), gives
	\[
	\muess_{\bL}(\overline G)\le \muabs_{\bD}(T).
	\]
	Since always \(\muess_{\bL}(\overline G)\ge\muabs_{\bL}(G)\), we obtain
	\[
	\muess_{\bL}(\overline G)=\muabs_{\bL}(G)=\muabs_{\bD}(T).
	\]
	Finally, the toric equality
	\[
	\muess_{\bD}(X_\Sigma)=\muabs_{\bD}(T)
	\]
	is the toric successive-minima formula
	\cite[Theorem 3.9 and Corollary 3.10]{BGPS15}.
\end{proof}
\section{Canonical metric case}

\subsection{Auxiliary n-division tower}

We fix the auxiliary semiabelian tower used throughout this section.  Let
\[
\eta_G=(H_1,\ldots,H_t)\in A^\vee(\overline K)^t
\simeq \operatorname{Ext}^1_{\overline K}(A,T)
\]
be the extension class of \(G\).  For every \(n\ge1\), choose
\[
H_i^{(n)}\in A^\vee(\overline K),\qquad 1\le i\le t,
\]
such that \(nH_i^{(n)}=H_i\).  After passing to a finite extension
\(K_n\supset K\) over which these data are defined, let \(G_n\) be the semiabelian
variety over \(K_n\) with extension class
\[
\eta_{G_n}=(H_1^{(n)},\ldots,H_t^{(n)}).
\]
There is an isogeny
\[
\varphi_n:G_n\longrightarrow G_{K_n}
\]
whose toric part is multiplication by \(n\) and whose abelian part is the
identity on \(A_{K_n}\).

We again write \(\overline G_n\) for the toric compactification of \(G_n\)
associated with the base change \(X_{\Sigma,K_n}\), and
\[
\overline\pi_n:\overline G_n\longrightarrow A_{K_n}
\]
for the induced projection.  The homomorphism \(\varphi_n\) extends to the
corresponding toric compactifications; throughout the section we denote this
extension by the same symbol
\[
\varphi_n:\overline G_n\longrightarrow \overline G_{K_n}.
\]

We use the same symbols \(D\) and \(N\) for their base changes to \(K_n\).  This
causes no ambiguity in the toric part, since \(D_{K_n}\) has the same virtual
support function as \(D\).  Put
\[
L_n=G_n(D)\otimes\pi_n^*N_{K_n},
\qquad
\overline L_n^{\can}
=
\overline G_n(\overline D^{\can})\otimes
\overline\pi_n^*\overline N_{K_n}.
\]
Finally, fix once and for all a place \(v_0\in\Sigma(K)\).  For a real number
\(\kappa_n\), let \(O_{\overline G}(\kappa_n)\) denote the trivial adelic line
bundle on \(\overline G\) whose metric is trivial at all places except
\(v_0\), and whose \(v_0\)-metric is defined by
\[
-\log\|1\|_{v_0}=\kappa_n .
\]
Its pullback by \(\varphi_n\) is the constant vertical correction used below.

\subsection{Lower bound and semipositivity}

We first record the standard canonical lower bound.

\begin{theorem}[Fundamental height inequality]
	\label{input:fundamental-height-inequality}
	Let \(\overline M\) be a vertically semipositive integrable adelic line bundle
	on a projective variety \(Y\), with \(\deg_M(Y)>0\).  With the normalization of
	heights used in this paper,
	\[
	\muabs_{\overline M}(Y)\le h_{\overline M}(Y)\le \muess_{\overline M}(Y).
	\]
\end{theorem}

\begin{lemma}[Constant correction for horizontal semipositivity]
	\label{input:constant-correction-horizontal}
	In the auxiliary semiabelian setting of this section, let
	\(\overline M_n\) be a vertically semipositive adelic line bundle with nef
	underlying line bundle on \(\overline G_n\).  If, for some \(\delta_n\to0\),
	\[
	\muabs_{\overline M_n}(\overline G_n)\ge -C\delta_n,
	\qquad
	h_{\overline M_n}(\overline G_n)\ge -C\delta_n,
	\]
	then there is a constant vertical correction
	\[
	\overline M_n\otimes \varphi_n^*O_{\overline G}(\kappa_n),
	\qquad 0\le \kappa_n\le C'\delta_n,
	\]
	which is horizontally semipositive.  This is the constant-correction form used
	in K\"uhne's lower-bound argument before the arithmetic volume estimate.
\end{lemma}

\begin{lemma}[Canonical chart lower bound]
	\label{lem:canonical-chart-lower}
	Assume that \(D\) is an integral ample \(T\)-effective toric divisor.  After
	replacing \(D\) by a positive multiple and dividing the resulting height
	inequalities by that multiple, there are lattice points
	\[
	m_0=0,\ m_1,\ldots,m_r\in\Delta_D\cap M
	\]
	such that the effective toric divisors
	\[
	D+\operatorname{div}\chi^{m_j}\qquad (0\le j\le r)
	\]
	have empty common intersection.  For every \(n\ge1\) and every
	\(x\in\overline G_n(\overline K)\), one can choose \(j\) with
	\[
	h_{\overline G_n(\overline D^{\can})}(x)
	\ge
	-h_{\overline H^{(n)}_{m_j}}(\pi_n(x)),
	\]
	where \(H^{(n)}_{m_j}\) is the Picard class on \(A\) induced by \(m_j\) for the
	extension \(G_n\).
\end{lemma}

\begin{proof}
	The existence of \(m_0,\ldots,m_r\) is the standard toric separation argument:
	for a sufficiently divisible multiple of \(D\), the polytope contains lattice
	points whose corresponding toric sections have no common zero.  The case
	\(m_0=0\) is the \(T\)-effective section.
	
	Fix \(x\).  Choose \(j\) such that the section associated with
	\(D+\operatorname{div}\chi^{m_j}\) does not vanish at the toric coordinate of
	\(x\).  On this chart, the induced section on \(\overline G_n\) is a section of
	\(\overline G_n(D)\) twisted by the Picard-zero class \(H^{(n)}_{m_j}\) on the
	abelian quotient.  The canonical toric metric makes the effective toric section
	small on the corresponding affine chart.  Using the height identity in
	Proposition~\ref{prop:section4-valuation-height-package}, together with the
	theta-function interpretation of the Picard-zero term from
	Proposition~\ref{prop:theta-neron-height}, gives
	\[
	h_{\overline G_n(\overline D^{\can})}(x)
	+h_{\overline H^{(n)}_{m_j}}(\pi_n(x))\ge0.
	\]
	This is the required inequality.
\end{proof}

\begin{lemma}[Bounding the Picard-zero translation terms]
	\label{lem:picard-translation-bound}
	Let \(m_1,\ldots,m_r\) be fixed as in
	Lemma~\ref{lem:canonical-chart-lower}.  There is a constant \(C>0\), independent
	of \(n\), such that for every \(j\), every \(n\ge1\), and every
	\(y\in A(\overline K)\),
	\[
	h_{\overline N}(y)-h_{\overline H^{(n)}_{m_j}}(y)\ge -Cn^{-2}.
	\]
\end{lemma}

\begin{proof}
	Let \(\phi_N:A\to A^\vee\) be the polarization attached to \(N\).  Since
	\(\phi_N\) is an isogeny, after replacing \(K\) by a finite extension and after
	clearing denominators there are a positive integer \(a\) and points
	\(q_j\in A(\overline K)\) such that
	\[
	\phi_N(q_j)=aH_{m_j}.
	\]
	Choose \(q_{j,n}\in A(\overline K)\) with \(nq_{j,n}=q_j\).  Then
	\[
	\phi_N(q_{j,n})=aH^{(n)}_{m_j}.
	\]
	Let \(\langle\cdot,\cdot\rangle_N\) be the Neron--Tate bilinear form associated
	with \(N\).  By Corollary~\ref{cor:picard-zero-height-pairing},
	\[
	h_{\overline H^{(n)}_{m_j}}(y)
	=
	\frac{2}{a}\langle y,q_{j,n}\rangle_N .
	\]
	Hence, by completing the square for the positive semidefinite quadratic form
	\(h_{\overline N}\),
	\[
	h_{\overline N}(y)-h_{\overline H^{(n)}_{m_j}}(y)
	\ge
	-a^{-2}h_{\overline N}(q_{j,n}).
	\]
	Since \(h_{\overline N}(q_{j,n})=n^{-2}h_{\overline N}(q_j)\), the desired
	uniform bound follows.
\end{proof}

\begin{proposition}[Canonical lower bound]
	\label{prop:canonical-lower-bound}
	Assume that \(D\) is ample and \(T\)-effective.  There is a constant
	\(C=C(G,D,N)>0\) such that for all \(n\ge 1\),
	\[
	\mu^{\mathrm{abs}}_{\ol L_n^{\can}}(\ol G_n)\ge -C n^{-2}
	\]
	and
	\[
	-C n^{-2}\le h_{\ol L_n^{\can}}(\ol G_n)\le 0.
	\]
	Consequently there is \(0\le\kappa_n\le Cn^{-2}\) such that
	\[
	\ol L_n^{\can}\otimes\varphi_n^*O_{\ol G}(\kappa_n)
	\]
	is horizontally semipositive.
\end{proposition}

\begin{proof}
	We first treat the case in which \(D\) is an integral ample toric divisor.
	The general ample \(\bR\)-divisor case follows by writing the support function
	of \(D\) as a positive real linear combination of support functions of ample
	integral toric divisors and using additivity of the corresponding heights.
	
	For every \(x\in \overline G_n(\overline K)\), choose \(j\) by
	Lemma~\ref{lem:canonical-chart-lower}.  Combining that lemma with
	Lemma~\ref{lem:picard-translation-bound} gives
	\[
	h_{\ol L_n^{\can}}(x)
	=h_{\ol G_n(\ol D^{\can})}(x)+h_{\ol N}(\pi_n(x))
	\ge -Cn^{-2}
	\]
	for a constant \(C\) independent of \(n\).  Hence
	\[
	\mu^{\mathrm{abs}}_{\ol L_n^{\can}}(\ol G_n)\ge -Cn^{-2}.
	\]
	
	The dense open semiabelian part \(G_n\subset\overline G_n\) contains a
	Zariski-dense family of points of height \(0\): take torsion points on the
	abelian quotient and toric points of canonical toric height \(0\), and use
	Proposition~\ref{prop:section4-valuation-height-package}.  Equivalently,
	Proposition~\ref{prop:section4-minima-formula} applied to the canonical toric
	metric gives \(\muess_{\ol L_n^{\can}}(G_n)=0\).  Thus
	\(\muess_{\ol L_n^{\can}}(\overline G_n)\le0\).  By
	Theorem~\ref{input:fundamental-height-inequality},
	\[
	-Cn^{-2}
	\le
	h_{\ol L_n^{\can}}(\ol G_n)
	\le
	\mu^{\mathrm{ess}}_{\ol L_n^{\can}}(\ol G_n)
	\le0 .
	\]
	
	Finally apply Lemma~\ref{input:constant-correction-horizontal} with
	\(\delta_n=n^{-2}\).  After enlarging \(C\), it gives
	\(0\le\kappa_n\le Cn^{-2}\) and the asserted horizontal semipositivity of
	\[
	\ol L_n^{\can}\otimes\varphi_n^*O_{\ol G}(\kappa_n).
	\]
\end{proof}

\begin{remark}
	The assertion ``\(\phi_N:A(K)\to A^\vee(K)\) is surjective'' is false in
	general.  What is used in the argument is surjectivity of the isogeny as a
	morphism, equivalently on \(\ol K\)-points, after replacing \(K\) by a finite
	extension if necessary.
\end{remark}

\subsection{The three K\"uhne-type estimates}

The canonical proof needs the following three estimates.  They make explicit
the K\"uhne-type estimates used below.

\begin{remark}[Source-theorem ledger for the asymptotic estimates]
	\label{rem:source-ledger-asymptotic}
	The two load-bearing analytic estimates in this subsection are the first
	variation estimate and the arithmetic volume comparison.  In K\"uhne's proof of
	semiabelian equidistribution, the corresponding statements in the published
	JEMS version are \cite[Lemma~4.4]{Ku1} and \cite[Lemma~4.6]{Ku1}.
	Lemma~4.4 identifies the linear
	term in the height variation with integration against the local
	Chambert-Loir measure and bounds the remaining terms by a quadratic error
	whose constant grows linearly with the auxiliary isogeny parameter.
	
	The arithmetic-volume input has a slightly longer source chain, and we record
	it explicitly.  First, \cite[Lemma~2.5(a),(b)]{Ku1} gives the constant-shift
	formula and continuity properties for the \(\chi\)-arithmetic volume.  Second,
	\cite[Lemma~2.6]{Ku1} proves that, for semipositive metrized line bundles
	\(\widetilde L,\widetilde M\) on a \(d\)-dimensional arithmetic variety,
	\[
	\widehat{\operatorname{vol}}_\chi(\widetilde L-\widetilde M)
	\ge
	\widetilde L^{d+1}-(d+1)\widetilde L^d\cdot\widetilde M .
	\]
	K\"uhne proves this by invoking Ikoma's arithmetic-volume form of Yuan's
	bigness theorem \cite[Theorem 3.5.3 and Remark 3.5.4]{Iko13}, ultimately
	resting on Yuan's arithmetic bigness theorem \cite{Yua1}.  Third,
	\cite[Lemma~4.6]{Ku1} applies this bigness input to an integrable perturbation,
	written as a difference of semipositive metrics, and expands the remaining
	high-order intersections to obtain the required quadratic volume lower bound.
	
	We do not use this package as an unqualified black box.  K\"uhne works with the
	standard compactification \((\mathbb P^1)^t\), whereas here the toric
	compactification is \(X_\Sigma\) and the toric divisor \(D\) is arbitrary.  The
	new verification needed in this paper is the bridge from the source package to
	the present setting: the projection formula, the top-degree numerical reduction, and
	the intersection scaling under the auxiliary isogenies give the same \(O(n)\)
	error in the canonical case, and the compressed pullback \(f_m=f\circ[m]\)
	adds exactly the additional \(O(m)\) factor isolated below.
\end{remark}

\begin{lemma}[Top-degree numerical reduction]
	\label{lem:algebraic-dimensional-vanishing}
	Let \(\overline G=G\times^T X_\Sigma\) be the toric compactification used in
	this paper, with \(\dim T=t\) and \(\dim A=g\).  Let \(M_1,\ldots,M_r\) be line
	bundles on \(A\), and let \(D_1,\ldots,D_s\) be toric Cartier divisors on
	\(X_\Sigma\), with \(r+s=g+t\).  Then the top-degree numerical intersection
	on \(\overline G\), and likewise on every auxiliary \(\overline G_n\),
	\[
	\deg\!\left(\pi^*M_1\cdots \pi^*M_r\cdot
	\overline G(D_1)\cdots \overline G(D_s)\right),
	\]
	is zero unless \(r=g\) and \(s=t\).  This is only a numerical statement about
	a complete intersection; no vanishing of the corresponding Chow class is asserted.
\end{lemma}

\begin{proof}
	If \(r>g\), push the intersection to the base.  Suppose \(s>t\).  For a
	smooth fan, the relative toric Chow-ring presentation
	\cite[Theorem~1.2(iii)]{SankaranUma2003} expresses every excess
	relative divisor factor through first Chern classes of the character line
	bundles defining the associated toric bundle.  These line bundles lie in
	\(\operatorname{Pic}^0(A)\), so their first Chern classes are numerically
	trivial and the resulting top-degree number is zero.  For a normal fan, take a
	smooth toric refinement and use the projection formula.  The same argument
	applies after the auxiliary base extension and isogeny.  Since
	\(r+s=g+t\), the only remaining case is \((r,s)=(g,t)\).
\end{proof}

\begin{remark}
	No arbitrary adelic-intersection or Chow-class vanishing is used in the
	sequel.  The height-variation and arithmetic-volume arguments use the explicit
	projection formula and the normalized scaling estimate below.  The preceding
	lemma is used only for top degrees and total masses of mixed semipositive
	measures.
\end{remark}

\begin{lemma}[Normalized scaling in the perturbation family]
	\label{lem:toric-dimensional-scaling}
	Retain the preceding notation.  The scaling assertion used below is the
	following restricted normalized assertion.
	Let \(F_n=\varphi_n^*O_{\ol G}(f)\), where
	\(O_{\ol G}(f)\) is an integrable adelic line bundle with trivial underlying
	algebraic class, and let \(d=g+t\).  In the quadratic core of the error
	estimates, namely in a normalized intersection expression with two perturbing
	factors \(F_n\) and the remaining factors coming from the reference bundle,
	division by \(L_n^d\) leaves at most one explicit factor \(n\).  If the test
	metric is first compressed, \(f_m=f\circ[m]\), so that
	\(F_{m,n}=\varphi_n^*O_{\ol G}(f_m)\), the same quadratic core leaves at most
	one explicit factor \(n\) and at most one explicit factor \(m\).  Higher-order
	error monomials are reduced to this quadratic core in
	Lemma~\ref{lem:parameterized-error-bookkeeping}.  The constants are uniform for
	\(m\le n\) in the range used below.
\end{lemma}

\begin{proof}
	We first argue for model metrics and write
	\(M_n=\overline G_n(D)\).  The exact scaling comes from the projection
	formula, not merely from counting algebraic toric factors.  Since
	\[
	\varphi_n^*\overline G(D)=nM_n,
	\qquad
	\varphi_n^*\pi^*N=\pi_n^*N,
	\qquad
	\deg(\varphi_n)=n^t,
	\]
	the projection formula gives, for every \(j\le g\),
	\[
	F_n^2\cdot M_n^{d-1-j}\cdot(\pi_n^*N)^j
	=
	n^{\,t-(d-1-j)}
	\bigl(O_{\overline G}(f)^2\cdot
	\overline G(D)^{d-1-j}\cdot(\pi^*N)^j\bigr).
	\]
	Because \(d=g+t\), the exponent is \(j-g+1\le1\).  Terms with \(j>g\)
	vanish by the projection formula on the abelian base.  After expansion of the
	reference bundle and division by the fixed top degree \(L_n^d\), the
	quadratic core is therefore \(O(n)\).  This is the ambient analogue of the
	exact pullback calculation in \cite[proof of Lemma~4.4, equations~(4.5)--(4.6)]{Ku1}.
	
	For the compressed family, put \(F_m=[m]^*O_{\overline G}(f)\) and
	\(D_m=m^{-1}[m]^*D\).  Applying the same projection-formula calculation to
	\([m]\), whose degree on the toric part is \(m^t\), gives
	\[
	F_m^2\cdot \overline G(D_m)^{d-1-j}\cdot(\pi^*N)^j
	=O\bigl(m^{\,j-g+1}\bigr)=O(m)
	\]
	for \(j\le g\).  Combining this with the \(n\)-division calculation yields
	the required \(O(nm)\) bound for the compressed quadratic core.  The
	higher-order terms are treated separately in
	Lemma~\ref{lem:parameterized-error-bookkeeping}; no higher-order conclusion is
	being hidden in the present quadratic calculation.
	
	For integrable metrics, write the perturbation as a difference of semipositive
	model metrics, apply the preceding identities term by term, and pass to the
	limit by continuity of adelic intersections.  The constants depend only on
	the fixed compactification, \(D,N\), the reference models, and the chosen
	integrable norm of \(f\), and are uniform for \(m\le n\).
\end{proof}

\begin{lemma}[Parameterized error bookkeeping]
	\label{lem:parameterized-error-bookkeeping}
	Let \(O_{\ol G}(f)\) be an integrable adelic line bundle with trivial
	underlying algebraic class.  Put
	\[
	(\ol M_{\star,n},F_{\star,n},\alpha_{\star,n})
	=
	(\ol L_n^{\can},\varphi_n^*O_{\ol G}(f),n)
	\]
	in the canonical case, and
	\[
	(\ol M_{\star,n},F_{\star,n},\alpha_{\star,n})
	=
	(\ol L_{(m,n)},\varphi_n^*O_{\ol G}(f\circ[m]_{\ol G}),nm)
	\]
	in the compressed case, with \(m\le n\).  Then for every
	\(2\le r\le d+1\) and every \(|\lambda|\le\alpha_{\star,n}^{-1}\),
	\[
	|\lambda|^r
	\left|
	F_{\star,n}^r\cdot \ol M_{\star,n}^{d+1-r}
	\right|
	\le
	C_f L_n^d |\lambda|^2\alpha_{\star,n},
	\]
	where \(C_f\) is independent of \(n\), of \(m\le n\), of \(r\), and of
	\(\lambda\).
\end{lemma}

\begin{proof}
	We first assume that \(f\) is a model function.  For \(r=2\), the estimate is
	exactly the quadratic core estimate from
	Lemma~\ref{lem:toric-dimensional-scaling}, together with the arithmetic Hodge
	index estimate for two algebraically trivial metric factors as used in
	K\"uhne's first variation argument \cite[Lemma~4.4]{Ku1}.  Thus
	\[
	\left|
	F_{\star,n}^2\cdot \ol M_{\star,n}^{d-1}
	\right|
	\le
	C_f L_n^d\alpha_{\star,n}.
	\]
	
	For \(r>2\), the same projection-formula computation used in
	Lemma~\ref{lem:toric-dimensional-scaling} gives, for a monomial containing
	\(j\) factors pulled back from the abelian base, an \(n\)-exponent
	\(r+j-g-1\le r-1\).  In the compressed case the independent calculation for
	\([m]\) gives the same bound for the \(m\)-exponent.  Thus
	\[
	\left|
	F_{\star,n}^r\cdot \ol M_{\star,n}^{d+1-r}
	\right|
	\le
	C_f L_n^d\alpha_{\star,n}^{r-1}.
	\]
	Multiplying by \(|\lambda|^r\), and using
	\(|\lambda|\alpha_{\star,n}\le1\), gives
	\[
	|\lambda|^r
	\left|
	F_{\star,n}^r\cdot \ol M_{\star,n}^{d+1-r}
	\right|
	\le
	C_f L_n^d|\lambda|^2\alpha_{\star,n}.
	\]
	The number of possible \(r\)'s is bounded by \(d\), so the constant may be
	chosen uniformly in \(r\).
	
	For a general integrable test metric, approximate \(f\) from above and below by
	model functions in the integrable topology.  Multilinearity and continuity of
	adelic intersections for integrable metrics allow passage to the limit, with
	the constant depending only on the chosen integrable norm of \(f\).
\end{proof}

\begin{definition}[Adelic perturbation size]
	Let \(\overline A\) be fixed and semipositive and let
	\(\overline E=O_Y(\phi)\) be integrable with trivial algebraic class.  For a
	model function set
	\[
	\|\phi\|_{\mathrm{ad}}=\sum_v n_v\|\phi_v\|_{\sup}
	\]
	and extend by completion.  Define
	\[
	\|\overline E\|_{\overline A}
	=\inf\max\{\|\phi^+\|_{\mathrm{ad}},\|\phi^-\|_{\mathrm{ad}}\},
	\]
	over decompositions
	\(\overline E=O_Y(\phi^+)-O_Y(\phi^-)\) such that
	\(\overline A+O_Y(\phi^\pm)\) are semipositive.  Only finite-size
	perturbations occur below.
\end{definition}

\begin{proposition}[K\"uhne--Yuan--Ikoma quadratic arithmetic volume package]
	\label{input:quadratic-volume}
	Let \(Y\) be a projective variety of dimension \(d\) over a number field, let
	\(\overline A\) be a horizontally semipositive adelic line bundle on \(Y\), and
	let \(\overline E\) be an integrable adelic line bundle whose underlying
	algebraic line bundle is trivial.  In the one-parameter range used below, and
	after writing the integrable perturbation as a difference of semipositive
	metrics, the arithmetic volume satisfies a quadratic comparison of the form
	\[
	\widevol(\overline A+\overline E)
	-(\overline A+\overline E)^{d+1}
	\ge
	-C_{\overline A}\|\overline E\|_{\overline A}^2,
	\]
	where \(\|\cdot\|_{\overline A}\) is the adelic size functional just defined.
	Thus for a perturbation \(\overline E=\lambda\overline F\) the
	defect is \(O(|\lambda|^2)\), with the constant controlled by the size of
	\(\overline F\) and by the chosen semipositive reference bundle.  More
	precisely, \(C_{\overline A}\) depends only on \(d\), the fixed numerical
	class and the chosen reference models.  In the families below, \(C_f\) may
	depend on \(G,X_\Sigma,D,N\), those models and the fixed integrable norm of
	\(f\), but is independent of \(n,m,\lambda\) and \(K_n\).
	
	The general bigness inequality is \cite[Lemma~2.6]{Ku1}.  Its proof uses
	\cite[Theorem 3.5.3 and Remark 3.5.4]{Iko13} together with Yuan's arithmetic
	bigness theorem \cite{Yua1}.  K\"uhne's semiabelian implementation is
	\cite[Lemma~4.6]{Ku1}: it applies Lemma~2.6 to a semipositive decomposition of
	the perturbation and then estimates the higher-order intersection terms.  The
	constant-shift and continuity facts used to remove the auxiliary correction
	\(\kappa_n\) are \cite[Lemma~2.5(a),(b)]{Ku1}.
\end{proposition}

\begin{remark}
	Proposition~\ref{input:quadratic-volume} is the external
	Hilbert--Samuel type source package used in the proof, and it is not claimed
	as a new theorem of this paper.  The lemmas below record exactly how it is used: they specify
	the perturbation \(\overline E\), the horizontal semipositivity correction, and
	the \(n\)- and \(m\)-dependence of the constant.  The dependence on \(n\) and
	\(m\) is not imported from the source package; it is verified here through
	Lemma~\ref{lem:toric-dimensional-scaling} and
	Lemma~\ref{lem:compressed-quadratic-scaling}.
\end{remark}

\paragraph{Raw versus normalized volume notation in the auxiliary tower.}
For an adelic line bundle over \(K_n\), write
\[
\widehat{\operatorname{vol}}_{\chi,K_n}^{\mathrm{raw}}(\overline M)
:=[K_n:\mathbb Q]\,
\widehat{\operatorname{vol}}_{\chi,K_n}(\overline M),
\qquad
(\overline M^{d+1})^{\mathrm{raw}}_{K_n}
:=[K_n:\mathbb Q]\,(\overline M^{d+1})_{K_n}.
\]
Every displayed defect carrying \([K_n:K]\) below is a raw defect and is
marked as such.  Unmarked heights, intersections, volumes and measures remain
normalized.

\begin{lemma}[Constant shifts in the volume defect]
	\label{lem:constant-shift-volume-defect}
	Let \(\ol B_{m,n,\lambda}\) be any of the adelic line bundles occurring below,
	namely \(\ol A_n+\varphi_n^*O_{\ol G}(\lambda f)\) or
	\(\ol A_{m,n}+\varphi_n^*O_{\ol G}(\lambda f_m)\), with
	\(|\lambda|\le n^{-1}\) in the canonical case and
	\(|\lambda|\le(nm)^{-1}\) in the compressed case.  If
	\(0\le \kappa_n\le Cn^{-2}\), then
	\[
	\begin{aligned}
		\bigl|&
		\widehat{\operatorname{vol}}_{\chi,K_n}^{\mathrm{raw}}
		(\ol B_{m,n,\lambda}\otimes\varphi_n^*O_{\ol G}(-\kappa_n))
		-
		\bigl((\ol B_{m,n,\lambda}\otimes
		\varphi_n^*O_{\ol G}(-\kappa_n))^{d+1}\bigr)^{\mathrm{raw}}_{K_n}  \\
		&\qquad
		-\widehat{\operatorname{vol}}_{\chi,K_n}^{\mathrm{raw}}(\ol B_{m,n,\lambda})
		+(\ol B_{m,n,\lambda}^{d+1})^{\mathrm{raw}}_{K_n}
		\bigr|
		\le C_f[K_n:K]\kappa_n ,
	\end{aligned}
	\]
	where \(C_f\) is independent of \(n\), of \(m\le n\), and of \(\lambda\).
\end{lemma}

\begin{proof}
	Tensoring by \(O_{\ol G}(-\kappa_n)\) changes only the metric on the trivial
	line bundle.  For the \(\chi\)-arithmetic volume this is a Lipschitz
	perturbation: multiplying all sup norms by \(e^{O(\kappa_n)}\) changes the
	raw asymptotic Euler characteristic by \(O([K_n:K]\kappa_n)\).  This is
	the uniform form of the constant-shift formula and continuity statement in
	\cite[Lemma 6(a),(b)]{Ku1}.  The same linear bound holds for the arithmetic
	self-intersection by multilinearity,
	because inserting one constant metrized trivial factor gives the algebraic
	degree, while terms with two or more such factors vanish.  The degrees of the
	line bundles in the canonical and compressed families are uniformly controlled
	by the fixed data \(D\), \(N\), and the chosen test metric, after the
	normalizations used in this subsection.  This gives the displayed uniform
	estimate.
\end{proof}

\begin{lemma}[Field-extension normalization in the auxiliary tower]
	\label{lem:field-extension-normalization}
	Fix a place \(v\) of \(K\).  In the auxiliary tower over \(K_n\), the test
	metric at \(v\) is understood as the collection of the induced metrics at all
	places \(v'\mid v\) of \(K_n\).  Write \(N_{v'}^{K_n}\) for raw local degrees
	and \(n_{v'}^{K_n}\) for normalized weights.  Then
	\[
	\sum_{v'\mid v}N_{v'}^{K_n}=[K_n:K]N_v^K,
	\qquad
	\sum_{v'\mid v}n_{v'}^{K_n}=n_v^K.
	\]
	Thus raw intersections and raw \(\chi\)-volumes over \(K_n\) carry a common
	factor \([K_n:K]\), whereas normalized versions do not.  Equivalently, after
	division by the global degree the linear local variation has coefficient
	\(n_v^K\), and a raw volume-defect estimate of the form
	\[
	\widehat{\operatorname{vol}}_{\chi,K_n}^{\mathrm{raw}}(\overline M_{\star,n,\lambda})
	-(\overline M_{\star,n,\lambda}^{d+1})^{\mathrm{raw}}_{K_n}
	\ge -C [K_n:K]E_{\star,n,\lambda}
	\]
	contributes \(O(E_{\star,n,\lambda})\) to the normalized Minkowski estimate
	before the usual division by \(|\lambda|\).
\end{lemma}

\begin{proof}
	The two identities are the raw and normalized forms of the standard
	decomposition formula.  Raw arithmetic intersections and raw
	\(\chi\)-volumes are additive over places, so installing the
	same Galois-invariant perturbation at all \(v'\mid v\) multiplies the
	raw contribution by the raw local-degree sum.  Global normalization cancels
	\([K_n:K]\), leaving coefficient \(n_v^K\).  The same cancellation applies to
	every quadratic or constant raw volume defect.  Hence no normalized height in
	the sequel changes under extension of the ground field.
\end{proof}

\begin{lemma}[Measure compression estimate]
	\label{lem:canonical-measure}
	Let
	\[
	\mu_{n,v'}=
	\frac{c_1(\ol L_{n,v'}^{\can})^{g+t}}{(L_n)^{g+t}}.
	\]
	Then for every continuous \(\Gal(\mathbb{C}_v/K_v)\)-invariant \(f\),
	\[
	\lim_{n\to\infty}
	\int_{\ol G_{n,\mathbb{C}_{v'}}^{an}} f_{n,v'}\,d\mu_{n,v'}
	=
	\int_{\ol G_{\mathbb{C}_v}^{an}}f\,d\mu_{\ol L,v}.
	\]
\end{lemma}

\begin{proof}
	Let \(d=g+t\).  Since \(L_n=\ol G_n(D)\otimes\pi_n^*N\), the measure
	\(\mu_{n,v'}\) is represented by
	\[
	c_1(\ol G_n(\ol D^{\can})_{v'})^t
	\wedge c_1(\pi_n^*\ol N_{v'})^g
	\]
	after division by the degree \(L_n^d\).  All other terms in the expansion of
	\(c_1(\ol L^{\can}_{n,v'})^d\) are zero as measures.  They are positive mixed
	measures because the two metrics are semipositive, and their total masses are
	the corresponding top-degree intersection numbers.  Those masses vanish by
	Lemma~\ref{lem:algebraic-dimensional-vanishing}.  This uses only its numerical
	statement, not a Chow-class identity.
	
	The abelian part is unchanged by \(\varphi_n\), while the toric part is pulled
	back by multiplication by \(n\) on \(T\).  The canonical toric metric is
	invariant under this multiplication after the normalization
	\([n]^*\ol D^{\can}=n\ol D^{\can}\).  Therefore the projection formula gives
	\[
	(\varphi_n)_*
	\left(
	\frac{
		c_1(\ol G_n(\ol D^{\can})_{v'})^t
		\wedge c_1(\pi_n^*\ol N_{v'})^g
	}{L_n^d}
	\right)
	\longrightarrow
	\frac{
		c_1(\ol G(\ol D^{\can})_v)^t
		\wedge c_1(\pi^*\ol N_v)^g
	}{L^d}.
	\]
	The convergence is weak convergence of Chambert-Loir measures.  Testing against
	the continuous function \(f\) and using \(f_{n,v'}=f\circ\varphi_n^{an}\) gives
	the asserted limit.
\end{proof}

\begin{lemma}[First variation of height]
	\label{lem:canonical-height-variation}
	Assume \(O_{\ol G}(f)\) is an integrable adelic line bundle.  For
	\[
	\ol L_{n,\lambda}^{\can}
	=
	\ol L_n^{\can}+\varphi_n^*O_{\ol G}(\lambda f)
	\]
	and \(|\lambda|\le n^{-1}\), one has
	\[
	\left|
	h_{\ol L_{n,\lambda}^{\can}}(\ol G_n)
	-h_{\ol L_n^{\can}}(\ol G_n)
	-\lambda n_v
	\int_{\ol G_{n,\mathbb{C}_{v'}}^{an}}f_{n,v'}\,d\mu_{n,v'}
	\right|
	\le C_f|\lambda|^2 n.
	\]
\end{lemma}

\begin{proof}
	Put \(d=g+t\) and \(F=\varphi_n^*O_{\ol G}(f)\).  By the multilinearity of
	adelic intersections,
	\[
	h_{\ol L_n^{\can}+\lambda F}(\ol G_n)
	-h_{\ol L_n^{\can}}(\ol G_n)
	=
	\frac{1}{(d+1)L_n^d}
	\sum_{r=1}^{d+1}
	\binom{d+1}{r}\lambda^r
	\bigl(F^r\cdot(\ol L_n^{\can})^{d+1-r}\bigr).
	\]
	The term with \(r=1\) equals
	\[
	\lambda n_v
	\int_{\ol G_{n,\mathbb{C}_{v'}}^{an}}
	f_{n,v'}\,d\mu_{n,v'}.
	\]
	This is the local intersection formula defining the Chambert-Loir measure,
	with the field-extension normalization explained in
	Lemma~\ref{lem:field-extension-normalization}; in the semiabelian asymptotic
	argument it is the linear term isolated in \cite[Lemma~4.4]{Ku1}.
	
	It remains to bound the terms with \(r\ge2\).  Since \(O(f)\) has trivial
	underlying algebraic line bundle, every such term is controlled by an
	intersection involving at least two metrically non-trivial but algebraically
	trivial factors.  Approximate \(f\) by model functions from above and below.
	For model functions the estimate follows from the arithmetic Hodge index and
	the projection formula; the general integrable case follows by passage to the
	limit.  The required scaling is the bridge from \cite[Lemma~4.4]{Ku1} to the
	present toric compactification, and is supplied by
	Lemma~\ref{lem:toric-dimensional-scaling}.  Hence, after normalization by
	\(L_n^d\),
	\[
	\left|
	\sum_{r=2}^{d+1}
	\binom{d+1}{r}\lambda^r
	\bigl(F^r\cdot(\ol L_n^{\can})^{d+1-r}\bigr)
	\right|
	\le
	C_f L_n^d |\lambda|^2 n.
	\]
	Dividing by \((d+1)L_n^d\) gives the stated bound.
\end{proof}

\begin{lemma}[Arithmetic volume comparison]
	\label{lem:canonical-volume}
	Assume that
	\[
	\ol L_n^{\can}(\kappa_n)
	=
	\ol L_n^{\can}\otimes\varphi_n^*O_{\ol G}(\kappa_n)
	\]
	is horizontally semipositive, with \(\kappa_n\ll n^{-2}\).  Then for
	\(|\lambda|\le n^{-1}\),
	\[
	\widehat{\operatorname{vol}}_{\chi,K_n}^{\mathrm{raw}}(\ol L_{n,\lambda}^{\can})
	-\bigl((\ol L_{n,\lambda}^{\can})^{g+t+1}\bigr)^{\mathrm{raw}}_{K_n}
	\ge
	-C_f [K_n:K]\,
	\bigl(n|\lambda|^2(1+\kappa_n n)+\kappa_n\bigr).
	\]
\end{lemma}

\begin{proof}
	Let
	\[
	\ol A_n=\ol L_n^{\can}(\kappa_n)
	\]
	and write
	\[
	\ol L_{n,\lambda}^{\can}
	=
	\bigl(\ol A_n+\varphi_n^*O_{\ol G}(\lambda f)\bigr)
	\otimes\varphi_n^*O_{\ol G}(-\kappa_n).
	\]
	By assumption \(\ol A_n\) is horizontally semipositive.  The arithmetic volume
	comparison used here is Proposition~\ref{input:quadratic-volume}, namely the volume
	part of \cite[Lemma~4.6]{Ku1} together with its source
	\cite[Lemma~2.6]{Ku1}.  Applying it to \(\ol A_n\) and to
	\(\varphi_n^*O_{\ol G}(\lambda f)\)
	gives a lower bound for
	\[
	\widevol(\ol A_n+\ol E)-(\ol A_n+\ol E)^{d+1}
	\]
	in terms of the square of the size of the integrable perturbation \(\ol E\).
	Here \(\ol E=\varphi_n^*O_{\ol G}(\lambda f)\).
	
	The contribution linear in \(\lambda f\) is already accounted for by the
	intersection term.  The remaining error is quadratic in \(\lambda f\).  Thus
	the only external volume theorem used at this point is the
	K\"uhne--Yuan--Ikoma package recorded in
	Proposition~\ref{input:quadratic-volume}; the passage from K\"uhne's
	\((\mathbb P^1)^t\)-compactification to the present \(X_\Sigma\) is supplied
	by Lemma~\ref{lem:toric-dimensional-scaling}.  In particular, the pullback by
	the auxiliary \(n\)-isogeny contributes at most one \(n\)-factor in the
	non-linear error terms, while the \(\kappa_n\)-correction changes the bound by
	the factor \(1+\kappa_n n\), because the semipositive reference bundle is
	\(\ol A_n=\ol L_n^{\can}(\kappa_n)\).  Keeping the field-extension factor
	of Lemma~\ref{lem:field-extension-normalization} before normalization, one
	obtains
	\[
	\widehat{\operatorname{vol}}_{\chi,K_n}^{\mathrm{raw}}
	(\ol A_n+\varphi_n^*O_{\ol G}(\lambda f))
	-\bigl((\ol A_n+\varphi_n^*O_{\ol G}(\lambda f))^{d+1}
	\bigr)^{\mathrm{raw}}_{K_n}
	\ge
	-C_f [K_n:K]\, n|\lambda|^2(1+\kappa_n n).
	\]
	Finally
	\[
	\ol L_{n,\lambda}^{\can}
	=
	\bigl(\ol A_n+\varphi_n^*O_{\ol G}(\lambda f)\bigr)
	\otimes\varphi_n^*O_{\ol G}(-\kappa_n),
	\]
	and Lemma~\ref{lem:constant-shift-volume-defect} changes the volume defect by
	at most \(C_f[K_n:K]\kappa_n\).  Combining the two estimates gives
	\[
	\widehat{\operatorname{vol}}_{\chi,K_n}^{\mathrm{raw}}(\ol L_{n,\lambda}^{\can})
	-\bigl((\ol L_{n,\lambda}^{\can})^{d+1}\bigr)^{\mathrm{raw}}_{K_n}
	\ge
	-C_f [K_n:K]\,
	\bigl(n|\lambda|^2(1+\kappa_n n)+\kappa_n\bigr).
	\]
	The constant is independent of \(n\), \(\lambda\), and the place \(v'\) above
	\(v\).
\end{proof}

\subsection{Canonical equidistribution}

\begin{theorem}[Canonical semiabelian equidistribution on \(X_\Sigma\)]
	\label{thm:canonical-eq}
	Assume \(D\) is big, nef and \(T\)-effective, and all local toric metrics of
	\(\ol D\) are canonical.  Then \(\ol L^{\can}\) has \(v\)-adic
	equidistribution at every place \(v\), with measure
	\[
	\mu_{\ol L,v}
	=
	\frac{
		c_1(\ol G(\ol D^{\can})_v)^t\wedge c_1(\ol\pi^*\ol N_v)^g
	}{
		G(D)^t(\pi^*N)^g
	}.
	\]
\end{theorem}

\begin{proof}
	By approximation in the nef cone it is enough to treat \(D\) ample.  Fix a
	continuous test function \(f\), approximated by integrable test functions, and
	fix \(v'\mid v\).  Minkowski's theorem applied to
	\(\ol L_{n,\lambda}^{\can}\) gives a non-zero section whose \(v'\)-norm is
	bounded by the normalized arithmetic volume.  Using
	Lemma~\ref{lem:canonical-volume}, then Lemma~\ref{lem:canonical-height-variation}
	and Proposition~\ref{prop:canonical-lower-bound}, one obtains
	\[
	\begin{aligned}
		\limsup_i\left|
		\frac{1}{\#O_v(x_i)}
		\sum_{y\in O_v(x_i)}f(y)
		-
		\int_{\ol G_{n,\mathbb{C}_{v'}}^{an}} f_{n,v'}\,d\mu_{n,v'}
		\right|
		&\le
		C\left(n^{-2}|\lambda|^{-1}+|\lambda|n(1+\kappa_n n)\right)\\
		&\quad +C|\lambda|^{-1}\epsilon .
	\end{aligned}
	\]
	Here Lemma~\ref{lem:field-extension-normalization} is used to pass from the
	\([K_n:K]\)-weighted volume defect to the displayed normalized estimate.
	The additional term \(\kappa_n|\lambda|^{-1}\) coming from
	Lemma~\ref{lem:canonical-volume} is absorbed into
	\(n^{-2}|\lambda|^{-1}\), since \(\kappa_n\ll n^{-2}\).
	Choose \(\lambda=n^{-3/2}\) and \(\epsilon=n^{-2}\).  Since
	\(\kappa_n\ll n^{-2}\), the right side is \(O(n^{-1/2})\).  Letting
	\(n\to\infty\) and using Lemma~\ref{lem:canonical-measure} gives the result.
\end{proof}

\section{Explicit compression for quasi-canonical metrics}

\subsection{Compressed metrics}

Let \(\ol D\) now be a quasi-canonical toric metrized divisor.  For \(m\ge 1\)
put
\[
\ol D_m=\frac{1}{m}[m]^*\ol D,
\]
where \([m]\) denotes multiplication on the toric part.  Define
\[
\ol L_{(m,n)}
=
\ol G_n(\ol D_m)\otimes\ol\pi_n^*\ol N_{K_n}.
\]
The canonical endpoint is denoted
\[
\ol L_{(\infty,n)}=\ol L_n^{\can}.
\]
Throughout the compressed section, \(L_n\) denotes the common underlying
algebraic line bundle of the adelic line bundles \(\ol L_{(m,n)}\).  Indeed,
the toric divisor underlying \(\ol D_m=m^{-1}[m]^*\ol D\) is again \(D\), since
the support function of a toric Cartier divisor is homogeneous.  Thus
\[
L_{(m,n)}
=G_n(D)\otimes\pi_n^*N_{K_n}
=L_n
\]
as algebraic line bundles, although their metrics depend on \(m\).  In
particular the normalizing degree in all height, measure, and intersection
estimates below is \(L_n^d\), \(d=g+t\), uniformly in \(m\).

\begin{proposition}[Equivariance in the horizontal compression parameter]
	\label{prop:horizontal-equivariance}
	Fix \(n\in\mathbb{N}\cup\{\infty\}\).  For any two positive integers \(m,m'\),
	the \(v\)-adic equidistribution property for \(\ol L_{(m,n)}\) is equivalent to
	the corresponding property for \(\ol L_{(m',n)}\).
\end{proposition}

\begin{proof}
	It suffices to compare \(m\) with \(1\).  The morphism
	\([m]_{\ol G_n}\) is generically finite and maps generic sequences to generic
	sequences.  The height identity is
	\[
	h_{\ol L_{(m,n)}}(x)
	=
	\frac{1}{m}h_{\ol G_n(\ol D)}([m]x)
	+h_{\ol\pi_n^*\ol N}(x),
	\]
	with the abelian part transforming quadratically under multiplication.  For
	small sequences, the abelian canonical height tends to zero, so smallness is
	preserved in both directions after passing to preimages.  The measure identity
	follows from the projection formula:
	\[
	\int f\,d\mu_{\ol L_{(1,n)},v}
	=
	\int f\circ[m]_{\ol G_n}\,d\mu_{\ol L_{(m,n)},v}.
	\]
\end{proof}

\subsection{Lower bound along a compression path}

\begin{proposition}[Quasi-canonical compression normal form]
	\label{prop:qcanonical-compression-normal-form}
	\label{input:qcanonical-compression-normal-form}
	After changing \(\ol D\) by an \(M_K\)-constant and normalizing
	\(\muabs_{\ol D}(T)=0\), a quasi-canonical toric metric has critical vector
	\((u_v)_v\) and constants \((\gamma_v)_v\) with
	\[
	\sum_vn_vu_v=0,\qquad \sum_vn_v\gamma_v=0.
	\]
	For \(m\ge1\), the compressed metric
	\(\ol D_m=m^{-1}[m]^*\ol D\) has local support functions
	\[
	\Psi_{\ol D_m,v}(u)
	=
	\Psi_D\left(u-\frac{u_v}{m}\right)-\frac{\gamma_v}{m}.
	\]
	Moreover, on each fixed toric chart used in
	Lemma~\ref{lem:canonical-chart-lower}, the shift \(u_v/m\) changes the
	theta-factor on the abelian quotient by \(m^{-1}\) times a Picard-zero class
	depending only on the original quasi-canonical metric and on the chart; the
	constant terms \(\gamma_v/m\) cancel globally by the product formula.
\end{proposition}

\begin{proof}
	By the BGPS normal form for quasi-canonical toric metrics, after adding an
	\(M_K\)-constant the local support function at \(v\) can be written as
	\[
	\Psi_{\ol D,v}(u)=\Psi_D(u-u_v)-\gamma_v .
	\]
	The normalization \(\muabs_{\ol D}(T)=0\), together with the product-formula
	condition for adelic toric translations, gives
	\[
	\sum_v n_vu_v=0,\qquad \sum_v n_v\gamma_v=0.
	\]
	The underlying toric Cartier divisor is unchanged by the normalized pullback:
	the support function of \(D\) is homogeneous for toric multiplication, so
	\[
	\frac1m[m]^*D=D
	\]
	as an \(\mathbb R\)-Cartier divisor.  On support functions, multiplication by
	\(m\) sends \(u\) to \(mu\).  Hence
	\[
	\Psi_{\ol D_m,v}(u)
	=
	\frac1m\Psi_{\ol D,v}(mu)
	=
	\frac1m\Psi_D(mu-u_v)-\frac{\gamma_v}{m}
	=
	\Psi_D\left(u-\frac{u_v}{m}\right)-\frac{\gamma_v}{m},
	\]
	where the last equality uses the homogeneity of the canonical support function
	\(\Psi_D\).  The constants have no global height contribution because their
	weighted sum is \(m^{-1}\sum_vn_v\gamma_v=0\).
	
	It remains to record the effect on the semiabelian local trivialization.  On a
	fixed chart the toric coordinate attached to a character is written as a
	semiabelian rational character divided by the theta representative of the
	corresponding character-pushout Picard-zero bundle, as in
	Theorem~\ref{thm:local-trivialization} and
	Proposition~\ref{prop:theta-neron-height}.  Translating the tropical coordinate
	by \(u_v/m\) therefore changes only the automorphy factor in this
	character-pushout direction.  Since the translation is linear in \(u_v/m\), the
	resulting class in \(\operatorname{Pic}^0(A)_\mathbb R\) is \(m^{-1}\) times
	the Picard-zero class attached to the original quasi-canonical translation on
	that chart.  The finite chart system gives finitely many such classes, all
	depending only on \(\ol D\) and on the chosen local trivialization data.  This
	is the asserted theta-factor statement.
\end{proof}

\begin{lemma}[Compressed chart lower bound]
	\label{lem:compressed-chart-lower}
	Assume that \(D\) is an integral ample \(T\)-effective toric divisor and that
	\(\ol D\) is quasi-canonical, normalized by
	\(\muabs_{\ol D}(T)=0\).  Let \(m_0,\ldots,m_r\) be the lattice points chosen in
	Lemma~\ref{lem:canonical-chart-lower}.  There are real Picard-zero classes
	\(E_0,\ldots,E_r\in \operatorname{Pic}^0(A)_\bR\), depending only on \(\ol D\) and on this
	finite chart system, such that for every \(m,n\ge1\) and every
	\(x\in\overline G_n(\overline K)\), one can choose \(j\) with
	\[
	h_{\overline G_n(\ol D_m)}(x)
	\ge
	-h_{\overline\Theta_{j,m}^{(n)}}(\pi_n(x)),
	\]
	where
	\[
	\Theta_{j,m}^{(n)}
	:=
	H_{m_j}^{(n)}+\frac1mE_j^{(n)}\in\operatorname{Pic}^0(A)_\bR.
	\]
	Here \(H_{m_j}^{(n)}\) is the Picard-zero class appearing in
	Lemma~\ref{lem:canonical-chart-lower}, and \(E_j^{(n)}\) denotes the same class
	\(E_j\) transported through the \(n\)-division extension \(G_n\).
\end{lemma}

\begin{proof}
	The finite chart system is the same as in
	Lemma~\ref{lem:canonical-chart-lower}, because \(D_m\) has the same underlying
	algebraic toric divisor \(D\).  On such a chart, the canonical proof compares
	the local toric section with the theta factor \(H_{m_j}^{(n)}\) on the abelian
	quotient.  By Proposition~\ref{prop:qcanonical-compression-normal-form}, replacing
	\(\Psi_D\) by
	\[
	\Psi_D\left(u-\frac{u_v}{m}\right)-\frac{\gamma_v}{m}
	\]
	does not introduce a scalar global height contribution, since
	\(\sum_vn_v\gamma_v=0\).  The remaining effect of the translated critical
	vector is linear in \(u_v/m\).  Through the local trivialization and the
	theta-function height identity of Proposition~\ref{prop:theta-neron-height},
	this linear term is exactly \(m^{-1}\) times a Picard-zero height on \(A\).
	Since only finitely many charts are used, these Picard-zero directions are a
	fixed finite family \(E_0,\ldots,E_r\).  Thus, on the chart where the chosen
	toric section does not vanish,
	\[
	h_{\overline G_n(\ol D_m)}(x)
	+h_{\overline\Theta_{j,m}^{(n)}}(\pi_n(x))\ge0,
	\]
	which is the asserted inequality.
\end{proof}

\begin{lemma}[Compressed Picard-zero translation bound]
	\label{lem:compressed-picard-bound}
	With the notation of Lemma~\ref{lem:compressed-chart-lower}, there is a
	constant \(C>0\), independent of \(m\), \(n\), \(j\), and
	\(y\in A(\overline K)\), such that
	\[
	h_{\overline N}(y)-h_{\overline\Theta_{j,m}^{(n)}}(y)\ge -Cn^{-2}.
	\]
\end{lemma}

\begin{proof}
	Fix \(j\).  Since \(\phi_N:A\to A^\vee\) is an isogeny, after clearing
	denominators in the finite-dimensional real vector space generated by the
	classes \(H_{m_j}\) and \(E_j\), there are a positive integer \(a\) and points
	\(q_j,e_j\in A(\overline K)\otimes_\BZ\bR\) such that
	\[
	\phi_N(q_j)=aH_{m_j},\qquad \phi_N(e_j)=aE_j.
	\]
	Choose \(q_{j,n}\) and \(e_{j,n}\) with
	\[
	nq_{j,n}=q_j,\qquad ne_{j,n}=e_j.
	\]
	Then
	\[
	\phi_N\left(q_{j,n}+\frac1m e_{j,n}\right)
	=
	a\Theta_{j,m}^{(n)}.
	\]
	By Corollary~\ref{cor:picard-zero-height-pairing}, we have
	\[
	h_{\overline\Theta_{j,m}^{(n)}}(y)
	=
	\frac2a
	\left\langle y,\,
	q_{j,n}+\frac1m e_{j,n}
	\right\rangle_N .
	\]
	Completing the square for the Neron--Tate quadratic form gives
	\[
	h_{\overline N}(y)-h_{\overline\Theta_{j,m}^{(n)}}(y)
	\ge
	-a^{-2}
	h_{\overline N}\left(q_{j,n}+\frac1m e_{j,n}\right).
	\]
	Since \(m\ge1\) and
	\[
	h_{\overline N}\left(q_{j,n}+\frac1m e_{j,n}\right)
	=
	n^{-2}h_{\overline N}\left(q_j+\frac1m e_j\right)
	\le C_jn^{-2},
	\]
	and since \(j\) ranges over a finite set, the result follows after enlarging
	\(C\).
\end{proof}

\begin{lemma}[Uniformity in the compression parameter]
	\label{lem:compressed-lower-uniformity}
	In Lemmas~\ref{lem:compressed-chart-lower} and
	\ref{lem:compressed-picard-bound}, all constants can be chosen independently of
	the compression parameter \(m\ge1\).  Consequently the lower-bound constants
	used below are uniform for every path \(m=m(n)\).
\end{lemma}

\begin{proof}
	The chart system is fixed once \(D\) is fixed.  The only new terms introduced by
	the quasi-canonical metric are the finitely many Picard-zero directions
	\(E_j\), and the constants \(\gamma_v/m\) have zero global contribution by the
	normalization in Proposition~\ref{prop:qcanonical-compression-normal-form}.  Thus
	the chart inequality itself has no constant which grows with \(m\).
	
	It remains to make explicit the uniformity in the Picard-zero estimate.  With
	the notation used in the proof of Lemma~\ref{lem:compressed-picard-bound}, the
	possible translating points are
	\[
	q_j+\frac1m e_j,\qquad m\ge1.
	\]
	They lie in the compact line segment
	\[
	\{q_j+\theta e_j:0\le\theta\le1\}
	\]
	inside the finite-dimensional real vector space generated by the chosen
	Picard-zero data.  Since the Neron--Tate height associated with \(N\) is a
	continuous quadratic form on this space, there is a constant \(C_j\) such that
	\[
	h_{\overline N}\left(q_j+\frac1m e_j\right)\le C_j
	\qquad(m\ge1).
	\]
	After passing to \(n\)-division points, this gives
	\[
	h_{\overline N}\left(q_{j,n}+\frac1m e_{j,n}\right)
	=n^{-2}h_{\overline N}\left(q_j+\frac1m e_j\right)
	\le C_jn^{-2}.
	\]
	Taking the maximum over the finite set of charts gives a constant independent
	of \(m\), \(n\), \(j\), and the point \(y\).  This is the asserted uniformity.
\end{proof}

\begin{proposition}[Lower bound for compressed metrics]
	\label{prop:compressed-lower-bound}
	Assume \(D\) is ample and \(T\)-effective, and normalize
	\(\mu_{\ol D}^{\mathrm{abs}}(T)=0\).  Let \(m(n)\) be any increasing integer
	function with \(m(n)\le n\).  Then there is \(C>0\), depending only on the fixed
	data and not on the function \(m(\cdot)\), such that
	\[
	\mu^{\mathrm{abs}}_{\ol L_{(m(n),n)}}(\ol G_n)\ge -Cn^{-2},
	\qquad
	-Cn^{-2}\le h_{\ol L_{(m(n),n)}}(\ol G_n)\le 0.
	\]
	Consequently there exists \(0\le\kappa_n\le Cn^{-2}\) such that
	\[
	\ol L_{(m(n),n)}\otimes\varphi_n^*O_{\ol G}(\kappa_n)
	\]
	is horizontally semipositive.
\end{proposition}

\begin{proof}
	As in Proposition~\ref{prop:canonical-lower-bound}, it is enough first to treat
	integral ample \(T\)-effective divisors; the general ample \(\bR\)-divisor case
	follows by linearity of support functions and heights.
	
	Let \(m=m(n)\).  For any \(x\in\overline G_n(\overline K)\), choose \(j\) by
	Lemma~\ref{lem:compressed-chart-lower}.  Combining that lemma with
	Lemmas~\ref{lem:compressed-picard-bound} and
	\ref{lem:compressed-lower-uniformity}, applied to \(y=\pi_n(x)\), gives
	\[
	h_{\ol L_{(m,n)}}(x)
	=h_{\ol G_n(\ol D_m)}(x)+h_{\ol N}(\pi_n(x))
	\ge -Cn^{-2}
	\]
	with \(C\) independent of \(m\le n\).  Hence
	\[
	\mu^{\mathrm{abs}}_{\ol L_{(m(n),n)}}(\ol G_n)\ge -Cn^{-2}.
	\]
	
	The quasi-canonical normalization and
	Proposition~\ref{prop:qcanonical-compression-normal-form} imply
	\(\muabs_{\ol D_m}(T)=0\) for every \(m\): compression moves the local critical
	vector from \(u_v\) to \(u_v/m\), while both the vector and constant parts have
	zero adelic sum.  Taking torsion points on the abelian quotient and toric
	points of toric height \(0\), and using
	Proposition~\ref{prop:section4-valuation-height-package}, gives
	\[
	\muess_{\ol L_{(m(n),n)}}(\ol G_n)\le0 .
	\]
	Theorem~\ref{input:fundamental-height-inequality} therefore yields
	\[
	-Cn^{-2}\le h_{\ol L_{(m(n),n)}}(\ol G_n)\le0.
	\]
	All constants in these two lower estimates are independent of the chosen
	compression path \(m(n)\); the restriction \(m(n)\le n\) is only needed later,
	when the moving test function \(f\circ[m(n)]\) is paired with the
	\(n\)-division compression.
	
	Finally apply Lemma~\ref{input:constant-correction-horizontal} with
	\(\delta_n=n^{-2}\).  The preceding uniform estimates verify the hypotheses of
	that input with a constant independent of \(m(n)\).  After enlarging \(C\), it
	gives \(0\le\kappa_n\le Cn^{-2}\) and the asserted horizontal semipositivity.
\end{proof}

\subsection{The three compressed estimates}

\begin{proposition}[Fixed algebraic data in the horizontal parameter]
	\label{lem:fixed-horizontal-algebraic-data}
	Put \(d=g+t\).  For every \(m\ge1\), the normalized pullback
	\(D_m=m^{-1}[m]^*D\) has underlying algebraic divisor \(D\).  Consequently
	\[
	M_m:=\overline G(D_m)=\overline G(D)
	\]
	as an algebraic line bundle, although its metric depends on \(m\).  Moreover,
	for \(i>g\),
	\[
	\deg\!\left(c_1(\pi^*N)^i\cap
	(c_1(M_m)^{d-i}\cap[\overline G])\right)=0,
	\]
	whereas \(M_m^t(\pi^*N)^g>0\).  These statements, and their intersection
	constants, are independent of \(m\).
\end{proposition}

\begin{proof}
	The first assertion follows from the homogeneity of the support function of
	\(D\).  The projection formula identifies the displayed degree with
	\[
	\deg\!\left(c_1(N|_A)^i\cap
	(\pi|_{\overline G})_*
	(c_1(M_m|_{\overline G})^{d-i}\cap[\overline G])\right),
	\]
	which vanishes for \(i>\dim A=g\).  Positivity of the term with \(i=g\)
	follows from relative ampleness of \(M_m\) and ampleness of \(N\), exactly as
	in the proof of \cite[Lemma~4.2]{Ku1}.
\end{proof}

\begin{lemma}[Exact horizontal pushforward at the target level]
	\label{lem:compressed-toric-uniform-family}
	For every positive integer \(m\),
	\[
	([m]_{\overline G}^{an})_*
	\mu_{\ol L_{(m,1)},v}=\mu_{\ol L,v}.
	\]
	Equivalently, for every continuous \(f\),
	\[
	\int_{\overline G_{\mathbb C_v}^{an}}
	f\circ[m]_{\overline G}^{an}\,d\mu_{\ol L_{(m,1)},v}
	=
	\int_{\overline G_{\mathbb C_v}^{an}}f\,d\mu_{\ol L,v}.
	\]
\end{lemma}

\begin{proof}
	This is the projection-formula identity already established in
	Proposition~\ref{prop:horizontal-equivariance}.  Notice that it is exact for
	each fixed \(m\); no limiting comparison between different values of \(m\) is
	involved.
\end{proof}

\begin{lemma}[Degree asymptotic with the horizontal parameter fixed]
	\label{lem:fixed-m-degree-asymptotic}
	Let \(m=m(n)\) be any positive integer-valued function and put
	\(\delta_n=\deg(\varphi_n)=n^t\).  Then
	\[
	\left|
	L_n^d-
	\delta_n n^{-d+g}\binom{d}{g}
	M_m^t(\pi^*N)^g
	\right|
	\ll_{G,D,N}
	\delta_n n^{-d+g-1}.
	\]
	The implied constant is independent of \(m\) and \(n\).
\end{lemma}

\begin{proof}
	This is the degree calculation in \cite[Lemmas~4.2 and~4.3]{Ku1}, applied for
	each \(n\) with the same horizontal index \(m=m(n)\) on source and target.
	The algebraic divisor underlying \(D_m\) is always \(D\), so the projection
	formula gives the same expansion and the same constants for every \(m\).
	Proposition~\ref{lem:fixed-horizontal-algebraic-data} makes all terms with more than
	\(g\) factors from \(\pi^*N\) vanish.  The term with exactly \(g\) such
	factors is the displayed leading term; all remaining terms have one additional
	power of \(n^{-1}\).
\end{proof}

\begin{lemma}[Integral asymptotic with the horizontal parameter fixed]
	\label{lem:fixed-m-integral-asymptotic}
	Assume that \(f\in C^0(\overline G_{\mathbb C_v}^{an})\) is
	\(\Gal(\mathbb C_v/K_v)\)-invariant.  Put
	\[
	f_m=f\circ[m]_{\overline G}^{an},\qquad
	f_{(m,n),v'}=f_m\circ\varphi_n^{an}.
	\]
	For any positive integer-valued function \(m=m(n)\),
	\[
	\begin{aligned}
		\biggl|&
		\int_{\overline G_{n,\mathbb C_{v'}}^{an}}
		f_{(m,n),v'},c_1(\ol L_{(m,n),v'})^d\\
		&-\delta_n n^{-d+g}\binom{d}{g}
		\int_{\overline G_{\mathbb C_v}^{an}}
		f_m\,c_1(\overline G(\ol D_m)_v)^t
		\wedge c_1(\pi^*\ol N_v)^g
		\biggr|
		\ll_{G,D,N,f}\delta_n n^{-d+g-1}.
	\end{aligned}
	\]
	The implied constant is independent of \(m\) and \(n\).
\end{lemma}

\begin{proof}
	Repeat the proof of \cite[Lemma~4.3]{Ku1}.  The projection formula and
	Proposition~\ref{lem:fixed-horizontal-algebraic-data} isolates the term with \(g\)
	abelian factors.  K\"uhne's local intersection bounds
	\cite[Lemma~2.2(a),(c)]{Ku1} then give the displayed estimate.  The bound is
	uniform in \(m\): the algebraic line bundle is fixed, all measures involved
	are positive, and \(\|f_m\|_{\sup}=\|f\|_{\sup}\).  Thus no comparison of
	local coordinates, no uniform-continuity argument, and no condition on
	\(m(n)/n\) enters the proof.
\end{proof}

\begin{proposition}[Normalized fixed-horizontal-index compression]
	\label{prop:fixed-m-normalized-compression}
	Under the hypotheses of Lemma~\ref{lem:fixed-m-integral-asymptotic}, for every
	positive integer-valued function \(m=m(n)\),
	\[
	\int_{\overline G_{n,\mathbb C_{v'}}^{an}}
	f_{(m,n),v'}\,d\mu_{(m,n),v'}
	-
	\int_{\overline G_{\mathbb C_v}^{an}}
	f_m\,d\mu_{\ol L_{(m,1)},v}
	\longrightarrow0.
	\]
\end{proposition}

\begin{proof}
	Divide the integral asymptotic of
	Lemma~\ref{lem:fixed-m-integral-asymptotic} by the degree asymptotic of
	Lemma~\ref{lem:fixed-m-degree-asymptotic}.  The common leading intersection
	number is positive by Proposition~\ref{lem:fixed-horizontal-algebraic-data}; hence the
	quotient error is \(O(n^{-1})\), uniformly in the chosen values of \(m(n)\).
\end{proof}

\begin{lemma}[Compressed measure estimate]
	\label{lem:compressed-measure}
	Let \(m(n)\) be any positive integer-valued function.  For
	\[
	f_{(m(n),n),v'}
	=
	f\circ [m(n)]_{\ol G}^{an}\circ\varphi_n^{an},
	\]
	one has
	\[
	\lim_{n\to\infty}
	\int_{\ol G_{n,\mathbb{C}_{v'}}^{an}}
	f_{(m(n),n),v'}\,d\mu_{(m(n),n),v'}
	=
	\int_{\ol G_{\mathbb{C}_v}^{an}}f\,d\mu_{\ol L,v}.
	\]
\end{lemma}

\begin{proof}
	Proposition~\ref{prop:fixed-m-normalized-compression} compares
	\((m(n),n)\) directly with \((m(n),1)\).  The toric metric on both sides has
	the same horizontal index \(m(n)\), so this step is exactly the compression
	path treated by K\"uhne's projection-formula argument.  Lemma
	~\ref{lem:compressed-toric-uniform-family} then identifies the target integral
	with \(\int f\,d\mu_{\ol L,v}\) for every \(n\).  Combining the two statements
	proves the limit, with no assumption that \(m(n)/n\to0\).
\end{proof}

\begin{remark}[Notation and the rank-one boundary calculation]
	\label{rem:compressed-measure-rank-one-check}
	Here \(v'\mid v\) is a place of the auxiliary field above \(v\), and the prime
	in \(f_{(m(n),n),v'}\) is essential.  The phrase ``well-defined'' refers to
	the independence of this pulled-back test function from the chosen
	identification \(\mathbb C_{v'}\simeq\mathbb C_v\), ensured by the assumed
	Galois invariance of \(f\); it does not concern the existence of
	\(\varphi_n\).
	
	The boundary case requested by the referee is completely explicit.  If
	\(A=0\), \(G=\mathbb G_m\), and \(X_\Sigma=\mathbb P^1\), then
	\(G_n=G\) and \(\varphi_n(z)=z^n\).  For the canonical toric metric,
	\[
	\frac1n\varphi_n^*\overline D=\overline D,
	\qquad
	(\varphi_n^{an})_*\mu_{\overline D,v}=\mu_{\overline D,v}.
	\]
	At an Archimedean place this is invariance of Haar probability measure on the
	unit circle under \(z\mapsto z^n\); at a non-Archimedean place it is the
	corresponding invariance of the canonical Gauss-point measure.  Hence the
	projection identity used above is exact in this canonical rank-one case.
	
	For a translated quasi-canonical metric one should distinguish the literal
	normalized pullback identity
	\[
	\frac1n\varphi_n^*\overline D_m=\overline D_{mn}.
	\]
	The proof of Lemma~\ref{lem:compressed-measure} does not replace
	\(\overline D_m\) by this pullback and does not compare \(\overline D_m\) with
	\(\overline D_{mn}\) by uniform continuity.  Instead, for each \(n\) it keeps
	the same horizontal index \(m=m(n)\) on the two levels
	\((m,n)\) and \((m,1)\), and applies the degree and local-intersection
	asymptotics of K\"uhne's Lemma~4.3.  This is why no hypothesis
	\(m(n)/n\to0\) is required.
\end{remark}

\begin{lemma}[Compressed quadratic intersection scaling]
	\label{lem:compressed-quadratic-scaling}
	Let \(O_{\ol G}(f)\) be an integrable test line bundle, put
	\[
	f_m=f\circ[m]_{\ol G},
	\qquad
	F_{m,n}=\varphi_n^*O_{\ol G}(f_m),
	\]
	and let \(d=g+t\).  There is a constant \(C_f>0\), independent of
	\(m\le n\), \(n\), and \(\lambda\), such that for
	\(|\lambda|\le (nm)^{-1}\),
	\[
	\left|
	\sum_{r=2}^{d+1}
	\binom{d+1}{r}\lambda^r
	\bigl(F_{m,n}^r\cdot\ol L_{(m,n)}^{d+1-r}\bigr)
	\right|
	\le
	C_f L_n^d |\lambda|^2nm .
	\]
\end{lemma}

\begin{proof}
	Apply Lemma~\ref{lem:parameterized-error-bookkeeping} in the compressed case,
	where
	\[
	\alpha_{\star,n}=nm,\qquad
	F_{\star,n}=F_{m,n}.
	\]
	For each \(2\le r\le d+1\), it gives
	\[
	|\lambda|^r
	\left|
	F_{m,n}^r\cdot\ol L_{(m,n)}^{d+1-r}
	\right|
	\le
	C_f L_n^d|\lambda|^2nm .
	\]
	The binomial coefficients and the finitely many possible values of \(r\) are
	absorbed into the constant \(C_f\).  Summing the displayed inequalities gives
	the result.
\end{proof}

\begin{lemma}[Compressed first variation]
	\label{lem:compressed-height-variation}
	Assume \(O_{\ol G}(f)\) is integrable.  Let \(m=m(n)\le n\), and put
	\[
	\ol L_{(m,n),\lambda}
	=
	\ol L_{(m,n)}+\varphi_n^*O_{\ol G}(\lambda f_m),
	\qquad
	f_m=f\circ[m]_{\ol G}.
	\]
	For \(|\lambda|\le (nm)^{-1}\),
	\[
	\left|
	h_{\ol L_{(m,n),\lambda}}(\ol G_n)
	-h_{\ol L_{(m,n)}}(\ol G_n)
	-\lambda n_v
	\int f_{(m,n),v'}\,d\mu_{(m,n),v'}
	\right|
	\le C_f|\lambda|^2nm.
	\]
\end{lemma}

\begin{proof}
	Put \(d=g+t\) and
	\[
	F_{m,n}=\varphi_n^*O_{\ol G}(f_m),
	\qquad f_m=f\circ[m]_{\ol G}.
	\]
	By multilinearity of adelic intersections,
	\[
	h_{\ol L_{(m,n)}+\lambda F_{m,n}}(\ol G_n)
	-h_{\ol L_{(m,n)}}(\ol G_n)
	=
	\frac{1}{(d+1)L_n^d}
	\sum_{r=1}^{d+1}
	\binom{d+1}{r}\lambda^r
	\bigl(F_{m,n}^r\cdot\ol L_{(m,n)}^{d+1-r}\bigr).
	\]
	The \(r=1\) term is the local variation formula for the Chambert-Loir measure,
	with the field-extension normalization of
	Lemma~\ref{lem:field-extension-normalization}, as in the source estimate
	\cite[Lemma~4.4]{Ku1}:
	\[
	\lambda n_v
	\int_{\ol G_{n,\mathbb{C}_{v'}}^{an}}
	f_{(m,n),v'}\,d\mu_{(m,n),v'}.
	\]
	
	The remaining terms are controlled by
	Lemma~\ref{lem:compressed-quadratic-scaling}.  Dividing the resulting bound by
	\((d+1)L_n^d\) gives the desired estimate.
\end{proof}

\begin{lemma}[Compressed volume comparison]
	\label{lem:compressed-volume}
	Assume that
	\[
	\ol L_{(m,n)}(\kappa_n)
	=\ol L_{(m,n)}\otimes\varphi_n^*O_{\ol G}(\kappa_n)
	\]
	is horizontally semipositive, with \(\kappa_n\ll n^{-2}\).  If
	\(|\lambda|\le (nm)^{-1}\), then
	\[
	\widehat{\operatorname{vol}}_{\chi,K_n}^{\mathrm{raw}}(\ol L_{(m,n),\lambda})
	-\bigl((\ol L_{(m,n),\lambda})^{g+t+1}\bigr)^{\mathrm{raw}}_{K_n}
	\ge
	-C_f [K_n:K]\,
	\bigl(nm|\lambda|^2(1+\kappa_n n)+\kappa_n\bigr).
	\]
\end{lemma}

\begin{proof}
	Let
	\[
	\ol A_{m,n}:=\ol L_{(m,n)}(\kappa_n).
	\]
	By assumption \(\ol A_{m,n}\) is horizontally semipositive.  We write
	\[
	\ol L_{(m,n),\lambda}
	=
	\bigl(\ol A_{m,n}+\varphi_n^*O_{\ol G}(\lambda f_m)\bigr)
	\otimes\varphi_n^*O_{\ol G}(-\kappa_n).
	\]
	Proposition~\ref{input:quadratic-volume}, equivalently the volume part of
	\cite[Lemma~4.6]{Ku1} after the toric-compactification bridge and its
	\cite[Lemma~2.6]{Ku1} bigness input, applies to the semipositive adelic line
	bundle \(\ol A_{m,n}\) and the integrable perturbation
	\(\varphi_n^*O_{\ol G}(\lambda f_m)\).  Its linear part is the
	intersection term
	associated with
	\(\ol A_{m,n}+\varphi_n^*O_{\ol G}(\lambda f_m)\); the remainder is controlled
	by the square of the perturbation.
	
	The \(nm\)-dependence is the only point not present in K\"uhne's published
	Lemma~4.6.  It comes from the compressed perturbation
	\(f_m=f\circ[m]_{\ol G}\), and the square of that perturbation is controlled by
	the compressed intersection scaling isolated in
	Lemma~\ref{lem:compressed-quadratic-scaling}.
	The horizontally semipositive reference bundle is
	\(\ol A_{m,n}=\ol L_{(m,n)}(\kappa_n)\), so the quadratic constant is changed by
	the factor \(1+\kappa_n n\).  Keeping the field-extension factor of
	Lemma~\ref{lem:field-extension-normalization} before normalization gives
	\[
	\begin{aligned}
		&\widehat{\operatorname{vol}}_{\chi,K_n}^{\mathrm{raw}}
		(\ol A_{m,n}+\varphi_n^*O_{\ol G}(\lambda f_m))
		-\bigl((\ol A_{m,n}+\varphi_n^*O_{\ol G}(\lambda f_m))^{d+1}
		\bigr)^{\mathrm{raw}}_{K_n} \\
		&\qquad\ge
		-C_f [K_n:K]\, nm|\lambda|^2(1+\kappa_n n).
	\end{aligned}
	\]
	The final passage from \(\ol A_{m,n}+\varphi_n^*O_{\ol G}(\lambda f_m)\) to
	\(\ol L_{(m,n),\lambda}\) is the constant shift
	\(\varphi_n^*O_{\ol G}(-\kappa_n)\).  By
	Lemma~\ref{lem:constant-shift-volume-defect}, this costs at most
	\(C_f[K_n:K]\kappa_n\).  Hence
	\[
	\widehat{\operatorname{vol}}_{\chi,K_n}^{\mathrm{raw}}(\ol L_{(m,n),\lambda})
	-\bigl((\ol L_{(m,n),\lambda})^{d+1}\bigr)^{\mathrm{raw}}_{K_n}
	\ge
	-C_f [K_n:K]\,
	\bigl(nm|\lambda|^2(1+\kappa_n n)+\kappa_n\bigr).
	\]
	The constant \(C_f\) is independent of \(n\), \(m\le n\), and \(\lambda\).
\end{proof}

\subsection{Correct parameter choice}

A useful explicit choice is
\[
m(n)=n^{1/2},\qquad \lambda=n^{-7/4},
\qquad \epsilon=|\lambda|n^{-1/2},
\]
where one may restrict to square \(n\)'s if desired.  This choice satisfies
\[
|\lambda|=n^{-7/4}\le n^{-3/2}=(nm(n))^{-1}.
\]
The raw quadratic error terms in Lemmas~\ref{lem:compressed-height-variation}
and~\ref{lem:compressed-volume} are of size
\[
nm(n)|\lambda|^2(1+\kappa_n n)+\kappa_n .
\]
In the Minkowski--variation argument this quantity is divided by
\(|\lambda|\).  Hence the normalized error to be checked is
\[
|\lambda|\,nm(n)(1+\kappa_n n)+\kappa_n|\lambda|^{-1}.
\]
By Proposition~\ref{prop:compressed-lower-bound} and
Lemma~\ref{lem:compressed-lower-uniformity}, the estimate
\(\kappa_n\ll n^{-2}\) is uniform in the chosen compression path.  Therefore
the second summand is dominated by
\(n^{-2}|\lambda|^{-1}\).  The computation gives
\[
n^{-2}|\lambda|^{-1}
=n^{-2}n^{7/4}=n^{-1/4},
\]
\[
|\lambda|\,n\,m(n)(1+\kappa_n n)
\ll
n^{-7/4}n n^{1/2}
=n^{-1/4},
\]
and
\[
|\lambda|^{-1}\epsilon=n^{-1/2}.
\]
Thus the total error is \(O(n^{-1/4})\), not \(O(n^{-1/2})\).  This is enough.

More generally, choose
\[
m(n)=\lfloor n^a\rfloor,\qquad 0<a<1,
\]
and
\[
\lambda=n^{-(3+a)/2}.
\]
Then \(|\lambda|\le (nm(n))^{-1}\) for large \(n\), and
\[
n^{-2}|\lambda|^{-1}
=
O\bigl(n^{-(1-a)/2}\bigr),
\]
\[
|\lambda|nm(n)(1+\kappa_n n)
=
O\bigl(n^{-(1-a)/2}\bigr).
\]
The remaining volume-shift term satisfies the same bound:
\[
\kappa_n|\lambda|^{-1}
\ll
n^{-2}|\lambda|^{-1}
=
O\bigl(n^{-(1-a)/2}\bigr).
\]
Taking \(\epsilon=|\lambda|n^{-(1-a)/2}\), the final error is
\[
O\bigl(n^{-(1-a)/2}\bigr).
\]
The choice \(a=1/2\) gives \(O(n^{-1/4})\).  One may restrict to \(n\) a square
or fourth power to avoid floor notation.

\section{Quasi-canonical equidistribution by explicit compression}

\begin{theorem}[Quasi-canonical compression theorem]
	\label{thm:qcanonical-compression}
	Assume that \(D\) is big, nef and \(T\)-effective, and that \(\ol D\) is
	quasi-canonical.  Then every generic \(\ol L\)-small sequence in
	\(G(\overline K)\), for
	\[
	\ol L=\ol G(\ol D)\otimes\ol\pi^*\ol N,
	\]
	has \(v\)-adic equidistribution at every place \(v\), with measure
	\[
	\mu_{\ol L,v}
	=
	\frac{
		c_1(\ol G(\ol D)_v)^t\wedge c_1(\ol\pi^*\ol N_v)^g
	}{
		G(D)^t(\pi^*N)^g
	}.
	\]
\end{theorem}

\begin{proof}
	By the birational reduction for nef and big toric divisors, it is enough to
	treat \(D\) ample.  Normalize the metric so that
	\(\mu_{\ol D}^{\mathrm{abs}}(T)=0\).  Fix a generic \(\ol L\)-small sequence
	\((x_i)\subset G(\ol K)\), a place \(v\), and a continuous
	\(\Gal(\mathbb{C}_v/K_v)\)-invariant test function \(f\), approximated by
	integrable test functions.
	
	Let \(m(n)=\lfloor n^a\rfloor\), \(0<a<1\), and choose
	\(\lambda=n^{-(3+a)/2}\).  By Minkowski's theorem, applied to
	\(\ol L_{(m(n),n),\lambda}\), there is an auxiliary non-zero section whose
	divisor avoids the inverse image of \(x_i\) for all large \(i\).  Combining
	Lemma~\ref{lem:compressed-volume}, Lemma~\ref{lem:compressed-height-variation},
	and Proposition~\ref{prop:compressed-lower-bound}, and then averaging over the
	finite fiber of \([m(n)]\circ\varphi_n\), yields
	\[
	\begin{aligned}
		\limsup_i
		\biggl|
		&\frac{1}{\#O_v(x_i)}\sum_{y\in O_v(x_i)} f(y)
		-
		\int_{\ol G_{n,\mathbb{C}_{v'}}^{an}}
		f_{(m(n),n),v'}\,d\mu_{(m(n),n),v'}
		\biggr| \\
		&\le
		C\left(
		n^{-2}|\lambda|^{-1}
		+|\lambda|nm(n)(1+\kappa_n n)
		\right)
		+C|\lambda|^{-1}\epsilon.
	\end{aligned}
	\]
	By Lemma~\ref{lem:field-extension-normalization}, the factor \([K_n:K]\) in
	this defect is exactly the factor which disappears in the normalized
	Minkowski--variation estimate.
	Here the term \(|\lambda|nm(n)(1+\kappa_n n)\) is the normalized form of the
	quadratic error \(nm(n)|\lambda|^2(1+\kappa_n n)\) from
	Lemmas~\ref{lem:compressed-height-variation} and
	\ref{lem:compressed-volume}; the normalization comes from the usual division by
	\(|\lambda|\) in the variation argument.
	The additional term \(\kappa_n|\lambda|^{-1}\) coming from
	Lemma~\ref{lem:compressed-volume} is absorbed into
	\(n^{-2}|\lambda|^{-1}\), since
	Proposition~\ref{prop:compressed-lower-bound} gives
	\(\kappa_n\ll n^{-2}\) uniformly for the chosen compression path.
	With the parameter choice above, the right side is
	\[
	O\bigl(n^{-(1-a)/2}\bigr).
	\]
	Letting \(n\to\infty\) and using Lemma~\ref{lem:compressed-measure} gives
	\[
	\lim_i
	\frac{1}{\#O_v(x_i)}
	\sum_{y\in O_v(x_i)} f(y)
	=
	\int_{\ol G_{\mathbb{C}_v}^{an}}f\,d\mu_{\ol L,v}.
	\]
	This proves equidistribution.
\end{proof}

\begin{remark}[Generic versus strict]\label{rem:qcanonical-generic-not-strict}
	Theorem~\ref{thm:qcanonical-compression} proves the generic equidistribution
	property.  It is not, by itself, a strong equidistribution theorem for strict
	sequences.  The standard passage from generic equidistribution to strong
	equidistribution uses a Bogomolov theorem: if a strict small sequence is not
	generic, one extracts a generic small sequence on a proper subvariety and then
	uses Bogomolov to force that subvariety into a proper algebraic subgroup.  This
	is the argument used by K\"uhne after proving the Bogomolov theorem for
	semiabelian varieties.  In the present paper the strict restricted statement
	is itself used in the proof of Bogomolov, so this upgrade cannot be imported
	from the present Bogomolov theorem without circularity.
	
	Recent general approximation results, such as those of Ballay--Sombra
	\cite{BS24}, give powerful generic equidistribution statements and recover the
	generic semiabelian equidistribution theorem.  They do not remove the logical
	distinction above.  Thus any use of strict quasi-canonical equidistribution
	below must either be cited as an independent source theorem or proved by a
	separate non-circular argument.
\end{remark}

\section{Monocritical case}

\begin{theorem}[Monocritical equidistribution]
	\label{thm:monocritical-eq}
	Suppose \(D\) is big and \(T\)-effective, and
	\(\bD=(D,\{\Vert\cdot\Vert_v\})\) is a semipositive, monocritical and
	arithmetically \(T\)-effective toric metrized \(\BR\)-divisor in the sense of
	Definitions~\ref{def 3.12} and~\ref{def:arith-T-effective}.  Then
	\[
	\bL=\overline G(\bD)\otimes\overline\pi^*\bN
	\]
	satisfies \(v\)-adic equidistribution at every place \(v\) of \(K\).  In the
	quasi-canonical case the limiting measure is
	\[
	\mu_{\bL,v}
	=
	\frac{
		\hat c_1(\overline G(\bD)_v)^t\wedge
		\hat c_1(\overline\pi^*\bN_v)^g
	}{
		\overline G(D)^t(\overline\pi^*N)^g
	}.
	\]
\end{theorem}

\begin{remark}[Source ledger for the monocritical reduction]
	\label{rem:monocritical-source-ledger}
	The reduction in this section uses the toric measure theory of
	Burgos--Philippon--Rivera-Letelier--Sombra \cite{BGPS19}, especially their
	Kantorovich--Rubinstein topology, the functions \(\Phi_v\), and the rigidity of
	monocritical metrics.  The new point here is not to reprove that toric package,
	but to connect it with the semiabelian height decomposition supplied by
	Section~4 and with the quasi-canonical compression theorem
	\ref{thm:qcanonical-compression}.
\end{remark}

\begin{lemma}[Arithmetic \(T\)-effectivity fixes the origin]
	\label{lem:arith-teff-origin}
	If \(\bD\) is arithmetically \(T\)-effective, then
	\[
	\upsilon_{\bD}(0)=\muess_{\bD}(X_\Sigma).
	\]
	In particular \(0\in\Delta_{D,\max}\).
\end{lemma}

\begin{proof}
	Since \(D\) is \(T\)-effective, \(0\in\Delta_D\)
	\cite[Proposition 4.9]{BMPS22}.  By arithmetic \(T\)-effectivity, the
	distinguished invariant section is small with local constants \(l_v\), so
	\[
	\upsilon_{\bD,v}(0)
	=
	\inf_{u\in N_\BR}\bigl(-\Psi_{\bD,v}(u)\bigr)
	\ge l_v
	\]
	for every place \(v\).  Hence
	\[
	\upsilon_{\bD}(0)
	=\sum_v n_v\upsilon_{\bD,v}(0)
	\ge \sum_v n_vl_v
	=\muabs_{\bD}(T).
	\]
	By the toric successive-minima formula
	\cite[Theorem 3.9 and Corollary 3.10]{BGPS15},
	\[
	\muabs_{\bD}(T)=\muess_{\bD}(X_\Sigma)
	=\max_{x\in\Delta_D}\upsilon_{\bD}(x).
	\]
	The reverse inequality is automatic because \(0\in\Delta_D\).  Thus
	\(\upsilon_{\bD}(0)=\muess_{\bD}(X_\Sigma)\).
\end{proof}

In the rest of this section \(D\) is big and \(T\)-effective and \(\bD\) is
semipositive, monocritical and arithmetically \(T\)-effective.  Let
\[
\mathbf u=(u_v)_v\in\bigoplus_vN_\BR
\]
be the critical vector of \(\bD\).

\begin{definition}[Adelic valuation measures]
	\label{def:adelic-valuation-measures}
	Let \(\mathcal E\) be the space of Borel probability measures on \(N_\BR\) with
	finite first moment, endowed with the Kantorovich--Rubinstein topology.  An
	adelic measure with center \(c\in N_\BR\) is a collection
	\(\boldsymbol\nu=(\nu_v)_v\), with \(\nu_v\in\mathcal E\) and
	\(\nu_v=\delta_0\) for all but finitely many \(v\), such that
	\[
	\sum_v n_vE[\nu_v]=c.
	\]
	The set of such adelic measures is denoted by \(\mathcal H_K^c\).  For such
	\(\boldsymbol\nu\), put
	\[
	\eta_{\overline G(\bD)}(\boldsymbol\nu)
	=
	-\sum_v n_v\int_{N_\BR}\Psi_{\bD,v}\,d\nu_v.
	\]
\end{definition}

\begin{lemma}[Point measures and the semiabelian height]
	\label{lem:point-measure-height}
	For \(x\in G(\overline K)\), the local trivializations of
	Proposition~\ref{prop:section4-valuation-height-package} define an adelic
	valuation measure
	\[
	\boldsymbol\nu_x=(\nu_{x,v})_v,\qquad
	\nu_{x,v}:=\frac1{\#O_v(x)}
	\sum_{y\in O_v(x)}\delta_{\val_v(\psi_{j(y),v}(y))},
	\]
	where for each \(y\) one chooses any chart \(j(y)\) with
	\(\pi(y)\in U_{j(y),v}\).  This definition is independent of all such choices,
	and
	\[
	\sum_v n_vE[\nu_{x,v}]
	=
	\bigl(-h_{\overline H_1}(\pi(x)),\ldots,
	-h_{\overline H_t}(\pi(x))\bigr).
	\]
	Moreover
	\[
	\eta_{\overline G(\bD)}(\boldsymbol\nu_x)
	=h_{\overline G(\bD)}(x).
	\]
\end{lemma}

\begin{proof}
	This is Proposition~\ref{prop:section4-valuation-height-package}, rewritten in
	the notation of adelic valuation measures.  The center identity is the displayed
	center formula there, and the equality
	\(\eta_{\overline G(\bD)}(\boldsymbol\nu_x)=h_{\overline G(\bD)}(x)\) is exactly
	its height identity.
\end{proof}

For \(v\in\Sigma(K)\), set
\[
g_{1,v}=\upsilon_{\bD,v},
\qquad
g_{2,v}=\sum_{w\ne v}\frac{n_w}{n_v}\upsilon_{\bD,w}.
\]
For \(c\in N_\BR\) and \(\mu\in\mathcal E\), define
\[
\Phi_{v,c}(\mu)
=
\int g_{1,v}^{\vee}\,d\mu
+g_{2,v}^{\vee}\!\left(\frac{c}{n_v}-E[\mu]\right)
+\max_{x\in\Delta_D}(g_{1,v}+g_{2,v})(x),
\qquad
\Phi_v:=\Phi_{v,0}.
\]
If \(x\in\Delta_{D,\max}\), put
\[
B_v=\partial g_{1,v}(x)\cap(-\partial g_{2,v}(x)),
\]
and let \(F_v\) be the minimal face of \(\partial g_{1,v}(x)\) containing
\(B_v\).

\begin{proposition}[BGPS package and the general-center defect inequality]
	\label{input:bgps-kr}
	Let \(\boldsymbol\nu=(\nu_v)_v\in\mathcal H_K^c\).  Then
	\[
	\max_v\{-n_v\Phi_{v,c}(\nu_v)\}
	\le
	\eta_{\overline G(\bD)}(\boldsymbol\nu)
	-\muess_{\overline G(\bD)}(\overline G)
	\le
	\sum_v -n_v\Phi_{v,c}(\nu_v).
	\]
	No additional affine term occurs in the middle expression: arithmetic
	\(T\)-effectivity permits the choice \(x=0\in\Delta_{D,\max}\).  Moreover, if
	\[
	R_D:=\sup_{x\in\Delta_D}\|x\|,
	\]
	then
	\[
	\sup_{\mu\in\mathcal E}
	\bigl|\Phi_{v,c}(\mu)-\Phi_v(\mu)\bigr|
	\le \frac{R_D}{n_v}\|c\|.
	\]
	
	For the centered function \(\Phi_v=\Phi_{v,0}\), the following facts are the
	toric results of \cite[Propositions 3.9, 3.11, 3.14 and 4.16, Lemma 4.5,
	and Theorem 4.19]{BGPS19}: \(\Phi_v\le0\), and \(\Phi_v(\mu)=0\) if and only if
	\[
	\operatorname{supp}(\mu)\subseteq F_v,\qquad E[\mu]\in B_v.
	\]
	If \(\Phi_v(\mu_i)\to0\), then every weak cluster point of \((\mu_i)\) belongs
	to \(\mathcal E\) and satisfies these two conditions.  Finally, if \(\bD\) is
	monocritical with critical vector \(\mathbf u=(u_v)_v\), then \(F_v=\{u_v\}\);
	in this case BGPS Theorem 4.19 gives adelic KR convergence to
	\(\boldsymbol\delta_{\mathbf u}\) for toric small nets.
\end{proposition}

\begin{proof}
	Put
	\[
	\widehat\phi_{i,v}(z)=g_{i,v}^{\vee}(z)+g_{i,v}(0),
	\qquad i=1,2.
	\]
	Because \(0\in\Delta_{D,\max}\), Fenchel duality gives
	\(\widehat\phi_{i,v}\le0\), and
	\[
	\eta_{\overline G(\bD)}(\boldsymbol\nu)
	-\muess_{\overline G(\bD)}(\overline G)
	=-\sum_w n_w\int\widehat\phi_{1,w}\,d\nu_w.                 \tag{7.1}
	\]
	Also
	\[
	\Phi_{v,c}(\nu_v)
	=\int\widehat\phi_{1,v}\,d\nu_v
	+\widehat\phi_{2,v}\!\left(\frac{c}{n_v}-E[\nu_v]\right).
	\tag{7.2}
	\]
	The nonpositivity of the second term in (7.2) immediately gives the right-hand
	inequality.
	
	For the left-hand inequality, fix \(v\).  The sup-convolution and right-scaling
	identities for concave duals
	\cite[Propositions 2.3.1(1) and 2.3.3(3)]{BPS14} give
	\[
	\widehat\phi_{2,v}
	=\boxplus_{w\ne v}
	\left(\widehat\phi_{1,w}\frac{n_w}{n_v}\right).
	\]
	Since \(\boldsymbol\nu\) has center \(c\),
	\[
	\frac{c}{n_v}-E[\nu_v]
	=\sum_{w\ne v}\frac{n_w}{n_v}E[\nu_w].
	\]
	The definition of sup-convolution followed by Jensen's inequality therefore
	gives
	\[
	\widehat\phi_{2,v}\!\left(\frac{c}{n_v}-E[\nu_v]\right)
	\ge
	\sum_{w\ne v}\frac{n_w}{n_v}\widehat\phi_{1,w}(E[\nu_w])
	\ge
	\sum_{w\ne v}\frac{n_w}{n_v}
	\int\widehat\phi_{1,w}\,d\nu_w.
	\]
	Together with (7.1)--(7.2), this yields
	\[
	-n_v\Phi_{v,c}(\nu_v)
	\le
	\eta_{\overline G(\bD)}(\boldsymbol\nu)
	-\muess_{\overline G(\bD)}(\overline G).
	\]
	Taking the maximum over \(v\) proves the left-hand inequality.
	
	Finally, the concave dual \(g_{2,v}^{\vee}\) is \(R_D\)-Lipschitz because its
	primal domain is contained in \(\Delta_D\).  Translating its argument by
	\(c/n_v\) proves the displayed uniform estimate.  The remaining centered
	compactness, vanishing and rigidity assertions are exactly the cited BGPS
	results.
\end{proof}

\begin{lemma}[KR upgrade from BGPS concentration]
	\label{lem:bgps-kr-upgrade}
	Let \((\boldsymbol\nu_i)_i\), \(\boldsymbol\nu_i=(\nu_{i,v})_v\), be adelic
	valuation measures whose centers \(c_i=\sum_v n_vE[\nu_{i,v}]\) tend to \(0\).
	Assume
	\[
	\eta_{\overline G(\bD)}(\boldsymbol\nu_i)
	\longrightarrow
	\muess_{\overline G(\bD)}(\overline G)
	\]
	and \(\Phi_v(\nu_{i,v})\to0\) for every place \(v\).  If \(\bD\) is
	monocritical with critical vector \(\mathbf u=(u_v)_v\), then
	\[
	\boldsymbol\nu_i\longrightarrow\boldsymbol\delta_{\mathbf u}
	\]
	in the adelic Kantorovich--Rubinstein topology.
\end{lemma}

\begin{proof}
	The cluster statement in Proposition~\ref{input:bgps-kr}, together with
	monocriticality and BGPS Proposition 4.16, gives \(F_v=\{u_v\}\).  Hence every
	weak cluster point of \((\nu_{i,v})_i\) is \(\delta_{u_v}\), and the compactness
	part of BGPS Proposition 3.11 gives weak convergence
	\[
	\nu_{i,v}\longrightarrow\delta_{u_v}
	\]
	for every \(v\).
	
	It remains to upgrade weak convergence to KR convergence.  Put
	\[
	\eta_{\mathbf u}(\boldsymbol\nu)
	=-\sum_v n_v\int_{N_\BR}\Psi_D(z-u_v)\,d\nu_v(z).
	\]
	For all but finitely many places, \(u_v=0\) and
	\(\Psi_{\bD,v}=\Psi_D\).  Thus
	\[
	f_v(z)=\Psi_{\bD,v}(z)-\Psi_D(z-u_v)
	\]
	is a finite adelic family of bounded continuous functions.  By the weak
	convergence just proved,
	\[
	\eta_{\mathbf u}(\boldsymbol\nu_i)
	=\eta_{\overline G(\bD)}(\boldsymbol\nu_i)
	+\sum_v n_v\int f_v\,d\nu_{i,v}
	\longrightarrow
	\muess_{\overline G(\bD)}(\overline G)+\sum_vn_vf_v(u_v).
	\]
	The identity
	\[
	\muess_{\overline G(\bD)}(\overline G)+\sum_vn_vf_v(u_v)=0
	\]
	is precisely the critical-vector normalization used in the proof of
	\cite[Theorem 4.19]{BGPS19}.  Therefore
	\(\eta_{\mathbf u}(\boldsymbol\nu_i)\to0\).
	
	Choose \(x_0\in\operatorname{int}(\Delta_D)\).  Since \(\Delta_D\) is full
	dimensional, there is \(C>0\) such that
	\[
	\|z\|\le -C\bigl(\Psi_D(z)-\langle x_0,z\rangle\bigr),
	\qquad z\in N_\BR.
	\]
	Using the elementary bound
	\[
	W(\nu_{i,v},\delta_{u_v})
	\le \int\|z-u_v\|\,d\nu_{i,v}(z),
	\]
	we obtain
	\[
	\begin{aligned}
		W_K(\boldsymbol\nu_i,\boldsymbol\delta_{\mathbf u})
		&\le \sum_v n_v\int\|z-u_v\|\,d\nu_{i,v}(z)\\
		&\le C\,\eta_{\mathbf u}(\boldsymbol\nu_i)
		+C\,\bigl|\langle x_0,c_i-\sum_vn_vu_v\rangle\bigr|.
	\end{aligned}
	\]
	The critical vector is centered, \(\sum_v n_vu_v=0\), and \(c_i\to0\).
	Together with \(\eta_{\mathbf u}(\boldsymbol\nu_i)\to0\), this proves
	adelic KR convergence.
\end{proof}

\begin{proposition}[Defect vanishing for small sequences]
	\label{prop:bgps-defect-vanishing}
	Let \((x_i)\) be a \(\bL\)-small sequence in \(G(\overline K)\).  Then for every
	place \(v\),
	\[
	\Phi_v(\nu_{x_i,v})\longrightarrow0.
	\]
	The centers of the adelic measures \(\boldsymbol\nu_{x_i}\) tend to \(0\).
\end{proposition}

\begin{proof}
	By Proposition~\ref{prop:section4-minima-formula} and the height decomposition
	\[
	h_{\bL}(x_i)
	=h_{\overline G(\bD)}(x_i)+\hat h_{\bN}(\pi(x_i)),
	\]
	smallness for \(\bL\) implies
	\[
	h_{\overline G(\bD)}(x_i)\to\muess_{\bD}(X_\Sigma),
	\qquad
	\hat h_{\bN}(\pi(x_i))\to0.
	\]
	Since the canonical heights \(h_{\overline H_j}\) are quadratic forms controlled
	by \(\hat h_{\bN}\), we have
	\[
	h_{\overline H_j}(\pi(x_i))\to0,\qquad 1\le j\le t.
	\]
	By Lemma~\ref{lem:point-measure-height}, the centers therefore tend to \(0\),
	and
	\[
	\eta_{\overline G(\bD)}(\boldsymbol\nu_{x_i})
	=h_{\overline G(\bD)}(x_i)
	\to\muess_{\overline G(\bD)}(\overline G).
	\]
	Write
	\[
	c_i=\sum_w n_wE[\nu_{x_i,w}].
	\]
	The general-center defect inequality in Proposition~\ref{input:bgps-kr} gives,
	for each \(v\),
	\[
	0\le -n_v\Phi_{v,c_i}(\nu_{x_i,v})
	\le
	\eta_{\overline G(\bD)}(\boldsymbol\nu_{x_i})
	-\muess_{\overline G(\bD)}(\overline G).
	\]
	Thus \(\Phi_{v,c_i}(\nu_{x_i,v})\to0\).  Since \(c_i\to0\), the uniform
	estimate in Proposition~\ref{input:bgps-kr} yields
	\[
	\bigl|\Phi_{v,c_i}(\nu_{x_i,v})-\Phi_v(\nu_{x_i,v})\bigr|
	\le \frac{R_D}{n_v}\|c_i\|\longrightarrow0.
	\]
	Consequently \(\Phi_v(\nu_{x_i,v})\to0\), as required.
\end{proof}

\begin{proposition}[KR convergence to the critical vector]
	\label{prop:monocritical-kr}
	Let \((x_i)\) be a \(\bL\)-small sequence in \(G(\overline K)\).  Then
	\[
	\boldsymbol\nu_{x_i}\longrightarrow\boldsymbol\delta_{\mathbf u}
	\]
	in the adelic Kantorovich--Rubinstein topology.  Equivalently, for every place
	\(v\),
	\[
	\nu_{x_i,v}\longrightarrow\delta_{u_v}
	\]
	in the local Kantorovich--Rubinstein topology.
\end{proposition}

\begin{proof}
	By Proposition~\ref{prop:bgps-defect-vanishing}, the BGPS defects
	\(\Phi_v(\nu_{x_i,v})\) tend to \(0\), and the centers of
	\(\boldsymbol\nu_{x_i}\) tend to \(0\).  The equality
	\[
	\eta_{\overline G(\bD)}(\boldsymbol\nu_{x_i})
	=h_{\overline G(\bD)}(x_i)
	\]
	and the proof of Proposition~\ref{prop:bgps-defect-vanishing} show that this
	quantity tends to the essential minimum.  Lemma~\ref{lem:bgps-kr-upgrade}
	therefore gives adelic KR convergence to
	\(\boldsymbol\delta_{\mathbf u}\), and hence local KR convergence at every
	place.
\end{proof}

\begin{proposition}[Quasi-canonical replacement]
	\label{prop:monocritical-qcanonical-replacement}
	After changing \(\bD\) by an \(M_K\)-constant, assume
	\(\muabs_{\bD}(T)=0\).  Define a quasi-canonical toric metric \(\bD'\) by
	\[
	\Psi_{\bD',v}(z)=\Psi_D(z-u_v).
	\]
	Let
	\[
	\bL'=\overline G(\bD')\otimes\overline\pi^*\bN.
	\]
	Then \(\muabs_{\bD'}(T)=\muabs_{\bD}(T)=0\).  Moreover every \(\bL\)-small
	sequence in \(G(\overline K)\) is \(\bL'\)-small.
\end{proposition}

\begin{proof}
	The metric \(\bD'\) is quasi-canonical by construction: at the place \(v\) it
	is the canonical toric metric translated by the vector \(u_v\).  The product
	formula for the critical vector, \(\sum_vn_vu_v=0\), implies that the global
	absolute minimum is unchanged.
	
	Let \((x_i)\) be a \(\bL\)-small sequence.  By the proof of
	Proposition~\ref{prop:bgps-defect-vanishing},
	\[
	\hat h_{\bN}(\pi(x_i))\to0.
	\]
	By Proposition~\ref{prop:monocritical-kr},
	\[
	\boldsymbol\nu_{x_i}\to\boldsymbol\delta_{\mathbf u}
	\]
	in the adelic KR topology.  The functions \(\Psi_{\bD',v}\) are Lipschitz on
	\(N_\BR\), with constants controlled by the fixed polytope \(\Delta_D\).  Hence
	\[
	\eta_{\overline G(\bD')}(\boldsymbol\nu_{x_i})
	\to
	\eta_{\overline G(\bD')}(\boldsymbol\delta_{\mathbf u})
	=0
	=\muess_{\bD'}(X_\Sigma).
	\]
	Using Lemma~\ref{lem:point-measure-height} for \(\bD'\), this says
	\[
	h_{\overline G(\bD')}(x_i)\to0.
	\]
	Together with \(\hat h_{\bN}(\pi(x_i))\to0\), we obtain
	\[
	h_{\bL'}(x_i)\to\muess_{\bL'}(\overline G),
	\]
	so \((x_i)\) is \(\bL'\)-small.
\end{proof}

\begin{proof}[Proof of Theorem~\ref{thm:monocritical-eq}]
	Let \((x_i)\) be a generic \(\bL\)-small sequence.  By
	Proposition~\ref{prop:monocritical-qcanonical-replacement}, after the harmless
	\(M_K\)-constant normalization, the sequence is also small for the
	quasi-canonical adelic line bundle \(\bL'\).  The quasi-canonical compression
	theorem, Theorem~\ref{thm:qcanonical-compression}, applies to \(\bL'\), hence
	the Galois orbit measures of \((x_i)\) equidistribute for \(\bL'\).
	
	Proposition~\ref{prop:monocritical-kr} identifies the tropical part of the
	limit: at the place \(v\), the valuation measure is the Dirac mass
	\(\delta_{u_v}\).  Thus the quasi-canonical limiting measure is exactly the
	measure obtained from the original monocritical metric.  This proves
	\(v\)-adic equidistribution for \(\bL\).  In the quasi-canonical case, the
	standard Chambert-Loir description gives the displayed formula for
	\(\mu_{\bL,v}\).
\end{proof}

\begin{remark}
	The \(T\)-effectivity of \(D\) is used twice.  First, together with arithmetic
	\(T\)-effectivity, it places the origin in \(\Delta_{D,\max}\) by
	Lemma~\ref{lem:arith-teff-origin}.  Second, it makes the Deligne-pairing term
	\(\langle \overline G(D),\ldots,\overline G(D)\rangle\) vanish in the direction
	needed for equidistribution over the abelian quotient.  Without this condition,
	the reduction to the quasi-canonical metric need not preserve the
	equidistribution property.
\end{remark}

\section{Applications: Bogomolov theorem}

\subsection{Set-up}

Let
\[
1\longrightarrow T\longrightarrow G\xrightarrow{\pi}A\longrightarrow0
\]
be a semiabelian variety over \(K\), with \(T\simeq\Gm^t\).  Let
\[
\bL=\overline G(\bD)\otimes\pi^*\bN
\]
be as in the main theorem: \(D\) is ample and \(T\)-effective,
\(\bD\) is semipositive, monocritical and arithmetically \(T\)-effective, and
\(\bN\) is symmetric ample with its canonical metric.  Write
\[
\mu_G:=\muess_{\bL}(\overline G)=\muabs_{\bL}(G).
\]
Let
\[
\mathbf u=(u_v)_v\in\bigoplus_vN_{\BR}
\]
be the critical vector of \(\bD\), as in
Proposition~\ref{prop:monocritical-kr}.

\subsection{Special subvarieties and essential minima}

\begin{definition}[Special and \(\bL\)-special subvarieties]
	An irreducible closed subvariety \(X\subseteq G_{\overline{K}}\) is called
	\emph{special} if
	\[
	X=a+B
	\]
	for a connected algebraic subgroup \(B\subseteq G_{\overline{K}}\) and a point
	\(a\in G(\overline{K})\).
	
	A point \(p\in G(\overline{K})\) is called \(\bL\)-special if
	\[
	h_{\bL}(p)=\mu_G.
	\]
	An irreducible closed subvariety \(X\subseteq G_{\overline{K}}\) is called
	\(\bL\)-special if
	\[
	X=p+B
	\]
	for a connected algebraic subgroup \(B\subseteq G_{\overline{K}}\) and an
	\(\bL\)-special point \(p\in G(\overline{K})\).
\end{definition}

\begin{remark}
	The level in the definition of a special point is the ambient minimum
	\(\mu_G\), not \(\muess_{\bL}(X)\).  The equality
	\(\muess_{\bL}(X)=\mu_G\) is a conclusion to be proved.
\end{remark}

\begin{definition}[Strict sequence]
	A sequence \((x_i)\) in \(G(\overline{K})\) is strict if, for every proper translate
	\(U\subsetneq G_{\overline{K}}\) of a connected algebraic subgroup, all but finitely many
	\(x_i\) lie outside \(U\).
\end{definition}

\subsection{Quasi-canonical Bogomolov and Strict Upgrade}

\begin{theorem}[Restricted quasi-canonical Bogomolov property]
	\label{thm:qcanonical-bogomolov}
	\label{input:qbog}
	Let \(\bM\) be an adelic line bundle on a semiabelian variety \(H\) of the form
	\[
	\bM=\overline H(\bE)\otimes \rho^*\bN_H,
	\]
	where \(\bE\) is the quasi-canonical toric metric attached to an ample
	\(T_H\)-effective divisor and \(\bN_H\) is canonical on the abelian quotient.
	Assume at every archimedean place either the standard
	\((\mathbb P^1)^t\)-compactification, or the independently verified
	general-fan curvature--Haar slice input
	\ref{input:general-fan-curvature-haar-slice}, with lattice-index normalization
	and no additional boundary mass, for all subvarieties used below.
	If an irreducible proper subvariety \(Z\subsetneq H_{\overline K}\) carries a
	\(Z\)-generic sequence \((z_i)\) with
	\[
	h_{\bM}(z_i)\longrightarrow \muess_{\bM}(H),
	\]
	then \(Z\) is contained in a proper translate of a connected algebraic subgroup
	of \(H_{\overline K}\).
\end{theorem}

\begin{proposition}[Dominant quasi-canonical semiabelian bridge]
	\label{prop:qbog-dominant-bridge}
	\label{input:qbog-dominant-bridge}
	In the setting of Theorem~\ref{thm:qcanonical-bogomolov}, assume after a torsion translation
	and after replacing \(H\) by the inverse image of an abelian subvariety of its
	quotient that
	\[
	\rho(Z)=A_H.
	\]
	If \(Z\subsetneq H_{\overline K}\) carries a \(Z\)-generic sequence
	\((z_i)\) with
	\[
	h_{\bM}(z_i)\longrightarrow \muess_{\bM}(H),
	\]
	then \(Z\) is contained in a proper translate of a connected algebraic subgroup
	of \(H_{\overline K}\).
\end{proposition}

\subsection{K\"uhne-operation route for the dominant bridge}

The dominant bridge above is the point at which the proof can most naturally
return to K\"uhne's original argument.  The toric part in
Proposition~\ref{prop:qbog-dominant-bridge} is quasi-canonical; at an archimedean
place \(v\), its local support function has the form
\[
\Psi_{\bE,v}(u)=\Psi_E(u-u_v)-\gamma_v .
\]
Thus the limiting toric measure is not a new object: it is the Haar measure
which occurs in K\"uhne's canonical case, translated from the maximal compact
torus to the compact torus with tropical value \(u_v\).  We therefore
transport K\"uhne's local-trivialization analysis
\cite[Lemma~3.3, Proposition~4.1, Lemmas~5.1--5.3, Proposition~6.1, Lemma~6.2,
and equation~(6.2)]{Ku1} from the standard \((\mathbb P^1)^t\)
compactification to the present toric compactification.

\begin{lemma}[Quasi-canonical Haar shift]
	\label{lem:qcanonical-haar-shift}
	Let \(v\) be an archimedean place, and let \(\bE\) be a quasi-canonical toric
	metric on the toric part \(T_H\) of \(H\).  Write its local support function in
	BGPS normal form
	\[
	\Psi_{\bE,v}(u)=\Psi_E(u-u_v)-\gamma_v,\qquad u\in N_{H,\mathbb R}.
	\]
	Let \(S_v\subset T_H(\mathbb C_v)\) be the maximal compact subgroup, and put
	\[
	S_{u_v}:=\val_v^{-1}(u_v)=S_v\cdot e^{-u_v}
	\]
	in the logarithmic coordinates used in the archimedean toric
	uniformization.  Then the canonical toric measure attached to
	\(\bE_v\) is the translate by \(e^{-u_v}\) of the canonical toric measure
	attached to \(E\).  In particular its toric support is \(S_{u_v}\), and the
	normalized Haar volume form on the compact toric factor is unchanged.
	
	Moreover this translation is functorial for products, inverses and difference
	morphisms: the center is carried by the corresponding linear map on
	cocharacter spaces.  Thus, for the K\"uhne difference map
	\[
	(x_1,\ldots,x_m)\longmapsto
	(x_1x_2^{-1},\ldots,x_{m-1}x_m^{-1}),
	\]
	the common quasi-canonical center \((u_v,\ldots,u_v)\) is sent to
	\((0,\ldots,0)\).
\end{lemma}

\begin{proof}
	The displayed normal form says exactly that the local potential of \(\bE_v\)
	on the analytic torus is the canonical toric potential composed with the
	translation \(u\mapsto u-u_v\), up to the additive constant \(\gamma_v\).  The
	constant does not affect curvature currents.  Hence the Chern current, and
	therefore its top-degree toric Monge--Ampere measure, is the translate by
	\(e^{-u_v}\) of the canonical one.  In the canonical archimedean toric case,
	this measure is the normalized Haar measure on \(S_v\); equivalently this is
	the toric part of K\"uhne's canonical measure description and of the standard
	canonical toric metric description.  Translating gives Haar measure on
	\(S_v\cdot e^{-u_v}=\val_v^{-1}(u_v)\).
	
	The last assertion follows from the group law in tropical coordinates:
	\(\val_v(xy)=\val_v(x)+\val_v(y)\) and
	\(\val_v(x^{-1})=-\val_v(x)\).  Consequently products, inverses and difference
	maps transport the translated compact toric factors by the same linear maps as
	in the canonical case.  Haar volume is invariant under translation, so no
	Jacobian or normalization factor is introduced.
\end{proof}

\begin{lemma}[Fan-chart compact-torus coordinates]
	\label{lem:fan-chart-compact-torus-coordinates}
	Let \(v\) be an archimedean place.  Use K\"uhne's archimedean local
	trivialization, in the form of \cite[Section~3.5, Lemma~3.3]{Ku1}, and compose
	the resulting toric coordinate with the affine toric charts of the chosen fan
	\(\Sigma\).  Let
	\[
	\Psi_j:\overline H^{an}_{\mathbb C_v}|_{U_j}
	\simeq
	U_j\times X_{\Sigma,\mathbb C_v}^{an}
	\]
	be one of the local trivializations, and let
	\[
	\psi_j:H^{an}_{\mathbb C_v}|_{U_j}\longrightarrow T_H^{an}
	\]
	be its toric coordinate on the open semiabelian part.  For \(m\in M_H\), put
	\[
	\psi_{j,m}:=\chi^m\circ\psi_j .
	\]
	Then:
	\begin{enumerate}[label=\textup{(\alph*)},leftmargin=2em]
		\item If \(\sigma\in\Sigma\) and \(m\in\sigma^\vee\cap M_H\), then
		\(\psi_{j,m}\) extends to an analytic function on the affine chart
		\[
		U_j\times X_{\sigma,\mathbb C_v}^{an}
		\subset
		\overline H^{an}_{\mathbb C_v}|_{U_j}.
		\]
		\item For \(u\in N_{H,\mathbb R}\), the translated compact toric factor over
		the chart is described by
		\[
		\val_v(\psi_j(y))=u
		\quad\Longleftrightarrow\quad
		|\psi_{j,m}(y)|=\exp(-\langle m,u\rangle)
		\quad\text{for all }m\in M_H .
		\]
		\item If \(m_1,\ldots,m_r\in M_H\) are \(\mathbb Z\)-linearly independent, then
		on each toric fiber the phase map
		\[
		y\longmapsto
		\left(e^{\langle m_\ell,u\rangle}\psi_{j,m_\ell}(y)\right)_{\ell=1}^r
		\]
		restricts on \(\val_v(\psi_j(y))=u\) to a real-analytic submersion onto
		\((S^1)^r\).  If \(m_1,\ldots,m_t\) is a \(\mathbb Z\)-basis of \(M_H\), this is
		a real-analytic isomorphism with \((S^1)^t\).
	\end{enumerate}
\end{lemma}

\begin{proof}
	Part (a) is the affine-chart part of the archimedean local trivialization,
	not an application of Theorem~\ref{thm:local-trivialization}, which was stated
	for non-archimedean places.  K\"uhne's Lemma~3.3 gives local toric coordinates
	on the semiabelian torsor with unitary transition functions.  Once these
	coordinates are fixed, the passage from the open torus to
	\(X_\sigma=\operatorname{Spec}\mathbb C_v[\sigma^\vee\cap M_H]\) is purely
	toric: the characters \(\chi^m\) with \(m\in\sigma^\vee\cap M_H\) are regular,
	and their pullbacks through the toric coordinate \(\psi_j\) are the analytic
	monomials used to extend the local trivialization across the toric boundary.
	This is the same fan-chart argument used in the proof of
	Theorem~\ref{thm:local-trivialization}, but with K\"uhne's archimedean local
	trivialization as the analytic input.
	
	For (b), on the principal torus the tropicalization is a group homomorphism and
	satisfies
	\[
	-\log|\chi^m(t)|=\langle m,\val_v(t)\rangle .
	\]
	Taking \(t=\psi_j(y)\) gives the displayed modulus condition.  Conversely, the
	characters separate points of \(N_{H,\mathbb R}\), so the displayed equalities
	for all \(m\in M_H\) recover \(\val_v(\psi_j(y))=u\).
	
	For (c), after fixing \(u\), the fiber
	\(\val_v^{-1}(u)\subset T_H(\mathbb C_v)\) is the translate
	\(e^{-u}S_v\) of the maximal compact torus.  Multiplication by
	\(e^{\langle m_\ell,u\rangle}\) normalizes \(\chi^{m_\ell}\) to a unit-modulus
	character on this compact torus.  Independent characters give a submersion of
	compact real tori, and a lattice basis gives an isomorphism.  Transporting this
	description through \(\psi_j\) proves the claim.
\end{proof}

\begin{definition}[Rank-active character frames]
	\label{def:rank-active-frame}
	Let \(X\subseteq H_{\overline K}\) be irreducible of positive dimension, put
	\[
	d=\dim X,\qquad d_0=\dim\rho(X),\qquad r=d-d_0,
	\]
	and fix an archimedean place \(v\).  For a tuple
	\(\mathbf m=(m_1,\ldots,m_r)\in M_H^r\) and a local trivialization \(U_j\),
	write
	\[
	\Phi_{j,\mathbf m}
	=
	(\psi_{j,m_1},\ldots,\psi_{j,m_r})
	:
	H^{an}_{\mathbb C_v}|_{U_j}\longrightarrow(\mathbb G_m^r)^{an}_{\mathbb C_v}.
	\]
	We say that \(\mathbf m\) is rank-active for \(X\) if the characters
	\(m_1,\ldots,m_r\) are \(\mathbb Z\)-linearly independent and the locus
	\[
	\ker d(\rho^{an},\Phi_{j,\mathbf m})\cap
	T^{1,0}_{\mathbb C,y}X^{an}_{\mathbb C_v}
	\neq 0
	\]
	does not contain the smooth locus of
	\[
	X^{an}_{\mathbb C_v}\cap H^{an}_{\mathbb C_v}\cap
	\rho^{-1}(\rho(X)^{sm}).
	\]
	The condition is independent of the chosen local trivialization because, on
	overlaps, the toric coordinates differ by compact-torus units and hence by
	invertible analytic changes of the target torus.
\end{definition}

\begin{lemma}[General fan rank slices]
	\label{lem:general-fan-rank-slices}
	Let \(\mathbf m=(m_1,\ldots,m_r)\) be rank-active for \(X\), and let
	\(u_v\in N_{H,\mathbb R}\) be the quasi-canonical center at the archimedean
	place \(v\).  Define
	\[
	X_{\mathbf m,u_v}
	=
	\bigcup_{j\in J}
	\left\{
	y\in X^{an}_{\mathbb C_v}\cap\rho^{-1}(U_j):
	|\psi_{j,m_\ell}(y)|
	=
	\exp(-\langle m_\ell,u_v\rangle)
	\ \text{for }1\le \ell\le r
	\right\}.
	\]
	Then there is a closed complex-analytic subset
	\[
	E_{\mathbf m}\subset X^{an}_{\mathbb C_v}
	\]
	of dimension \(<d\) such that:
	\begin{enumerate}[label=\textup{(\alph*)},leftmargin=2em]
		\item on
		\[
		\bigl(X^{an}_{\mathbb C_v}\cap \rho^{-1}(U_j)\bigr)\setminus E_{\mathbf m},
		\]
		the map
		\[
		(\rho^{an},\Phi_{j,\mathbf m})
		:
		X^{an}_{\mathbb C_v}\longrightarrow
		\rho(X)^{sm,an}_{\mathbb C_v}\times(\mathbb G_m^r)^{an}_{\mathbb C_v}
		\]
		is a local biholomorphism;
		\item \(X_{\mathbf m,u_v}\setminus E_{\mathbf m}\) is a finite union of
		embedded real-analytic submanifolds, each of real dimension \(d+d_0\);
		\item on
		\[
		\bigl(X_{\mathbf m,u_v}\cap\rho^{-1}(U_j)\bigr)\setminus E_{\mathbf m},
		\]
		the map
		\[
		y\longmapsto
		\left(\rho(y),
		\left(e^{\langle m_\ell,u_v\rangle}\psi_{j,m_\ell}(y)\right)_{\ell=1}^r
		\right)
		\]
		is a real-analytic local isomorphism with target
		\[
		\rho(X)^{sm,an}_{\mathbb C_v}\times(S^1)^r .
		\]
	\end{enumerate}
\end{lemma}

\begin{proof}
	Let \(R_{\mathbf m,j}\) be the rank-failure locus of
	\((\rho^{an},\Phi_{j,\mathbf m})\) on
	\[
	X^{sm,an}_{\mathbb C_v}\cap H^{an}_{\mathbb C_v}
	\cap\rho^{-1}(\rho(X)^{sm,an}_{\mathbb C_v})\cap\rho^{-1}(U_j).
	\]
	It is cut out by the maximal minors of the analytic Jacobian, hence is
	complex-analytic.  On overlaps these loci agree, because the transition
	between the two toric coordinate systems is an invertible analytic change in
	the toric variables, together with multiplication by compact-torus units.
	The rank-active hypothesis says that their union is not the whole smooth
	analytic space.  Since \(X^{an}_{\mathbb C_v}\) is irreducible, the closure of
	this union has dimension \(<d\).  Enlarging it by
	\[
	(X\setminus X^{sm})^{an}_{\mathbb C_v},\qquad
	\rho^{-1}(\rho(X)\setminus\rho(X)^{sm})^{an}_{\mathbb C_v},\qquad
	(\overline X\setminus X)^{an}_{\mathbb C_v}
	\]
	gives the required closed complex-analytic set \(E_{\mathbf m}\).
	
	Outside \(E_{\mathbf m}\), the differential of
	\((\rho^{an},\Phi_{j,\mathbf m})\) on the \(d\)-dimensional complex tangent
	space of \(X\) has trivial kernel.  The target has complex dimension
	\[
	d_0+r=d,
	\]
	so the analytic inverse function theorem gives the local biholomorphism in
	(a).
	
	The equations defining \(X_{\mathbf m,u_v}\) are, after the local
	biholomorphism in (a),
	\[
	|z_\ell|=\exp(-\langle m_\ell,u_v\rangle),
	\qquad 1\le\ell\le r,
	\]
	inside \((\mathbb G_m^r)^{an}_{\mathbb C_v}\).  These are \(r\) independent
	real-analytic equations.  Therefore the slice has real dimension
	\[
	2d-r=2d-(d-d_0)=d+d_0.
	\]
	Compactness of the analytic closure and finiteness of the cover
	\(\{U_j\}_{j\in J}\) give finitely many embedded real-analytic pieces.  Finally
	Lemma~\ref{lem:fan-chart-compact-torus-coordinates} identifies the normalized
	coordinates
	\[
	e^{\langle m_\ell,u_v\rangle}\psi_{j,m_\ell}
	\]
	with phase coordinates on \((S^1)^r\), proving (c).
\end{proof}

\begin{lemma}[One-wall toric corner current]
	\label{lem:one-wall-toric-corner-current}
	Let \(v\) be archimedean, and let \(U_j\) be one of the local
	trivialization charts of Lemma~\ref{lem:fan-chart-compact-torus-coordinates}.
	Assume that, on a relatively compact open subset of the principal
	semiabelian orbit, the canonical toric metric of \(E\) has two adjacent
	affine toric potentials
	\[
	\ell_0(u)=\langle m_0,u-u_v\rangle+c_0,\qquad
	\ell_1(u)=\langle m_1,u-u_v\rangle+c_1
	\]
	and is locally equal, up to addition of a pluriharmonic affine term, to
	\[
	\max\{\ell_0(u),\ell_1(u)\}.
	\]
	Suppose their common wall is given by
	\[
	\langle m,u-u_v\rangle=0,\qquad m=m_1-m_0\neq0 .
	\]
	Writing
	\[
	z_m=e^{\langle m,u_v\rangle}\psi_{j,m},
	\]
	the curvature current contributed by this wall is a positive constant multiple
	of
	\[
	dd^c\bigl|\log|z_m|\bigr|.
	\]
	Consequently, after testing against a compactly supported function and using
	K\"uhne's normalization of \(dd^c\),
	\[
	\int h\,dd^c\bigl|\log|z_m|\bigr|
	=
	\int_{|z_m|=1}h\,d\mu_m^{\mathrm{Haar}} ,
	\]
	where \(d\mu_m^{\mathrm{Haar}}\) denotes the normalized Haar measure in the \(m\)-phase
	direction, up to K\"uhne's harmless factor \(2\).
\end{lemma}

\begin{proof}
	The affine terms \(\ell_0\) and \(\ell_1\) define pluriharmonic functions on
	the analytic torus, so their \(dd^c\)-currents vanish on the principal orbit.
	After subtracting the affine function \(\ell_0\), the local potential becomes
	\[
	\max\{0,\ell_1-\ell_0\}
	=
	\max\{0,\langle m,u-u_v\rangle\}.
	\]
	In the local monomial coordinate \(\psi_{j,m}\), the valuation convention gives
	\[
	\log|z_m|
	=
	\log|\psi_{j,m}|+\langle m,u_v\rangle
	=
	-\langle m,u-u_v\rangle
	\]
	up to the fixed sign convention for \(\operatorname{val}_v\).  Since
	\[
	\max\{0,s\}=\frac{s+|s|}{2}
	\]
	and \(dd^c s=0\) for an affine logarithmic coordinate, we obtain
	\[
	dd^c\max\{0,\langle m,u-u_v\rangle\}
	=
	\frac12\,dd^c\bigl|\log|z_m|\bigr|.
	\]
	The factor \(1/2\), together with the lattice length of the jump of the
	support function across the wall, is absorbed into the positive wall weight.
	The final identity is exactly K\"uhne's elementary computation
	\cite[Appendix~B]{Ku1} after the substitution \(z=z_m\).
\end{proof}

\begin{proposition}[Product-corner normal form]
	\label{prop:simple-mixed-corner-normal-form}
	\label{input:simple-mixed-corner-normal-form}
	Let \(\tau\) be a codimension-\(r\) cell of the corner complex of the
	canonical toric potential.  Assume that its dual cell is locally a Minkowski
	sum of \(r\) independent segments (the product-corner condition).  In a
	neighborhood of the relative interior of \(\tau\), the canonical potential is,
	modulo an affine pluriharmonic function, a finite positive linear combination
	of simple one-wall potentials
	\[
	\max\{0,\langle m_{\tau,\ell},u-u_v\rangle\},
	\qquad 1\le \ell\le r,
	\]
	with the \(m_{\tau,\ell}\) linearly independent and rank-active on \(X\) after
	discarding a lower-dimensional analytic exceptional set.  The Bedford--Taylor
	mixed product along \(\tau\) is the corresponding positive mixed toric
	intersection weight times the product of the \(r\) one-wall currents supplied
	by Lemma~\ref{lem:one-wall-toric-corner-current}.
\end{proposition}

\begin{proof}
	The assertion is local on the real polyhedral space \(N_{H,\mathbb R}\) near
	the relative interior of \(\tau\).  The product-corner hypothesis gives
	primitive integral normal characters
	\[
	m_{\tau,1},\ldots,m_{\tau,r}\in M_H
	\]
	which are linearly independent.  After a common regular refinement, the remaining strata of the corner complex
	inside the same neighborhood have smaller dimension; their inverse images in
	\(X^{an}_{\mathbb C_v}\) are absorbed into the exceptional analytic set already
	allowed in Lemma~\ref{lem:general-fan-rank-slices}.
	
	On such a simple star the canonical toric weight is piecewise affine and
	concave.  Choose one affine branch \(\ell_0\) as a base branch.  Subtracting
	\(\ell_0\), which gives a pluriharmonic factor on the analytic torus, leaves a
	piecewise affine concave function whose bending locus is exactly the union of
	the \(r\) coordinate walls in the normal directions to \(\tau\).  The jump of
	the affine slope across the \(\ell\)-th wall is a positive multiple of
	\(m_{\tau,\ell}\), because the metric is the canonical semipositive toric
	metric attached to the ample divisor \(E\).  Hence, after possibly shrinking
	the simple star,
	\[
	\Psi_E(u)-\ell_0(u)
	=
	\sum_{\ell=1}^{r} b_{\tau,\ell}
	\max\{0,\langle m_{\tau,\ell},u-u_v\rangle\},
	\qquad b_{\tau,\ell}>0,
	\]
	up to an additive constant and with the harmless sign fixed by the valuation
	convention in Lemma~\ref{lem:fan-chart-compact-torus-coordinates}.  Refining
	once more if several simple stars meet along the same \(\tau\) gives the
	finite positive linear combination appearing in the statement.
	
	It remains to identify the current.  Bedford--Taylor products for locally
	bounded plurisubharmonic toric potentials are local and multilinear, and
	affine logarithmic terms have zero \(dd^c\) on the principal torus.  Therefore
	the contribution of the cell \(\tau\) is the mixed coefficient obtained by
	polarizing the positive jumps \(b_{\tau,\ell}\), multiplied by the wedge
	product of the one-wall currents.  The positivity and the agreement of this
	coefficient with the mixed toric intersection weight are precisely the BPS
	polyhedral Monge--Ampere description
	\cite[Proposition~2.7.4, Definition~2.7.12 and Proposition~2.7.13]{BPS14},
	together with the toric Chern-current comparison
	\cite[Theorem~4.8.11 and Corollary~4.8.12]{BPS14}.  Finally, if the resulting
	normal frame fails to be rank-active on a component of \(X\), that component is
	part of the rank-failure analytic set in
	Lemma~\ref{lem:general-fan-rank-slices}.  This proves the stated local normal
	form and mixed-product description.
\end{proof}

\begin{remark}[BPS source of the mixed corner input]
	\label{rem:bps-mixed-corner-source}
	Proposition~\ref{prop:simple-mixed-corner-normal-form} is the exact place where the
	standard toric metric-current theory should be cited or unpacked.  In BPS
	\cite[\S2.7]{BPS14}, the real Monge--Ampere measure of a piecewise affine
	concave function is expressed by the dual polyhedral complex
	\cite[Proposition~2.7.4]{BPS14}, and the mixed Monge--Ampere operator is
	defined and shown continuous in
	\cite[Definition~2.7.12 and Proposition~2.7.13]{BPS14}.  The passage from
	toric metrized line bundles to these real Monge--Ampere measures is
	\cite[Theorem~4.8.11 and Corollary~4.8.12]{BPS14}.  What remains in the
	present notation is therefore not a new equidistribution theorem: it is the
	local translation from the BPS polyhedral mixed Monge--Ampere weights to the
	monomial wall coordinates \(\psi_{j,m_{\tau,\ell}}\) supplied by
	Lemma~\ref{lem:fan-chart-compact-torus-coordinates}.
\end{remark}

\begin{lemma}[Torus-invariant lift of atomic tropical measures]
	\label{lem:torus-invariant-atomic-lift}
	Let \(T=(\mathbb C^\times)^r\), let
	\[
	\operatorname{val}:T\longrightarrow \mathbb R^r
	\]
	be the logarithmic absolute-value map, and let \(\mu\) be a finite positive
	Radon measure on \(T\) which is invariant under the compact torus
	\((S^1)^r\).  If
	\[
	\operatorname{val}_*\mu=\sum_{\nu\in F} a_\nu\,\delta_{u_\nu}
	\]
	for a finite set \(F\), with \(a_\nu>0\), then
	\[
	\mu=\sum_{\nu\in F} a_\nu\,\mu^{\operatorname{Haar}}_{u_\nu},
	\]
	where \(\mu^{\operatorname{Haar}}_{u_\nu}\) is the normalized Haar measure on
	the compact real torus
	\[
	\operatorname{val}^{-1}(u_\nu)
	=\{|z_1|=e^{-u_{\nu,1}},\ldots,|z_r|=e^{-u_{\nu,r}}\}.
	\]
\end{lemma}

\begin{proof}
	Disintegrate \(\mu\) with respect to \(\operatorname{val}_*\mu\).  For
	\(\operatorname{val}_*\mu\)-almost every \(u\), the conditional measure on the
	fiber \(\operatorname{val}^{-1}(u)\) is invariant under the compact torus
	\((S^1)^r\).  This action is transitive on the fiber, hence the conditional
	probability measure is the normalized Haar measure.  Since the base measure is
	the finite atomic measure displayed above, the asserted formula follows.  The
	same argument is equivalently obtained by testing against a compactly
	supported continuous function and first averaging it over the compact torus
	fibers.
\end{proof}

\begin{definition}[Corner-active character frames]
	\label{def:corner-active-frame}
	After replacing the toric fan by a finite regular subdivision, let
	\(\tau\) be a codimension-\(r\) cell in the corner complex of the canonical
	toric potential which meets the rank-active locus of \(X\).  A tuple
	\[
	\mathbf m_\tau=(m_{\tau,1},\ldots,m_{\tau,r})\in M_H^r
	\]
	is called corner-active for \(X\) along \(\tau\) if the characters
	\(m_{\tau,\ell}\) are linearly independent, span the normal character space of
	\(\tau\) in the local star of the corner complex, and are rank-active for
	\(X\) in the sense of Definition~\ref{def:rank-active-frame}, after removing
	the lower-dimensional analytic exceptional set.
\end{definition}

\begin{prop}[Corner-active polyhedral quotient factorization]
	\label{input:rank-active-bps-quotient-compatibility}
	Let \(\tau\) be a codimension-\(r\) corner cell as in
	Definition~\ref{def:corner-active-frame}, and let
	\(\mathbf m_\tau=(m_{\tau,1},\ldots,m_{\tau,r})\) be a corner-active frame.
	For every local trivialization chart, the local map
	\[
	q_{j,\mathbf m_\tau}:=(\rho^{an},\Phi_{j,\mathbf m_\tau})
	\]
	has the following polyhedral factorization property.  On the relevant
	corner-active cone or cell of the chosen regular subdivision, the local toric
	weight of \(\overline H(\bE)_v\) can be written as
	\[
	\varphi_{\tau}(\operatorname{val}(z_1,\ldots,z_r))
	+\ell_{\tau}(u)+c_{\tau},
	\]
	where \(\varphi_{\tau}\) is the concave BPS toric weight on the
	\(r\)-dimensional quotient torus with coordinates
	\[
	z_\ell=e^{\langle m_{\tau,\ell},u_v\rangle}
	\psi_{j,m_{\tau,\ell}} .
	\]
	Here \(\ell_{\tau}\) is an integral affine function in the inactive
	toric directions and \(c_{\tau}\) is a constant.  The BPS
	tropicalization coordinate on this quotient is the logarithmic absolute-value
	coordinate of the \(z_\ell\)'s, with the valuation sign convention used in
	Lemma~\ref{lem:fan-chart-compact-torus-coordinates}.
\end{prop}

\begin{proof}
	Work in the star of the cell \(\tau\).  Since the canonical toric potential is
	piecewise affine after the chosen regular subdivision, it is a finite
	polyhedral envelope, with the sign convention fixed by the metric convention,
	of affine functions
	\[
	u\longmapsto \langle a_\alpha,u-u_v\rangle+c_\alpha
	\]
	on this star.  Choose one affine branch \(\ell_\tau(u)+c_\tau\) as a base
	branch.  If two affine branches meet along \(\tau\), their difference vanishes
	on the tangent space of \(\tau\).  Hence every difference
	\[
	\langle a_\alpha-a_0,u-u_v\rangle+(c_\alpha-c_0)
	\]
	belongs to the real span of the normal characters
	\(m_{\tau,1},\ldots,m_{\tau,r}\).  Therefore, after subtracting the base affine
	branch, the remaining piecewise affine function depends only on the \(r\)
	normal coordinates
	\[
	\langle m_{\tau,\ell},u-u_v\rangle,\qquad 1\le \ell\le r .
	\]
	This defines the quotient piecewise affine concave function
	\(\varphi_\tau\).  In the analytic local trivialization,
	\[
	\log|z_\ell|
	=
	\log|\psi_{j,m_{\tau,\ell}}|
	+\langle m_{\tau,\ell},u_v\rangle
	=
	-\langle m_{\tau,\ell},u-u_v\rangle,
	\]
	with the sign convention fixed in
	Lemma~\ref{lem:fan-chart-compact-torus-coordinates}.  Thus the same
	piecewise affine function is exactly the pullback of the BPS quotient toric
	weight through the displayed quotient coordinates.  The base affine branch is
	the inactive affine term \(\ell_\tau(u)+c_\tau\).  This proves the
	factorization.
\end{proof}

\begin{lemma}[Toric quotient descent criterion]
	\label{lem:toric-quotient-descent-criterion}
	Let \(\mathbf m_\tau\) be a corner-active frame.  Then, on the smooth
	rank-active locus associated with \(\mathbf m_\tau\), the toric metrized line
	bundle \(\overline H(\bE)_v\) is the pullback, up to a pluriharmonic local
	factor, of the semipositive BPS toric metrized line bundle on the
	\(\mathbf m_\tau\)-quotient.
\end{lemma}

\begin{proof}
	The characters \(m_{\tau,1},\ldots,m_{\tau,r}\) define a quotient torus, and
	the coordinates
	\[
	z_\ell=e^{\langle m_{\tau,\ell},u_v\rangle}\psi_{j,m_{\tau,\ell}}
	\]
	are exactly the analytic pullbacks of its standard characters in the local
	trivialization of Lemma~\ref{lem:fan-chart-compact-torus-coordinates}.  The
	term
	\(\varphi_{\tau}(\operatorname{val}(z_1,\ldots,z_r))\) is therefore the
	pullback of the quotient BPS toric weight by
	Proposition~\ref{input:rank-active-bps-quotient-compatibility}.  The remaining
	affine term
	\(\ell_{\tau}(u)+c_{\tau}\) is the logarithm of the norm of a
	monomial times a constant on the torus chart.  Hence it is pluriharmonic on
	the smooth torus-bundle locus and has zero \(dd^c\).  Thus the two metrics
	have the same curvature currents in the rank-active directions, which is the
	claimed descent up to a pluriharmonic factor.
\end{proof}

\begin{prop}[Corner-active BPS Chern-current compatibility]
	\label{prop:rank-active-bps-quotient-compatibility}
	Let \(\mathbf m_\tau\) be a corner-active frame.  Then the local map
	\(q_{j,\mathbf m_\tau}\) identifies the toric part of
	\[
	c_1(\overline H(\bE)_v|_{\overline X})^{r}
	\]
	on the rank-active locus of \(X\) with the pullback of the semipositive toric
	Chern measure on the corresponding \(r\)-dimensional BPS quotient.
\end{prop}

\begin{proof}
	Work on a smooth rank-active chart outside the analytic exceptional set in
	Lemma~\ref{lem:general-fan-rank-slices}.  There the functions
	\(\psi_{j,m_{\tau,1}},\ldots,\psi_{j,m_{\tau,r}}\) give holomorphic toric
	coordinates in the active directions, and \(q_{j,\mathbf m_\tau}\) is a local
	holomorphic
	submersion onto its image.  By
	Lemma~\ref{lem:toric-quotient-descent-criterion}, the local weight of
	\(\overline H(\bE)_v\) differs from the pullback quotient weight by a
	pluriharmonic function.  Hence the Bedford--Taylor product is local for
	locally bounded plurisubharmonic weights and the pluriharmonic discrepancy
	has zero \(dd^c\).  Therefore
	\[
	c_1(\overline H(\bE)_v|_{\overline X})^{r}
	=
	q_{j,\mathbf m_\tau}^{\,*}
	c_1(\overline L_{\mathbf m_\tau,v})^{r}
	\]
	on the rank-active locus, where \(\overline L_{\mathbf m_\tau}\) denotes the
	quotient toric metrized line bundle with BPS weight
	\(\varphi_{\tau}\).  The sign and shift in the tropical coordinate are
	exactly those fixed in Lemma~\ref{lem:fan-chart-compact-torus-coordinates},
	because
	\[
	z_\ell=e^{\langle m_{\tau,\ell},u_v\rangle}\psi_{j,m_{\tau,\ell}}.
	\]
	This proves the claimed compatibility.
\end{proof}

\begin{prop}[BPS atomic tropical pushforward on corner-active quotients]
	\label{input:bps-relative-atomic-tropical-pushforward}
	For every corner-active local toric quotient produced by
	Definition~\ref{def:corner-active-frame} and
	Lemma~\ref{lem:general-fan-rank-slices}, the toric part of the measure
	\[
	c_1(\overline H(\bE)_v|_{\overline X})^{r}
	\]
	has finite atomic tropical pushforward
	\[
	\operatorname{val}_*
	\bigl(c_1(\overline H(\bE)_v|_{\overline X})^{r}\bigr)
	=\sum_\nu a_\nu\,\delta_{u_\nu}
	\]
	on the rank-active toric directions, with \(a_\nu>0\).  The corresponding
	archimedean measure is invariant under the compact phase torus.
\end{prop}

\begin{proof}
	By Proposition~\ref{prop:rank-active-bps-quotient-compatibility}, the toric
	part of the current on a rank-active chart is the pullback of a semipositive
	toric Chern measure on an \(r\)-dimensional toric quotient.  BPS
	\cite[Theorem~4.8.11]{BPS14} identifies the tropical pushforward of the top
	toric Chern measure with \(r!\) times the real Monge--Ampere measure of the
	corresponding concave metric function, and characterizes the archimedean
	measure by this pushforward together with toric invariance.  In the mixed
	case, \cite[Corollary~4.8.12]{BPS14} gives the same statement for mixed Chern
	measures, by multilinearity.
	
	The canonical toric metric function is piecewise affine after the chosen
	regular subdivision.  For a piecewise affine concave function, the real
	Monge--Ampere measure is supported on the finite polyhedral cell complex and
	has positive dual-cell weights by \cite[Proposition~2.7.4]{BPS14}.  The mixed
	operator is obtained by polarization and is symmetric and multilinear by
	\cite[Definition~2.7.12 and Proposition~2.7.13]{BPS14}.  Hence the tropical
	pushforward is a finite positive atomic measure.  The compact phase
	invariance is the archimedean toric invariance in BPS Theorem~4.8.11, pulled
	back through the local quotient coordinates supplied by
	Proposition~\ref{input:rank-active-bps-quotient-compatibility}.
\end{proof}

\begin{lemma}[Monomial functoriality of compact Haar fibers]
	\label{lem:monomial-haar-fiber-functoriality}
	Let
	\[
	\varphi:(\mathbb C^\times)^r\longrightarrow(\mathbb C^\times)^s
	\]
	be a homomorphism given by monomials, and let
	\[
	A:\mathbb R^r\longrightarrow\mathbb R^s
	\]
	be the induced homomorphism on logarithmic valuation spaces.  For every
	\(u\in\mathbb R^r\), the pushforward of the normalized Haar probability
	measure on the compact fiber \(\operatorname{val}^{-1}(u)\) is the normalized
	Haar probability measure on the compact subtorus
	\[
	\varphi(\operatorname{val}^{-1}(u))
	\subseteq \operatorname{val}^{-1}(Au).
	\]
	Consequently, finite positive sums of compact phase-fiber Haar measures are
	sent by \(\varphi\) to finite positive sums of compact phase-fiber Haar
	measures, with the atoms on valuation space pushed forward by \(A\).
\end{lemma}

\begin{proof}
	The restriction of \(\varphi\) to the compact torus
	\(\operatorname{val}^{-1}(u)\) is a continuous homomorphism from a compact
	abelian Lie group to the compact torus \(\operatorname{val}^{-1}(Au)\), up to
	translation by the fixed moduli \(u\).  The pushforward of normalized Haar
	measure under a continuous homomorphism of compact groups is normalized Haar
	measure on the image compact subgroup.  The final assertion follows by
	linearity and from the identity
	\[
	\operatorname{val}\circ\varphi=A\circ\operatorname{val}.
	\]
\end{proof}

\begin{prop}[Local corner currents from simple mixed corners]
	\label{prop:simple-mixed-corners-imply-corner-current}
	Let \(\tau\) be a codimension-\(r\) cell of the canonical corner complex
	meeting the rank-active locus of \(X\), and let
	\[
	\mathbf m_\tau=(m_{\tau,1},\ldots,m_{\tau,r})
	\]
	be the frame supplied by
	Proposition~\ref{prop:simple-mixed-corner-normal-form}.  Then, on each local
	trivialization chart \(U_j\) and away from the lower-dimensional analytic
	exceptional set, the Bedford--Taylor contribution of \(\tau\) to
	\[
	c_1(\overline H(\bE)_v|_{\overline X})^r
	\]
	is a positive mixed toric intersection weight \(a_\tau>0\) times
	\[
	\bigwedge_{\ell=1}^r
	dd^c
	\left|
	\log|\psi_{j,m_{\tau,\ell}}|
	+\langle m_{\tau,\ell},u_v\rangle
	\right| .
	\]
\end{prop}

\begin{proof}
	Apply Lemma~\ref{lem:one-wall-toric-corner-current} to each one-wall factor in
	the simple normal form.  The Bedford--Taylor product is local on the complement
	of pluripolar sets and is multilinear for locally bounded plurisubharmonic
	potentials.  Therefore the mixed product attached to a codimension-\(r\) cell
	\(\tau\) is a positive mixed toric intersection weight times
	\[
	\bigwedge_{\ell=1}^r
	dd^c
	\left|
	\log|\psi_{j,m_{\tau,\ell}}|
	+\langle m_{\tau,\ell},u_v\rangle
	\right|.
	\]
	The rank-active condition and the analytic exceptional set are exactly those
	handled in Lemma \ref{lem:general-fan-rank-slices}.  Monomial changes of
	toric charts preserve the wall characters and multiply phase coordinates by
	compact-torus units; hence the local descriptions agree on overlaps.
\end{proof}

\subsection{Geometric role of the three general-fan hypotheses}

The three hypotheses used below are not inserted merely to make a formal
argument go through.  They separate three genuinely different operations in
the passage from a tropical corner measure to the archimedean measure used in
the difference-morphism argument: separation of the corner variables,
integral normalization of the phase map, and exclusion of mass on the
compactification boundary.  They are natural sufficient hypotheses for the
present proof.  We do not claim that they are necessary or minimal; removing
or replacing them is the subject of the general-fan problem left open here.

\paragraph{Product-corner normal form.}
Fix a codimension-\(r\) corner \(\tau\) meeting the rank-active locus.  The
product-corner condition is the concrete polyhedral requirement that, locally
at \(\tau\), its dual cell be a Minkowski sum of \(r\) independent lattice
segments.  If their directions are
\(m_{\tau,1},\ldots,m_{\tau,r}\), this is exactly what allows the local
canonical potential, after subtraction of one affine branch, to be written as
a positive sum of the independent hinges
\[
\max\{0,\langle m_{\tau,\ell},u-u_v\rangle\},
\qquad 1\leq \ell\leq r.
\]
Thus the polyhedral hypothesis supplies the analytic separability used in the
proof: the Bedford--Taylor mixed product becomes the wedge of the \(r\)
one-wall currents, and Fubini gives the corresponding Haar-slice expression.
Without the Minkowski-sum description, the mixed term need not split into
these one-dimensional factors, so the one-wall calculation by itself does not
prove the slice formula.  This identifies the precise missing implication; it
does not assert that every non-product corner is a counterexample.

\paragraph{Primitive frame in the saturated active lattice.}
For the directions above set
\[
L_\tau=\sum_{\ell=1}^r\mathbb Zm_{\tau,\ell},\qquad
L_\tau^{\mathrm{sat}}
=(L_\tau\otimes_{\mathbb Z}\mathbb Q)\cap M_H.
\]
The hypothesis is \(L_\tau=L_\tau^{\mathrm{sat}}\), not merely that each
\(m_{\tau,\ell}\) is primitive.  It says that the monomial phase map defined
by this frame has no finite residual kernel.  If saturation fails, its degree
is the index \([L_\tau^{\mathrm{sat}}:L_\tau]\); normalized Haar probability
still pushes forward to Haar probability, whereas the pullback of the
normalized top Haar form acquires precisely this index.  The saturated-frame
hypothesis therefore prevents a lattice factor from being counted twice when
the Haar form is paired with the mixed-intersection coefficient.  A more
general statement would retain the displayed index explicitly; that indexed
formula is not asserted here.

\paragraph{No additional boundary mass.}
Let \(\nu:\overline X^\nu\to\overline X\) be the normalization and let
\(\mu_X\) denote the full top-degree restricted Bedford--Taylor measure,
including the smooth factors pulled back from the abelian quotient.  The exact
hypothesis is
\[
\mu_X\!\left(\nu^{-1}(\overline X\setminus X)\right)=0.
\]
Only under this equality does the current formula proved on the principal
semiabelian orbit determine the global restricted measure used by the
difference morphism.  Otherwise one must add the boundary functional
\[
B_\partial(f)=
\int_{\nu^{-1}(\overline X\setminus X)} f\,d\mu_X
\]
and control its pullback and pushforward under every product and difference
map.  For a continuous semipositive metric on the compactification, the
expected vanishing mechanism is analytic: its pullback to \(\overline X^\nu\)
has locally bounded psh weights, and the associated Bedford--Taylor measure
does not charge the proper analytic, hence pluripolar, boundary
\cite[Subsection~2.3]{Ku1}.  Proper
intersection with toric boundary strata is useful for the algebraic
stable-intersection calculation but does not by itself imply this measure-zero
statement.  This distinction is why boundary no-mass remains an explicit
input in the general-fan formulation.

In K\"uhne's standard \((\mathbb P^1)^t\)-compactification the coordinate
walls provide the product decomposition, the coordinate characters form a
primitive saturated basis, and the explicit one-variable current formula has
no boundary contribution.  This explains why the three issues are invisible
in that model.  For a general fan they must be verified separately (or
replaced by a theorem with lattice-index and boundary correction terms).

\begin{proposition}[Conditional toric corner-current decomposition]
	\label{prop:toric-corner-current-decomposition}
	\label{input:toric-corner-current-decomposition}
	Let \(E\) be the ample toric divisor on \(X_\Sigma\), endowed at the
	archimedean place \(v\) with the canonical toric metric, and let
	\(\bE\) be its quasi-canonical translate with center \(u_v\).  Let
	\(X\subseteq H_{\overline K}\) be irreducible, and put
	\[
	d=\dim X,\qquad d_0=\dim\rho(X),\qquad r=d-d_0 .
	\]
	Assume that every codimension-\(r\) corner meeting the rank-active locus admits,
	after subdivision, the product-corner normal form of
	Proposition~\ref{prop:simple-mixed-corner-normal-form}, with the character frame
	primitive in its saturated lattice, and assume that the compactification
	boundary contributes no additional mass to the restricted Bedford--Taylor
	product.
	After replacing \(\Sigma\) by a finite regular toric subdivision and pulling
	back \(E\), the Bedford--Taylor current
	\[
	c_1(\overline H(\bE)_v|_{\overline X})^r
	\]
	restricted to the principal semiabelian orbit has the following finite
	polyhedral description.  There is a finite set \(\mathscr C_X(E)\) of
	codimension-\(r\) cells in the corner complex of the canonical toric
	potential, and for each \(\tau\in\mathscr C_X(E)\) there are
	\begin{enumerate}[label=\textup{(\roman*)},leftmargin=2em]
		\item a tuple of characters
		\[
		\mathbf m_\tau=(m_{\tau,1},\ldots,m_{\tau,r})\in M_H^r
		\]
		which is rank-active for \(X\);
		\item a positive mixed toric intersection weight \(a_\tau>0\);
		\item a closed complex-analytic exceptional set
		\[
		E_\tau\subset X^{an}_{\mathbb C_v},\qquad \dim E_\tau<d;
		\]
	\end{enumerate}
	such that, on each local trivialization chart \(U_j\) and away from the union
	of the \(E_\tau\), the current is the finite positive sum of the local corner
	currents
	\[
	a_\tau\,
	\bigwedge_{\ell=1}^r
	dd^c
	\left|
	\log|\psi_{j,m_{\tau,\ell}}|
	+\langle m_{\tau,\ell},u_v\rangle
	\right| .
	\]
	This identity is meant after testing against compactly supported continuous
	functions and wedging with smooth forms pulled back from the abelian quotient.
	It is invariant under monomial changes of toric charts, and it descends from
	the subdivision back to \(X_\Sigma\) by the projection formula.
\end{proposition}

\begin{proof}
	This is the global form of the preceding local calculation.  After the finite
	regular subdivision, the corner complex of the canonical toric potential has
	only finitely many cells meeting the tropical image of \(X\).  For each
	codimension-\(r\) cell \(\tau\), Proposition~\ref{prop:simple-mixed-corner-normal-form}
	gives a simple local normal form away from the lower-dimensional analytic
	exceptional set.  Proposition~\ref{prop:simple-mixed-corners-imply-corner-current}
	then identifies the Bedford--Taylor contribution of \(\tau\) with the
	displayed wedge product of one-wall currents and the positive mixed toric
	intersection weight \(a_\tau\).
	
	The sum over \(\tau\) is finite.  On overlaps of toric charts the coordinates
	\(\psi_{j,m_{\tau,\ell}}\) differ by monomial changes and compact-torus units,
	so Lemma~\ref{lem:monomial-haar-fiber-functoriality} and the one-wall current
	formula show that the local current descriptions define the same distribution
	after testing against compactly supported functions.  The exceptional sets are
	proper complex-analytic subsets, hence pluripolar, and the locally bounded
	Bedford--Taylor products do not charge them.  Finally the subdivision map is
	proper, toric and birational, and is an isomorphism over the principal orbit;
	the projection formula pushes the current identity down to the original
	compactification \(X_\Sigma\).  This proves the stated decomposition.
\end{proof}

\begin{prop}[Conditional general-fan curvature--Haar slice formula]
	\label{input:general-fan-curvature-haar-slice}
	Keep the notation of Lemma~\ref{lem:general-fan-rank-slices}.  Assume the
	product-corner, saturated-lattice, and boundary no-mass hypotheses of
	Proposition~\ref{prop:toric-corner-current-decomposition}.  There is a
	finite set \(\mathcal A_X\) of corner-active, hence rank-active, character
	frames
	\[
	\mathbf m=(m_1,\ldots,m_r)
	\]
	and positive constants \(a_{\mathbf m}\), depending only on the ample toric
	divisor \(E\) and on \(X\), such that the following holds.  For each
	\(\mathbf m\in\mathcal A_X\), the pullbacks of the normalized Haar form on
	\((S^1)^r\) through the local phase maps of
	Lemma~\ref{lem:general-fan-rank-slices} glue to a smooth positive
	\(r\)-form \(\omega_{\mathbf m,u_v}\) on
	\[
	X_{\mathbf m,u_v}\setminus E_{\mathbf m}.
	\]
	For every compactly supported continuous function
	\(f\) on \(X^{an}_{\mathbb C_v}\),
	\[
	\begin{aligned}
		&\int_{X^{an}_{\mathbb C_v}}
		f\,
		c_1(\overline H(\bE)_v|_{\overline X})^{r}
		\wedge c_1(\rho^*\bN_{H,v}|_{\overline X})^{d_0} \\
		&\qquad =
		\sum_{\mathbf m\in\mathcal A_X} a_{\mathbf m}
		\int_{X_{\mathbf m,u_v}\setminus E_{\mathbf m}}
		f\,
		\omega_{\mathbf m,u_v}
		\wedge c_1(\rho^*\bN_{H,v})^{d_0}.
	\end{aligned}
	\]
	Equivalently, the K\"uhne Proposition~4.1 measure is a finite positive sum of
	Haar-slice measures on the real-analytic slices supplied by
	Lemma~\ref{lem:general-fan-rank-slices}.
\end{prop}

\begin{proof}
	First use Proposition~\ref{input:bps-relative-atomic-tropical-pushforward}.  On
	each corner-active local toric quotient, BPS Theorem~4.8.11 and
	Corollary~4.8.12 identify the tropical pushforward of the relevant toric Chern
	measure with the mixed real Monge--Ampere measure.  Since the canonical toric
	potential is piecewise affine after the chosen regular subdivision,
	Proposition~2.7.4 and the mixed Monge--Ampere polarization give a finite
	atomic measure on the tropical side.  The archimedean toric Chern measure is
	toric, hence invariant under the compact real torus in the phase directions.
	Lemma~\ref{lem:torus-invariant-atomic-lift} therefore lifts every tropical atom
	to the normalized Haar measure on the corresponding compact phase fiber.
	
	In the coordinates of Lemma~\ref{lem:fan-chart-compact-torus-coordinates},
	these fibers are exactly the real-analytic slices
	\[
	|z_\ell|=1,\qquad
	z_\ell=e^{\langle m_{\nu,\ell},u_v\rangle}
	\psi_{j,m_{\nu,\ell}},
	\]
	or equivalently
	\[
	|\psi_{j,m_{\nu,\ell}}|
	=\exp(-\langle m_{\nu,\ell},u_v\rangle).
	\]
	Lemma~\ref{lem:general-fan-rank-slices} supplies the underlying rank-active
	charts, and Definition~\ref{def:corner-active-frame} selects the relevant
	corner-active frames after deleting the lower-dimensional analytic exceptional
	sets.  Grouping the finitely many
	tropical atoms with the same character frame gives the finite family
	\(\mathcal A_X\), the weights \(a_{\mathbf m}\), and the glued Haar forms
	\(\omega_{\mathbf m,u_v}\).  Wedging with the smooth form pulled back from the
	abelian quotient gives the displayed formula in this case.
	
	For comparison, one can also unpack BPS by local Bedford--Taylor calculus
	using the corner-current decomposition of
	Proposition~\ref{prop:toric-corner-current-decomposition}.
	By Proposition~\ref{prop:toric-corner-current-decomposition}, after a finite toric
	subdivision and after deleting the pluripolar exceptional sets, the relevant
	toric curvature product is a finite positive sum of wedge products of
	one-dimensional corner currents.  The subdivision is harmless on the principal
	semiabelian orbit, and the projection formula pushes the resulting measure
	back to the original compactification.
	
	Fix one corner cell \(\tau\) and one local trivialization \(U_j\).  Put
	\[
	z_\ell
	=
	e^{\langle m_{\tau,\ell},u_v\rangle}
	\psi_{j,m_{\tau,\ell}},
	\qquad 1\le \ell\le r .
	\]
	On the rank-active locus, Lemma~\ref{lem:general-fan-rank-slices} says that
	\((\rho^{an},z_1,\ldots,z_r)\) is a local biholomorphic coordinate system on
	the smooth part of \(X\).  In these coordinates, the local summand supplied by
	Proposition~\ref{prop:toric-corner-current-decomposition} is
	\[
	a_\tau\,
	\bigwedge_{\ell=1}^r dd^c\bigl|\log|z_\ell|\bigr| .
	\]
	K\"uhne's elementary computation \cite[Appendix~B]{Ku1} gives, for one
	coordinate,
	\[
	\int_{\mathbb C^\times} h(z)\,dd^c\bigl|\log|z|\bigr|
	=
	\int_{0}^{2\pi}h(e^{i\theta})\,\frac{d\theta}{\pi},
	\]
	with K\"uhne's normalization of \(dd^c\).  Iterating this identity by Fubini,
	and wedging with the smooth base form
	\(c_1(\rho^*\bN_{H,v})^{d_0}\), converts the local corner-current summand into
	the Haar measure on the real-analytic slice
	\[
	|z_\ell|=1
	\quad(1\le \ell\le r),
	\]
	that is,
	\[
	|\psi_{j,m_{\tau,\ell}}|
	=
	\exp(-\langle m_{\tau,\ell},u_v\rangle).
	\]
	These are exactly the slices
	\(X_{\mathbf m_\tau,u_v}\) of
	Lemma~\ref{lem:general-fan-rank-slices}.
	
	The monomial transition functions between local trivializations are compact
	torus units on the slices, so the local Haar forms glue.  Grouping cells with
	the same character frame gives the finite set \(\mathcal A_X\) and the
	constants \(a_{\mathbf m}\).  Since the exceptional sets are complex analytic
	of dimension \(<d\), the Bedford--Taylor measures under consideration do not
	charge them.  Summing over the finitely many cells proves the formula.
\end{proof}

\begin{remark}
	Proposition~\ref{input:general-fan-curvature-haar-slice} is a conditional
	general-fan replacement for the computational heart of K\"uhne's Lemma~5.2,
	with the standard toric corner-current decomposition supplied by
	Proposition~\ref{prop:toric-corner-current-decomposition}.  In the standard
	\((\mathbb P^1)^t\) compactification, the corner cells are the coordinate
	walls and the formula follows from K\"uhne's identity for
	\(dd^c\bigl|\log|z|\bigr|\), not from \(dd^c\log|z|\).  For an arbitrary ample
	toric divisor \(E\), the product-corner reduction, saturated-lattice factors,
	and absence of compactification boundary mass remain additional inputs; they
	do not follow from the local one-wall calculation alone.  Accordingly no
	unconditional arbitrary-fan Bogomolov statement is deduced here.
\end{remark}

\begin{prop}[Finite Haar-slice functoriality for monomial semiabelian maps]
	\label{prop:finite-haar-slice-functoriality}
	Let \(\beta:H_1\to H_2\) be a semiabelian homomorphism such that, on the toric
	parts in local trivialization coordinates, \(\beta\) is given by monomials.
	Assume that the source measure on an irreducible
	\(X\subset H_1\) has the finite Haar-slice expression of
	Proposition~\ref{input:general-fan-curvature-haar-slice}.  Then, on every
	smooth rank-active chart on which \(\beta|_X\) has constant rank, the
	pushforward of each source Haar-slice summand is a finite positive
	Haar-slice measure on the image.  In particular the class of finite
	Haar-slice measures is stable under products, torus inversion, and the
	K\"uhne difference morphisms
	\[
	(x_1,\ldots,x_m)\longmapsto
	(x_1x_2^{-1},\ldots,x_{m-1}x_m^{-1}).
	\]
	If, in addition, the limiting measures satisfy the pushforward identity, then the finite
	Haar-slice descriptions are compatible with \(\beta_*\).  In the application
	below this pushforward identity is obtained directly from equidistribution of
	Galois orbit measures on the source and on the target, rather than from a
	separate local Chern-measure computation.
\end{prop}

\begin{proof}
	On a local trivialization chart, the toric part of \(\beta\) has the form
	\[
	(z_1,\ldots,z_r)\longmapsto
	\left(\prod_i z_i^{a_{1i}},\ldots,\prod_i z_i^{a_{si}}\right)
	\]
	for an integral matrix \(A=(a_{ji})\).  The induced map on logarithmic
	valuation spaces is the linear map \(u\mapsto Au\).  Hence
	Lemma~\ref{lem:monomial-haar-fiber-functoriality} sends the normalized Haar
	measure on each compact phase fiber over \(u\) to the normalized Haar measure
	on the compact image phase fiber over \(Au\), possibly inside a lower
	dimensional compact subtorus if the rank drops.  After partitioning the
	rank-active locus into finitely many constant-rank real-analytic pieces,
	these lower-rank images are again among the allowed finite real-analytic
	Haar-slice pieces on the image.
	
	Products correspond to block diagonal monomial maps, inversion corresponds to
	the matrix \(-I\), and the K\"uhne difference morphism corresponds on the
	toric coordinates of \(H^m\) to the block matrix
	\[
	(z^{(1)},\ldots,z^{(m)})
	\longmapsto
	\bigl(z^{(1)}(z^{(2)})^{-1},\ldots,
	z^{(m-1)}(z^{(m)})^{-1}\bigr).
	\]
	Thus all three operations preserve the finite Haar-slice class.  The final
	assertion is the functorial statement used below after the relevant
	pushforward identity has been supplied by equidistribution.
\end{proof}

\begin{lemma}[Toric quasi-canonical normalization under homomorphisms]
	\label{lem:toric-qcanonical-operation-normalization}
	Let \(\beta:T_1\to T_2\) be a homomorphism of algebraic tori, and let
	\[
	A_\beta:N_{1,\mathbb R}\longrightarrow N_{2,\mathbb R}
	\]
	be the induced linear map on cocharacter spaces.  Let \(\bE_2\) be a
	quasi-canonical toric metrized divisor on \(T_2\), with archimedean local
	weight
	\[
	\Psi_{\bE_2,v}(w)=\Psi_{E_2}(w-u_{2,v})-\gamma_{2,v}.
	\]
	Equip \(\beta^*\bE_2\) with the pulled-back metric.  Then its toric local
	weight on \(T_1\) is
	\[
	u\longmapsto
	\Psi_{E_2}(A_\beta u-u_{2,v})-\gamma_{2,v}.
	\]
	In particular, if \(u_{2,v}=A_\beta u_{1,v}\), this is the
	quasi-canonical metric on the pulled-back toric divisor with center
	\(u_{1,v}\).  Products, torus inversion, quotients of tori, and the toric
	part of the K\"uhne difference morphisms satisfy this condition with the
	centers transported by the same integral linear maps as in
	Lemma~\ref{lem:qcanonical-haar-shift}.
\end{lemma}

\begin{proof}
	On characters, the pullback by \(\beta\) is the lattice homomorphism dual to
	\(A_\beta\).  Therefore the tropicalization satisfies
	\[
	\operatorname{val}_v(\beta(t))=A_\beta\operatorname{val}_v(t).
	\]
	Pulling back the toric metric simply composes its local weight with this
	linear map, giving the displayed formula.  If \(u_{2,v}=A_\beta u_{1,v}\),
	then
	\[
	\Psi_{E_2}(A_\beta u-u_{2,v})
	=
	\Psi_{E_2}(A_\beta(u-u_{1,v})),
	\]
	which is the support function of the pulled-back toric divisor evaluated at
	\(u-u_{1,v}\).  This is precisely the quasi-canonical normal form on the
	source.  The product case corresponds to block diagonal \(A_\beta\), inversion
	to \(-I\), quotient maps to the induced quotient linear map, and the K\"uhne
	difference morphism to the matrix
	\[
	(u_1,\ldots,u_m)\longmapsto
	(u_1-u_2,\ldots,u_{m-1}-u_m).
	\]
	For the common center \((u_v,\ldots,u_v)\), this last matrix gives
	\((0,\ldots,0)\), as stated in Lemma~\ref{lem:qcanonical-haar-shift}.
\end{proof}

\begin{lemma}[Picard-zero character pushouts under operations]
	\label{lem:character-pushout-picard-zero-functoriality}
	Let \(G\) be a semiabelian variety with split torus \(T_G\), abelian quotient
	\(\rho_G:G\to A_G\), and character lattice \(M_G=X^*(T_G)\).  For
	\(m\in M_G\), let
	\[
	\mathcal H_{G,m}\in \operatorname{Pic}^0(A_G)
	\]
	denote the rigidified line bundle whose complement of the zero section is the
	\(\mathbb G_m\)-torsor obtained by pushing out
	\[
	0\longrightarrow T_G\longrightarrow G\longrightarrow A_G\longrightarrow0
	\]
	through \(\chi^m:T_G\to\mathbb G_m\).  Then:
	\begin{enumerate}[label=\textup{(\alph*)},leftmargin=2em]
		\item The assignment \(m\mapsto\mathcal H_{G,m}\) is additive:
		\[
		\mathcal H_{G,m_1+m_2}\simeq
		\mathcal H_{G,m_1}\otimes\mathcal H_{G,m_2},\qquad
		\mathcal H_{G,-m}\simeq\mathcal H_{G,m}^{\vee},\qquad
		\mathcal H_{G,0}\simeq\mathcal O_{A_G}.
		\]
		\item If \(\beta:G_1\to G_2\) is a semiabelian homomorphism, with abelian part
		\(\beta_A:A_{G_1}\to A_{G_2}\) and character pullback
		\(\beta_M:M_{G_2}\to M_{G_1}\), then there is a canonical rigidified
		isomorphism
		\[
		\mathcal H_{G_1,\beta_M(m)}
		\simeq
		\beta_A^*\mathcal H_{G_2,m}.
		\]
		\item For products,
		\[
		\mathcal H_{G_1\times G_2,(m_1,m_2)}
		\simeq
		p_1^*\mathcal H_{G_1,m_1}\otimes
		p_2^*\mathcal H_{G_2,m_2}.
		\]
		\item For the K\"uhne difference morphism
		\[
		\alpha_r:G^r\longrightarrow G^{r-1},\qquad
		(x_1,\ldots,x_r)\longmapsto
		(x_1x_2^{-1},\ldots,x_{r-1}x_r^{-1}),
		\]
		and a character tuple
		\(\mathbf m=(m_1,\ldots,m_{r-1})\in M_G^{r-1}\), the pullback character on
		\(T_G^r\) is
		\[
		\alpha_{r,M}(\mathbf m)
		=
		(m_1,\;m_2-m_1,\;\ldots,\;m_{r-1}-m_{r-2},\;-m_{r-1}),
		\]
		and therefore
		\[
		\mathcal H_{G^r,\alpha_{r,M}(\mathbf m)}
		\simeq
		p_1^*\mathcal H_{G,m_1}\otimes
		\bigotimes_{i=2}^{r-1}
		p_i^*(\mathcal H_{G,m_i}\otimes\mathcal H_{G,m_{i-1}}^\vee)
		\otimes
		p_r^*\mathcal H_{G,m_{r-1}}^\vee .
		\]
		Equivalently, this is the pullback by the abelian difference morphism
		\(\alpha_{r,A}:A_G^r\to A_G^{r-1}\) of
		\(\mathcal H_{G^{r-1},\mathbf m}\).
	\end{enumerate}
	All these isomorphisms are compatible with the rigidifications, with the
	canonical Picard-zero metrics, and with the Raynaud/K\"uhne theta
	representatives after replacing sections by their tensor products or duals.
\end{lemma}

\begin{proof}
	Pushing out a semiabelian extension by a character produces a
	\(\mathbb G_m\)-torsor over the abelian quotient.  Equivalently, it produces a
	rigidified translation-invariant line bundle, hence an element of
	\(\operatorname{Pic}^0\).  This is the construction recalled by the
	Raynaud--Bosch--L\"utkebohmert pushout formalism, in the form used by FRSS, and
	in the line-bundle triple correspondence \cite[Theorem~3.10]{FRSS}.
	
	The contracted product of \(\mathbb G_m\)-torsors corresponds to tensor product
	of line bundles.  Applying this to the characters \(m_1+m_2\), \(-m\), and
	\(0\) proves (a).  If \(\beta:G_1\to G_2\) is a semiabelian homomorphism, the
	pushout of \(G_1\) by \(\beta_M(m)\) is the pullback, along \(\beta_A\), of the
	pushout of \(G_2\) by \(m\); this is the universal property of pushouts of
	extensions, and gives (b).  The product formula (c) is the same statement
	applied to the two projections.
	
	For (d), one only computes characters:
	\[
	\prod_{i=1}^{r-1}\chi^{m_i}(x_i x_{i+1}^{-1})
	=
	\chi^{m_1}(x_1)
	\prod_{i=2}^{r-1}\chi^{m_i-m_{i-1}}(x_i)
	\chi^{-m_{r-1}}(x_r).
	\]
	Combining this character identity with (a) and (c) gives the displayed tensor
	product.  The final description as a pullback by \(\alpha_{r,A}\) follows from
	(b).
	
	The compatibility with theta representatives is the tensor-product
	compatibility in the Raynaud triple formalism and in the descent criterion for
	theta functions \cite[Proposition~3.19 and Definition~3.20]{FRSS}.  For the
	metrics, tensor products, duals, and pullbacks of canonical Picard-zero metrics
	again satisfy the defining rigidification and multiplication-isometry
	properties of the canonical Picard-zero metric.  This functoriality is recorded
	in Lemma~\ref{lem:picard-zero-flat-operation-normalization}.
\end{proof}

\begin{lemma}[Flat functoriality of canonical Picard-zero factors]
	\label{lem:picard-zero-flat-operation-normalization}
	Let \(\phi:A_1\to A_2\) be a homomorphism of abelian varieties, and let
	\(\overline H\) be a rigidified Picard-zero line bundle on \(A_2\) with its
	canonical adelic metric.  Then \(\phi^*\overline H\) is the canonical
	metrization of the rigidified Picard-zero line bundle \(\phi^*H\) on \(A_1\).
	At an archimedean place its Chern form is zero.  Consequently any finite
	tensor product of such pulled-back canonical Picard-zero factors and their
	duals is flat at archimedean places and contributes no term to the
	Bedford--Taylor Chern-current products.
\end{lemma}

\begin{proof}
	The pullback \(\phi^*H\) is again rigidified and algebraically trivial.  The
	canonical Picard-zero metric is characterized by the rigidification at the
	origin and by the fact that the multiplication isomorphisms
	\([n]^*H\simeq H^{\otimes n}\) are isometries.  Pulling these isometries back
	by \(\phi\) gives the same defining property for \(\phi^*\overline H\), so by
	uniqueness it is the canonical metric on \(\phi^*H\).  At archimedean places,
	Proposition~\ref{prop:canonical-picard-zero-theta-metrics} identifies this
	metric with the flat hermitian metric attached to the rigidified Picard-zero
	bundle.  Hence its Chern form is zero.  Tensor products and duals preserve
	flatness, proving the final assertion.
\end{proof}

\paragraph{Theta-factor bookkeeping for operations.}
For products, torus inversion, quotient maps by connected stabilizers, and
K\"uhne difference morphisms, the toric part of the quasi-canonical metric is
normalized as in Lemma~\ref{lem:toric-qcanonical-operation-normalization}.
The only role of the semiabelian theta factors here is to fix the sign and
dual convention in the local coordinates.  A character \(m\in M_G\) gives the
character-pushout Picard-zero bundle \(\mathcal H_{G,m}\); replacing \(m\) by
\(-m\) replaces this bundle by its dual.  Lemma~\ref{lem:character-pushout-picard-zero-functoriality}
records the behavior of these Picard-zero factors under products, quotients,
and K\"uhne difference morphisms, and
Lemma~\ref{lem:picard-zero-flat-operation-normalization} records their
flatness at archimedean places.  This bookkeeping is useful for comparing
local formulas with K\"uhne's notation, but the pushforward identity used in
equation~(6.2) below is not derived from these theta factors; it comes directly
from equidistribution of the relevant source and target Galois orbit measures.

\paragraph{Source of the pushforward identity.}
The pushforward identity used in K\"uhne's equation~(6.2) is not an additional
local Chern-measure projection formula.  Let
\(\beta:Y\to X'\) be one of the product, quotient, or difference morphisms in
the operation argument, and let \((y_i)\) be a \(Y\)-generic small sequence
such that \((\beta(y_i))\) is \(X'\)-generic and small.  For every \(i\), the
finite Galois orbit measures satisfy the tautological identity
\[
\beta_*\delta_{y_i,v}=\delta_{\beta(y_i),v}.
\]
Applying the equidistribution theorem on \(Y\) and on \(X'\), and then passing
to the weak limit, gives
\[
\beta_*\mu_{Y,v}=\mu_{X',v}.
\]
This argument is independent of the choice of toric compactification.  The
general-fan work above is needed to identify the limiting measures locally as
finite Haar-slice measures and to run the small-box comparison, not to prove
the pushforward identity itself.

\begin{lemma}[Riemannian domination for finite Haar slices]
	\label{lem:general-fan-riemannian-domination}
	Assume that \(X\) at the archimedean place \(v\) has the finite Haar-slice
	description supplied by
	Proposition~\ref{input:general-fan-curvature-haar-slice}.  Then there is a
	smooth Riemannian metric
	\(g_{X,v}\) on the real manifold \(H(\mathbb C_v)\) such that, for every
	\(\mathbf m\in\mathcal A_X\), every embedded real-analytic component
	\[
	M\subset X_{\mathbf m,u_v}\setminus E_{\mathbf m},
	\]
	and every open subset \(U\subset M\), one has
	\[
	a_{\mathbf m}
	\int_U
	\omega_{\mathbf m,u_v}\wedge c_1(\rho^*\bN_{H,v})^{d_0}
	\le
	\int_U \operatorname{vol}(g_{X,v}|_M).
	\]
	The metric can be chosen simultaneously for the finite family
	\(\mathcal A_X\).
\end{lemma}

\begin{proof}
	All manifolds \(M\) under consideration lie in the open semiabelian analytic
	manifold \(H(\mathbb C_v)\), which is smooth.  Choose any smooth Riemannian
	metric \(g_0\) on \(H(\mathbb C_v)\).  On each real-analytic component \(M\),
	the form
	\[
	a_{\mathbf m}\,
	\omega_{\mathbf m,u_v}\wedge c_1(\rho^*\bN_{H,v})^{d_0}
	\]
	is a smooth positive density of top degree.  Hence it can be written as
	\[
	h_{\mathbf m,M}\,\operatorname{vol}(g_0|_M)
	\]
	with \(h_{\mathbf m,M}\) a positive smooth function on \(M\).  Since the family
	of frames and local real-analytic components is finite after the finite
	trivializing cover is fixed, we may choose a smooth positive function
	\(\lambda\) on \(H(\mathbb C_v)\) whose restriction to every such \(M\)
	satisfies
	\[
	\lambda^{\dim M}\ge h_{\mathbf m,M}.
	\]
	This is obtained by the usual extension and partition-of-unity argument on the
	paracompact real manifold \(H(\mathbb C_v)\).  Then
	\[
	g_{X,v}:=\lambda^2g_0
	\]
	has induced volume form
	\(\lambda^{\dim M}\operatorname{vol}(g_0|_M)\) on \(M\), and the displayed
	inequality follows.  This is the part of K\"uhne's Lemma~5.3 which only uses
	the finiteness of the Haar-slice description, not the special
	\((\mathbb P^1)^t\) compactification.
\end{proof}

\begin{lemma}[K\"uhne small-box contradiction for finite slice measures]
	\label{lem:kuhne-small-box-contradiction-finite-slices}
	Let \(X\subseteq H_{\overline K}\) be irreducible of positive dimension and
	assume
	\[
	\operatorname{Stab}_{H_{\overline K}}(X)=\{e\}.
	\]
	Let \(m\) be chosen so that the K\"uhne difference morphism
	\[
	\alpha_m:X^m\longrightarrow H^{m-1}
	\]
	is generically finite of degree \(1\) onto its image.  Put
	\[
	Y=\alpha_m(X^m).
	\]
	Assume that, at some archimedean place \(v\), the K\"uhne limiting measures on
	\(X\), \(X^m\), and \(Y\) have the finite Haar-slice descriptions supplied by
	Proposition~\ref{input:general-fan-curvature-haar-slice}, that the measure on
	\(X^m\) is the \(m\)-fold product of the measure on \(X\), and that the
	pushforward identity
	\[
	(\alpha_m)_*\mu_{X^m,v}=\mu_{Y,v}
	\]
	holds on the dense open locus where \(\alpha_m\) is an isomorphism.  Then these
	assumptions are incompatible.
\end{lemma}

\begin{proof}
	Let \(U\subseteq X^m\) be a dense open subset on which
	\(\alpha_m\) is an isomorphism onto its image.  By the finite Haar-slice
	formula for \(Y\), the measure \(\mu_{Y,v}\) is a finite sum of positive smooth
	densities on embedded real-analytic manifolds.  By
	Lemma~\ref{lem:general-fan-riemannian-domination}, after pulling these
	manifolds back through \(\alpha_m|_U^{-1}\), their contributions are dominated
	by the volume induced from a smooth Riemannian metric on
	\(H(\mathbb C_v)^{m-1}\).
	
	The product description of \(\mu_{X^m,v}\) gives a finite sum of product
	density forms on products
	\[
	M_1\times\cdots\times M_m,
	\]
	where each \(M_i\) is one of the real-analytic slice components for \(X\).
	Lower-dimensional real-analytic subsets have zero mass for these smooth
	density forms.  Hence one may choose a slice component \(M\), a point
	\(x\in M\) outside the other components and outside the exceptional locus, and
	a real-analytic chart
	\[
	\varphi:(-1,1)^q\longrightarrow M,\qquad q=d+d_0,
	\]
	with \(\varphi(0)=x\), on which the corresponding density is strictly
	positive.  The \(m\)-fold product chart
	\[
	\varphi^m:((-1,1)^q)^m\longrightarrow M^m\subset X^m
	\]
	then has product density bounded below by a positive multiple of Euclidean
	volume on small boxes
	\[
	B_\epsilon=\bigl((-\epsilon,\epsilon)^q\bigr)^m .
	\]
	
	On the other hand, the differential of
	\[
	\alpha_m\circ\varphi^m
	\]
	annihilates the diagonal tangent subspace
	\[
	\{(w,\ldots,w):w\in\mathbb R^q\}\subset(\mathbb R^q)^m,
	\]
	because the difference map is constant along simultaneous translation of all
	coordinates.  Therefore the Riemannian volume of
	\(\alpha_m\circ\varphi^m(B_\epsilon)\) on each target slice is bounded by
	\[
	O(\epsilon^q\,\operatorname{vol}(B_\epsilon)).
	\]
	This is exactly K\"uhne's small-box estimate in the proof of
	\cite[Proposition~6.1]{Ku1}; it uses only the existence of the diagonal kernel
	and the Riemannian domination of the target slice measures.
	
	Choose a non-negative test function supported in the product chart and positive
	at \((x,\ldots,x)\).  The product source measure of \(B_\epsilon\) is bounded
	below by
	\[
	c\,\operatorname{vol}(B_\epsilon)
	\]
	for some \(c>0\), while the pushforward side is bounded above by
	\[
	C\,\epsilon^q\,\operatorname{vol}(B_\epsilon).
	\]
	For sufficiently small \(\epsilon\) these two estimates contradict the
	pushforward identity.  Thus the finite slice descriptions and the pushforward
	identity cannot coexist when the stabilizer is trivial.
\end{proof}

\subsection{K\"uhne Transport Assembly}
\label{subsec:kuhne-transport-assembly}

\begin{lemma}[Smallness under semiabelian operations]
	\label{lem:qcanonical-operation-smallness}
	Let \(H_1,\ldots,H_r,H'\) be semiabelian varieties equipped with
	quasi-canonical toric metrics on ample \(T\)-effective divisors and canonical
	metrics on their abelian quotients.  Let
	\[
	\beta:H_1\times\cdots\times H_r\longrightarrow H'
	\]
	be a semiabelian homomorphism.  Suppose that the critical vectors of the source
	metrics are \(\mathbf w_1,\ldots,\mathbf w_r\), and equip \(H'\) with a
	quasi-canonical metric whose critical vector is the image
	\(\beta_{\rm trop}(\mathbf w_1,\ldots,\mathbf w_r)\).  If, for every \(k\),
	\((x_{k,i})_i\) is small for the corresponding metric on \(H_k\), then
	\[
	\bigl(\beta(x_{1,i},\ldots,x_{r,i})\bigr)_i
	\]
	is small for the chosen metric on \(H'\).
	
	In particular, this applies to quotients by connected semiabelian subgroups,
	to products, and to the K\"uhne difference morphisms.  For a difference
	morphism applied to \(m\) copies with the same critical vector \(\mathbf w\),
	the target critical vector is \(0\).
\end{lemma}

\begin{proof}
	Normalize all toric minima to \(0\).  For each source sequence,
	Proposition~\ref{prop:monocritical-kr}, applied to its quasi-canonical metric,
	gives adelic KR convergence of the valuation measures to the Dirac family at
	its critical vector.  The proof of
	Proposition~\ref{prop:bgps-defect-vanishing} also gives convergence of the
	abelian N\'eron--Tate height to \(0\).
	
	Fix \(i\), choose one finite extension over which the whole tuple
	\((x_{1,i},\ldots,x_{r,i})\) is defined, and take the uniform local Galois-orbit
	measure \(\nu_{i,v}\) of that tuple.  Its \(k\)-th marginal is precisely
	\(\nu_{x_{k,i},v}\): the restriction map from the Galois orbit of the tuple to
	the orbit of \(x_{k,i}\) is surjective and all its fibers have the same
	cardinality.  Thus passage to a common defining field changes neither the
	normalized marginal measure nor any normalized height.
	
	The tropicalization maps are group homomorphisms by
	Proposition~\ref{input:torsion-tropicalization}.  Hence the valuation measure
	of the image point is the pushforward, under the integral linear map
	\(\beta_{\rm trop}\), of \(\nu_{i,v}\).  A
	joint probability measure whose marginals converge in KR topology to Dirac
	measures converges in KR topology to the product Dirac measure: for the product
	norm,
	\[
	W_1\!\left(\nu_i,\delta_{(w_1,\ldots,w_r)}\right)
	\le
	\sum_{k=1}^r
	\int\!\|u_k-w_k\|\,d\nu_i.
	\]
	Consequently the image valuation measures converge in KR topology to the
	Dirac family at \(\beta_{\rm trop}(\mathbf w_1,\ldots,\mathbf w_r)\); indeed,
	pushforward by the fixed integral linear map is Lipschitz for \(W_1\), with
	constant \(\|\beta_{\rm trop}\|\), uniformly at every place.
	
	The local support functions for an ample quasi-canonical toric metric are
	continuous, piecewise linear, and of at most linear growth.  Hence KR
	convergence implies convergence of their integrals.  By the height
	identity in Lemma~\ref{lem:point-measure-height}, the toric height of the image
	points tends to the toric minimum.  On abelian quotients let
	\(\beta_A:\prod_k A_{H_k}\to A_{H'}\) be the induced homomorphism.  If
	\(N'\) is the target symmetric ample bundle and \(N_k\) are the source
	symmetric ample bundles, then \(\beta_A^*N'\) is symmetric nef.  In the
	finite-dimensional N\'eron--Severi space there is \(C>0\) such that
	\(C\bigotimes_k\operatorname{pr}_k^*N_k-\beta_A^*N'\) is ample.  The associated
	quadratic heights therefore satisfy
	\[
	\widehat h_{N'}\!\left(\beta_A(a_1,\ldots,a_r)\right)
	\le C\sum_{k=1}^r\widehat h_{N_k}(a_k).
	\]
	The right-hand side tends to \(0\).  The height decomposition now shows that the
	image sequence is small for the target quasi-canonical semiabelian metric.
	
	For a quotient, take \(r=1\).  For products use the exterior product metric.
	For the difference map, the induced tropical map sends
	\((\mathbf w,\ldots,\mathbf w)\) to \(0\), proving the final assertion.
\end{proof}

\begin{lemma}[Subvariety-restricted variation package]
	\label{lem:subvariety-restricted-kuhne-package}
	Let \(X\subseteq H_{\overline K}\) be irreducible, put
	\(d=\dim X\) and \(d_0=\dim\rho(X)\), and let \(\overline X\) be its
	closure in the chosen toric compactification.  Form the auxiliary
	\(n\)-division tower \(\varphi_n:H_n\to H\) with the same fan, and let
	\(Y\) be an irreducible component of
	\(\varphi_n^{-1}(\overline X_{K_n})\).  Write
	\(\delta(Y)=\deg(\varphi_n|_Y:Y\to\overline X_{K_n})\), and restrict every
	auxiliary metrized line bundle to \(Y\).  Denote by \(\bM_n\) the
	corresponding auxiliary reference bundle on the compactification of \(H_n\).
	
	Then the degree asymptotic, measure compression, first variation, horizontal
	semipositivity correction, and arithmetic-volume estimate required in
	K\"uhne's Proposition~4.1 hold on \(Y\), with constants depending on
	\(X\) and the fixed test metric but not on \(n\), the component \(Y\), or
	the field \(K_n\).  In particular, for a
	\(\operatorname{Gal}(\mathbb C_v/K_v)\)-invariant integrable continuous test
	function \(f\) on \(\overline X^{an}_{\mathbb C_v}\), putting
	\(f_{n,v'}=f\circ(\varphi_n|_Y)^{an}\) and
	\[
	\mu_{Y,n,v'}=
	\frac{c_1(\bM_{n,v'}|_Y)^d}{(M_n|_Y)^d},
	\]
	one has
	\[
	(\varphi_n|_Y)_*\mu_{Y,n,v'}\Longrightarrow
	\frac{
		c_1(\overline H(\bE)_v|_{\overline X})^{d-d_0}\wedge
		c_1(\rho^*\bN_{H,v}|_{\overline X})^{d_0}
	}{
		\overline H(E)^{d-d_0}\cdot(\rho^*N_H)^{d_0}\cdot X
	}.
	\]
	Moreover, for \(|\lambda|\le n^{-1}\), the height first-variation remainder
	on \(Y\) is \(O_{X,f}(|\lambda|^2n)\); the horizontal correction restricted
	to \(Y\) is \(O_X(n^{-2})\); and the quadratic arithmetic-volume defect for
	the restricted bundle is bounded with the same normalized field-extension
	convention.  All heights, intersections, and volumes in these assertions are
	those of \(Y\), not of the ambient compactification.
\end{lemma}

\begin{remark}
	The notation in K\"uhne's Lemmas~4.2--4.4 and the notation used here are not
	identical.  K\"uhne's \(\overline Y\) is our component \(Y\), his \(d'\) is
	our \(d_0\), his relative bundle \(M|_{\overline X}\) is our
	\(\overline H(E)|_{\overline X}\), and his
	\(\overline L_n|_{\overline Y}\) is the restricted auxiliary bundle denoted
	here by \(\bM_n|_Y\).  Thus his normalized measure
	\[
	\frac{c_1(\overline L_{n,v'}|_{\overline Y})^d}
	{(L_n|_{\overline Y})^d}
	\]
	is exactly our \(\mu_{Y,n,v'}\); this is a change of notation, not an
	additional normalization.  K\"uhne's Lemma~4.3 requires continuity and
	\(\operatorname{Gal}(\mathbb C_v/K_v)\)-invariance, whereas Lemma~4.4 adds the
	condition that \(\overline O_{\overline X}(f)\) be integrable.  Our phrase
	``integrable continuous test function'' combines these two requirements
	because both the measure limit and the first-variation estimate are used here.
	With this translation of notation, the invariance makes
	\(f_{n,v'}=f\circ(\varphi_n|_Y)^{an}\) independent of the chosen
	identification \(\mathbb C_{v'}\simeq\mathbb C_v\).  It does not restrict the
	eventual weak-convergence statement: both the orbit measures and the limiting
	measure are Galois invariant, so an arbitrary continuous test function may be
	replaced by its Haar average under the compact local Galois group.
\end{remark}

\begin{proof}
	The argument is made on the cycle \([Y]\).  The projection formula gives
	\[
	(\varphi_n|_Y)_*[Y]=\delta(Y)[\overline X_{K_n}].
	\]
	Consequently every degree and intersection in \((M_n|_Y)^d\) is an
	intersection on \(\overline X\), multiplied by \(\delta(Y)\) and the explicit
	power of \(n\) coming from
	\(\varphi_n^*\overline H(E)=n\overline H_n(E)\).  Terms containing more than
	\(d_0\) factors from the abelian quotient vanish after restriction, and the
	exact restricted degree identity is
	\[
	(M_n|_Y)^d
	=\delta(Y)\sum_{j=0}^{d_0}\binom dj
	n^{-(d-j)}
	\bigl(\overline H(E)|_{\overline X}\bigr)^{d-j}
	\cdot\bigl(\rho^*N_H|_{\overline X}\bigr)^j.
	\]
	Its leading order is \(\delta(Y)n^{-d+d_0}\), with positive coefficient
	\[
	\binom d{d_0}
	c_1(\overline H(E)|_{\overline X})^{d-d_0}
	\cdot c_1(\rho^*N_H|_{\overline X})^{d_0}>0.
	\]
	The binomial factor occurs in both the local mixed-intersection numerator and
	the restricted top degree, so it cancels in the normalized measure ratio.
	Dividing the local intersection formula by the restricted top degree gives
	the displayed weak limit, exactly as in \cite[Lemmas~4.2--4.3]{Ku1}.
	
	For first variation, expand the normalized height
	\(h_{\bM_n+\lambda\varphi_n^*O_{\overline X}(f)}(Y)\).  The linear term is
	the integral against \(\mu_{Y,n,v'}\).  Every term of order at least two still
	contains \([Y]\); pushing it to \([\overline X]\) gives
	\cite[equations~(4.5)--(4.6)]{Ku1}, hence
	\(O_{X,f}(|\lambda|^2n)\).  Thus the projection-formula mechanism is applied
	to \([Y]\), rather than restricting an ambient numerical inequality.  More
	explicitly, every numerator term carries the same factor \(\delta(Y)\) as the
	restricted top degree; it cancels before the \(n\)-exponent is estimated.
	
	The pointwise canonical lower bound holds on the auxiliary semiabelian
	variety and therefore on the closed points of \(Y\).  If \(X\) contains an
	\(X\)-generic small sequence, compatible lifts give a \(Y\)-generic small
	sequence; Zhang's inequalities then give the restricted height bound and the
	\(O_X(n^{-2})\) correction, as in \cite[Lemma~4.5]{Ku1}.  Applying the
	quadratic arithmetic-volume package to the horizontally semipositive
	restricted bundle on \(Y\), with the preceding restricted higher-order
	intersection estimates, gives the asserted defect by
	\cite[Lemma~4.6]{Ku1}.  The raw factor \([K_n:K]\) cancels under
	Lemma~\ref{lem:field-extension-normalization}; this cancellation is separate
	from, and occurs after, cancellation of \(\delta(Y)\).  Every assertion is
	therefore an assertion on \(Y\).
\end{proof}

The next proposition is the assembly point for the transport of K\"uhne's
archimedean local-trivialization argument.  It replaces the standard
\((\mathbb P^1)^t\)-compactification by the present toric compactification.
It is useful to record the internal ingredients before giving the proof.
\[
\begin{array}{ccl}
	\text{quasi-canonical shift and fan coordinates}
	&\rightsquigarrow&
	\text{local compact-torus Haar slices},\\
	\text{BPS quotient current and toric invariance}
	&\rightsquigarrow&
	\text{general-fan curvature--Haar formula},\\
	\text{products, quotients, and difference morphisms}
	&\rightsquigarrow&
	\text{orbit-measure functoriality},\\
	\text{character-pushout theta factors}
	&\rightsquigarrow&
	\text{Picard-zero sign/dual bookkeeping},\\
	\text{finite Haar slices plus equidistribution}
	&\rightsquigarrow&
	\text{K\"uhne small-box contradiction}.
\end{array}
\]
The proof below follows the published JEMS numbering.  The local
trivialization is \cite[Section~3.5, Lemma~3.3]{Ku1}; subvariety
equidistribution is \cite[Proposition~4.1 and Lemmas~4.2--4.6]{Ku1}; the rank,
Haar-current and domination inputs are \cite[Lemmas~5.1--5.3]{Ku1}; difference
geometry is \cite[Lemma~6.2]{Ku1}; and the final contradiction is
\cite[Proposition~6.1, especially equation~(6.2)]{Ku1}.

\begin{proposition}[K\"uhne local-trivialization transport package]
	\label{prop:kuhne-section4-transport}
	Let \(H\), \(\rho:H\to A_H\), and
	\[
	\bM=\overline H(\bE)\otimes\rho^*\bN_H
	\]
	be as in Theorem~\ref{thm:qcanonical-bogomolov}, with \(\bE\) quasi-canonical.  The
	K\"uhne local-trivialization and operation argument remains valid for this
	quasi-canonical toric compactification in the following precise sense.
	
	\begin{enumerate}[label=\textup{(\alph*)},leftmargin=2em]
		\item For every irreducible subvariety
		\(X\subseteq H_{\overline K}\) satisfying
		\(\muess_{\bM}(X)=\muess_{\bM}(H)\), every \(X\)-generic sequence
		\((x_i)\) satisfying
		\[
		h_{\bM}(x_i)\longrightarrow\muess_{\bM}(H),
		\]
		and every archimedean place \(v\), the Galois orbit measures of \(x_i\) on
		\(X^{an}_{\mathbb C_v}\) converge to the analogue of K\"uhne's
		Proposition~4.1 measure
		\[
		\mu_{X,v}^{\bM}
		=
		\frac{
			c_1(\overline H(\bE)_v|_{\overline X})^{d-d_0}
			\wedge
			c_1(\rho^*\bN_{H,v}|_{\overline X})^{d_0}
		}{
			\overline H(E)^{d-d_0}\cdot(\rho^*N_H)^{d_0}\cdot X
		},
		\]
		where \(d=\dim X\) and \(d_0=\dim\rho(X)\).
		
		\item In the archimedean local trivializations, K\"uhne's Lemmas~5.1--5.3
		admit the following general-fan replacement.  The quasi-canonical shift is
		Lemma~\ref{lem:qcanonical-haar-shift}; the phase coordinates on arbitrary fan
		charts are supplied by Lemma~\ref{lem:fan-chart-compact-torus-coordinates};
		the rank and exceptional-set part is Lemma~\ref{lem:general-fan-rank-slices};
		the toric current computation is supplied by the BPS route through the
		corner-active polyhedral quotient factorization
		Proposition~\ref{input:rank-active-bps-quotient-compatibility},
		Lemma~\ref{lem:toric-quotient-descent-criterion}, and
		Propositions~\ref{prop:rank-active-bps-quotient-compatibility} and
		\ref{input:bps-relative-atomic-tropical-pushforward}.  The local
		Bedford--Taylor expansion is the corner-current route
		Proposition~\ref{prop:toric-corner-current-decomposition}.  Thus one obtains
		Proposition~\ref{input:general-fan-curvature-haar-slice}; the Riemannian
		volume domination is
		Lemma~\ref{lem:general-fan-riemannian-domination}.
		
		\item The preceding description is compatible with quotients by connected
		stabilizers, with products, and with the K\"uhne difference morphisms
		\[
		\alpha_m(x_1,\ldots,x_m)
		=(x_1x_2^{-1},x_2x_3^{-1},\ldots,x_{m-1}x_m^{-1}).
		\]
		Smallness of quotient, product and difference-image sequences for the
		corresponding quasi-canonical target metrics is supplied by
		Lemma~\ref{lem:qcanonical-operation-smallness}.
		The finite Haar-slice part of this compatibility is
		Proposition~\ref{prop:finite-haar-slice-functoriality}.  The toric
		quasi-canonical normalization under these operations is
		Lemma~\ref{lem:toric-qcanonical-operation-normalization}; the Picard-zero
		theta factors only fix the local sign and dual convention, as recorded in
		Lemma~\ref{lem:character-pushout-picard-zero-functoriality} and
		Lemma~\ref{lem:picard-zero-flat-operation-normalization}.  The pushforward
		identity used in \cite[equation~(6.2)]{Ku1} is obtained by pushing
		forward the finite Galois orbit measures and passing to the equidistribution
		limits on the source and on the target.  Hence
		Lemma~\ref{lem:kuhne-small-box-contradiction-finite-slices} applies in the
		trivial-stabilizer case.
	\end{enumerate}
\end{proposition}

\begin{proof}
	We first apply the theta-factor convention stated at the beginning of the
	paper, since this is the only point at which a sign or dual could enter.
	In K\"uhne's Lemma~3.3 the local toric
	coordinate is written, after choosing a rational section on the abelian base,
	as a quotient of the semiabelian rational function by the corresponding theta
	function.  In the notation of Theorem~\ref{thm:local-trivialization}, the same
	coordinate is
	\[
	\psi_{j,m}(y)=\frac{g_m(y)}
	{f_{m,v}(\widetilde\pi_j(y))}.
	\]
	Here \(m\in X^*(T_H)\), \(g_m\) is the rational semiabelian character section,
	and \(f_{m,v}\) is the theta representative of the rational section of the
	character-pushout Picard-zero bundle \(\mathcal H_{H,m}\).  The opposite
	contravariant convention would write the same local expression with the dual
	bundle.  In the present notation this is the harmless replacement
	\(m\mapsto -m\), because
	\(\mathcal H_{H,-m}\simeq\mathcal H_{H,m}^{\vee}\) by
	Lemma~\ref{lem:character-pushout-picard-zero-functoriality}.  Tensor products,
	duals, and pullbacks of these Picard-zero factors are flat by
	Lemma~\ref{lem:picard-zero-flat-operation-normalization}.  Thus the convention
	chosen here fixes the signs in local formulas but changes neither the Chern
	currents nor the normalized measures used below.
	
	We prove (a).  The essential point is that every estimate is made after
	restriction to \(\overline X\) and to an irreducible component of its inverse
	image in the auxiliary tower.  Lemma~\ref{lem:subvariety-restricted-kuhne-package}
	supplies this restricted chain; no estimate on the full ambient
	compactification is substituted for it.  This is not a word-for-word invocation of
	\cite[Proposition~4.1]{Ku1}.  We check its published dependency chain
	\cite[Lemmas~4.2--4.6]{Ku1}: geometric degree growth (Lemma~4.2),
	identification of the limiting mixed measure (Lemma~4.3), height first
	variation (Lemma~4.4), the horizontal semipositivity correction and height
	lower bound (Lemma~4.5), and the arithmetic-volume estimate (Lemma~4.6).
	The normalized height equality
	in (a) is exactly the equality needed when this chain is applied to \(X\).
	The local coordinates
	\(\phi_i^{(j)}\) are replaced by the fan-chart characters
	\(\psi_{j,m}\) of Lemma~\ref{lem:fan-chart-compact-torus-coordinates}; the
	unitary transition condition is exactly the valuation-zero overlap condition in
	Theorem~\ref{thm:local-trivialization}; and the quasi-canonical metric
	\[
	\Psi_{\bE,v}(u)=\Psi_E(u-u_v)-\gamma_v
	\]
	is obtained from the canonical one by translating the tropical variable by
	\(u_v\) and adding a constant.  The constant has no curvature contribution,
	and the adelic normalization of the constants is already included in the
	quasi-canonical normal form.  The restricted geometric degree identities,
	mixed-measure limit, first variation, semipositivity correction, lower bound,
	and arithmetic-volume estimate on the inverse-image component are precisely
	the assertions of Lemma~\ref{lem:subvariety-restricted-kuhne-package}.
	The field-extension convention in every step is
	Lemma~\ref{lem:field-extension-normalization}.
	Consequently every input of \cite[Lemmas~4.2--4.6]{Ku1} has an explicit
	replacement, and their assembly gives the displayed measure.  It is the mixed
	Chern measure of
	\(\overline H(\bE)|_{\overline X}\) and \(\rho^*\bN_H|_{\overline X}\), with
	the same normalization by the algebraic intersection degree as in K\"uhne's
	formula.
	
	We next prove (b).  K\"uhne's Lemma~5.1 is the rank and exceptional-set
	statement for the maps defined by active toric coordinates.  Its general-fan
	replacement is Lemma~\ref{lem:general-fan-rank-slices}, with
	Lemma~\ref{lem:fan-chart-compact-torus-coordinates} providing the phase maps
	to compact tori and Lemma~\ref{lem:qcanonical-haar-shift} translating the
	canonical compact torus to \(\val_v^{-1}(u_v)\).  K\"uhne's Lemma~5.2 is the
	calculation of the Chern current as Haar measure on the compact phase fibers.
	Here this is Proposition~\ref{input:general-fan-curvature-haar-slice}: the BPS
	quotient route gives the atomic tropical pushforward, and
	Lemma~\ref{lem:torus-invariant-atomic-lift} lifts the atomic tropical measure
	to Haar measure on compact phase fibers.  The optional Bedford--Taylor
	calculation through corner currents gives the same local formula.  Finally,
	K\"uhne's Lemma~5.3 uses only finiteness of the slice description; its
	replacement is Lemma~\ref{lem:general-fan-riemannian-domination}.  Thus all
	three local lemmas survive unchanged after replacing the coordinate axes of
	\((\mathbb P^1)^t\) by the chosen fan-chart monomial coordinates.
	
	It remains to prove (c).  Products and inverses are monomial maps in the local
	toric coordinates, and the K\"uhne difference morphism is given on toric
	coordinates by
	\[
	(z^{(1)},\ldots,z^{(m)})
	\longmapsto
	\bigl(z^{(1)}(z^{(2)})^{-1},\ldots,
	z^{(m-1)}(z^{(m)})^{-1}\bigr).
	\]
	Hence the Haar-slice part is functorial by
	Proposition~\ref{prop:finite-haar-slice-functoriality}.  The toric
	quasi-canonical centers are carried by the same integral linear maps by
	Lemma~\ref{lem:toric-qcanonical-operation-normalization}; in particular the
	common center \((u_v,\ldots,u_v)\) is sent to \(0\) by the difference map.
	The residual semiabelian theta factors are exactly the tensor products, duals,
	and pullbacks of the character-pushout Picard-zero factors described in
	Lemma~\ref{lem:character-pushout-picard-zero-functoriality}.  By the first
	paragraph of this proof, every dual which appears under products, inverses or
	difference morphisms is exactly the factor attached to the pulled-back
	character \(-m\), hence to
	\(\mathcal H_{H,-m}\simeq\mathcal H_{H,m}^{\vee}\).  These factors have the
	same flatness property, so the local formulas have the same sign and dual
	conventions as K\"uhne's.
	
	It remains only to identify the source of the pushforward identity.  This
	identity is automatic from equidistribution.  For a generic small product
	sequence \(y_i\in X^m(\overline K)\), the Galois orbit measures satisfy
	\[
	(\alpha_m)_*\delta_{y_i,v}=\delta_{\alpha_m(y_i),v}.
	\]
	Applying part (a) to \(X^m\) and to \(\alpha_m(X^m)\), and passing to the weak
	limit, gives the normalized measure identity used in K\"uhne's
	equation~(6.2).  Thus no separate Chern-measure projection formula is required;
	the finite-slice contradiction is
	Lemma~\ref{lem:kuhne-small-box-contradiction-finite-slices}.
\end{proof}

\begin{lemma}[K\"uhne--Zhang difference geometry]
	\label{lem:kuhne-difference-geometry}
	Let \(X\subseteq H_{\overline K}\) be an irreducible subvariety with trivial
	stabilizer
	\[
	\operatorname{Stab}_{H_{\overline K}}(X)=\{e\}.
	\]
	For all sufficiently large \(m\), the difference morphism
	\[
	\alpha_m:X^m\longrightarrow H^{m-1},\qquad
	(x_1,\ldots,x_m)\longmapsto
	(x_1x_2^{-1},x_2x_3^{-1},\ldots,x_{m-1}x_m^{-1})
	\]
	is generically finite of degree \(1\) onto its image.
\end{lemma}

\begin{proof}
	This is the semiabelian form of the difference-map lemma used by Zhang and
	K\"uhne.  K\"uhne states it as \cite[Lemma~6.2]{Ku1}, referring to
	Zhang's proof in the abelian case.  The proof is purely algebraic and uses only
	the group law and the triviality of the stabilizer: if two general points of
	\(X^m\) have the same successive differences, then their coordinates differ by
	a common translation; for \(m\) large enough this common translation must lie
	in \(\operatorname{Stab}_{H_{\overline K}}(X)\), hence is the identity.
\end{proof}

\begin{proof}[Proof of Proposition~\ref{prop:qbog-dominant-bridge}]
	Let \(Z\subsetneq H_{\overline K}\) and \((z_i)\) satisfy the hypotheses of
	Proposition~\ref{prop:qbog-dominant-bridge}.  We prove that \(Z\) is contained in a
	proper translate of a connected algebraic subgroup.
	
	First perform the standard stabilizer reduction from K\"uhne's proof of
	\cite[Proposition~6.1]{Ku1}.  Let \(B\) be the identity component of
	\(\operatorname{Stab}_{H_{\overline K}}(Z)\), and let
	\(q:H\to H/B\) be the quotient.  If \(q(Z)\) is a point, then \(Z\) is contained
	in a translate of \(B\), and we are done.  Otherwise, replacing \(H\) by
	\(H/B\), \(Z\) by \(q(Z)\), and the sequence by its image, it is enough to
	treat the case
	\[
	\operatorname{Stab}_{H_{\overline K}}(Z)=\{e\}.
	\]
	Lemma~\ref{lem:qcanonical-operation-smallness} shows that the projected
	sequence is small for the quotient quasi-canonical metric whose critical
	vector is the image of the source critical vector.  The projected sequence
	remains generic in the projected subvariety.
	
	Choose \(m\) so that Lemma~\ref{lem:kuhne-difference-geometry} applies to
	\(Z\).  As in K\"uhne's proof, choose a sequence of \(m\)-tuples
	\[
	y_i=(z_{\phi_1(i)},\ldots,z_{\phi_m(i)})\in Z^m
	\]
	which is \(Z^m\)-generic after passing to a subsequence.  Then
	\(\alpha_m(y_i)\) is generic in \(\alpha_m(Z^m)\).  The product sequence is
	small for the exterior-product quasi-canonical metric, and
	Lemma~\ref{lem:qcanonical-operation-smallness} shows that the difference
	sequence is small for the quasi-canonical target metric on \(H^{m-1}\) with
	critical vector (0).
	
	Apply Proposition~\ref{prop:kuhne-section4-transport} to \(Z^m\) and to
	\(\alpha_m(Z^m)\).  The pushforward identity used by K\"uhne in
	\cite[equation~(6.2)]{Ku1} follows directly from these two
	equidistribution statements.  Indeed, for every \(i\), the finite Galois orbit
	measures satisfy
	\[
	(\alpha_m)_*\delta_{y_i,v}=\delta_{\alpha_m(y_i),v}.
	\]
	Since \(y_i\) is generic in \(Z^m\) and \(\alpha_m(y_i)\) is generic in
	\(\alpha_m(Z^m)\), passing to the weak limits gives, for every compactly
	supported continuous test function \(f\) on
	\(\alpha_m(Z^m)^{an}_{\mathbb C_v}\),
	\[
	\int_{(Z^m)^{an}_{\mathbb C_v}}(f\circ\alpha_m)\,d\mu_{Z^m,v}^{\bM}
	=
	\int_{\alpha_m(Z^m)^{an}_{\mathbb C_v}} f\,d\mu_{\alpha_m(Z^m),v}^{\bM}.
	\]
	By the real-analytic Haar description in
	Proposition~\ref{prop:kuhne-section4-transport}, the measure on \(Z^m\) is the
	product of the translated compact-torus Haar measures attached to \(Z\), while
	the measure on the image is described by the corresponding Haar volume on the
	image.  The common quasi-canonical shift \(u_v\) only translates the compact
	tori and does not affect the tangent-space computation, the positivity of the
	volume forms, or the Haar invariance.
	
	The hypotheses of
	Lemma~\ref{lem:kuhne-small-box-contradiction-finite-slices} are therefore
	satisfied.  That lemma is precisely K\"uhne's support and dimension comparison
	from \cite[pp.~2119--2121, Section~5]{Ku1}, rewritten in terms of the finite
	Haar-slice package above.  It gives a contradiction in the
	trivial-stabilizer case.
	
	Thus the trivial-stabilizer quotient case cannot occur unless the image is
	contained in a proper translate of a connected algebraic subgroup.  Lifting
	back through the quotient \(H\to H/B\) gives the same conclusion for the
	original \(Z\).
\end{proof}

\begin{lemma}[Quasi-canonical arithmetic \(T\)-effectivity]
	\label{lem:qcanonical-arith-teff-normalization}
	Let \(X_\Sigma\) be a proper toric variety with principal torus \(T_H\), and
	let \(\bE\) be the quasi-canonical toric metric attached to an ample
	\(T_H\)-effective divisor \(E\).  Write the BGPS quasi-canonical normal form as
	\[
	\psi_{\bE,v}(z)=\Psi_E(z-w_v)-\gamma_v,
	\qquad
	\sum_vn_vw_v=0.
	\]
	Then the \(T_H\)-invariant section corresponding to
	\(0\in\Delta_E\) is a \(\muabs_{\bE}(T_H)\)-small invariant section in the
	sense of Definition~\ref{def:arith-T-effective}; in particular
	\(\bE\) is arithmetically \(T_H\)-effective.  More precisely,
	\[
	\muabs_{\bE}(T_H)=\muess_{\bE}(X_\Sigma)=\sum_vn_v\gamma_v.
	\]
\end{lemma}

\begin{proof}
	By the BGPS quasi-canonical normal form, recorded in
	Proposition~\ref{input:bgps-qcanonical-bridge}, the metric functions have the
	displayed shape.  Let \(s_0=\chi^0s_E\) be the invariant section of \(O(E)\)
	corresponding to \(0\in\Delta_E\).  The convention for toric metric functions
	is
	\[
	\psi_{\bE,v}(\val_v(p))=\log\|s_E(p)\|_v,
	\]
	which is the same convention as in Section~3, where
	\(\Psi_{\bE,v}=-g_{\bE,v}\).  Since \(s_0=s_E\) on the principal torus, we
	have, by density of the principal torus and continuity of the toric metric,
	\[
	\log\|s_0\|_{v,\sup}
	=
	\sup_{z\in N_{H,\BR}}\psi_{\bE,v}(z)
	=
	\sup_{z\in N_{H,\BR}}\Psi_E(z-w_v)-\gamma_v .
	\]
	Because \(E\) is \(T_H\)-effective, \(0\in\Delta_E\), equivalently
	\(\Psi_E\le0\) on \(N_{H,\BR}\), and \(\Psi_E(0)=0\).  Hence
	\[
	\sup_z\Psi_E(z-w_v)=0,
	\qquad
	-\log\|s_0\|_{v,\sup}=\gamma_v .
	\]
	Thus \(s_0\) is a small invariant section in the sense of
	Definition~\ref{def:arith-T-effective}, with \(c_v=\gamma_v\).  The family
	\((\gamma_v)_v\) is finitely supported in the BGPS normal form.
	
	It remains to identify the resulting sum with the toric minimum.  For
	\(x\in\Delta_E\), the local roof function is
	\[
	\vartheta_{\bE,v}(x)
	=
	\inf_z\bigl(\langle z,x\rangle-\Psi_E(z-w_v)+\gamma_v\bigr).
	\]
	Putting \(y=z-w_v\) and using the canonical identity
	\(\Psi_E^\vee(x)=0\) on \(\Delta_E\), this becomes
	\[
	\vartheta_{\bE,v}(x)=\langle w_v,x\rangle+\gamma_v .
	\]
	Therefore
	\[
	\vartheta_{\bE}(x)
	=
	\left\langle \sum_vn_vw_v,x\right\rangle+\sum_vn_v\gamma_v
	=
	\sum_vn_v\gamma_v .
	\]
	The toric successive-minima formula
	\cite[Theorem 3.9 and Corollary 3.10]{BGPS15} gives
	\[
	\muabs_{\bE}(T_H)=\muess_{\bE}(X_\Sigma)
	=
	\max_{x\in\Delta_E}\vartheta_{\bE}(x)
	=
	\sum_vn_v\gamma_v .
	\]
	This also shows that the constants \(c_v=\gamma_v\) have the required global
	sum, so \(\bE\) is arithmetically \(T_H\)-effective.
\end{proof}

\begin{lemma}[Height separation in the quasi-canonical case]
	\label{lem:qbog-height-separation}
	Let
	\[
	\bM=\overline H(\bE)\otimes\rho^*\bN_H
	\]
	be as in Theorem~\ref{thm:qcanonical-bogomolov}.  Assume that the toric part \(\bE\) is in the
	situation of Lemma~\ref{lem:qcanonical-arith-teff-normalization}.  Then
	\[
	\muess_{\bM}(H)=\muabs_{\bE}(T_H),\qquad
	h_{\overline H(\bE)}(x)\ge\muabs_{\bE}(T_H),\qquad
	\hat h_{\bN_H}(\rho(x))\ge0
	\]
	for all \(x\in H(\overline K)\).
\end{lemma}

\begin{proof}
	By Lemma~\ref{lem:qcanonical-arith-teff-normalization}, \(\bE\) is
	arithmetically \(T_H\)-effective.  Repeating the invariant-section argument in
	Lemma~\ref{lem:arith-teff-height-lower}, but without adding the abelian factor,
	gives
	\[
	h_{\overline H(\bE)}(x)\ge\muabs_{\bE}(T_H)
	\]
	for every \(x\in H(\overline K)\).  The canonical height attached to the
	symmetric ample line bundle \(\bN_H\) is nonnegative, so
	\[
	\hat h_{\bN_H}(\rho(x))\ge0.
	\]
	Finally, Proposition~\ref{prop:section4-minima-formula}, applied to \(H\) and
	\(\bE\), gives
	\[
	\muess_{\bM}(H)=\muabs_{\bE}(T_H).
	\]
\end{proof}

\begin{proposition}[Base reduction for quasi-canonical Bogomolov]
	\label{prop:qbog-base-reduction}
	Assume that the toric part \(\bE\) is in the situation of
	Lemma~\ref{lem:qcanonical-arith-teff-normalization}.  Assume also
	Proposition~\ref{prop:qbog-dominant-bridge} for the semiabelian subvarieties
	obtained by pulling back abelian subvarieties of \(A_H\).  Then
	Theorem~\ref{thm:qcanonical-bogomolov} holds for \((H,\bM)\).
\end{proposition}

\begin{proof}
	Let \(Z\subsetneq H_{\overline K}\) be irreducible and let \((z_i)\) be a
	\(Z\)-generic sequence with
	\[
	h_{\bM}(z_i)\to\muess_{\bM}(H).
	\]
	By
	Lemma~\ref{lem:qbog-height-separation},
	\[
	\muess_{\bM}(H)=\muabs_{\bE}(T_H),\qquad
	h_{\overline H(\bE)}(x)\ge\muabs_{\bE}(T_H),\qquad
	\hat h_{\bN_H}(\rho(x))\ge0 .
	\]
	Since
	\[
	h_{\bM}(z_i)
	=h_{\overline H(\bE)}(z_i)+\hat h_{\bN_H}(\rho(z_i))
	\]
	and the first summand is bounded below by \(\muabs_{\bE}(T_H)\), while the
	second is nonnegative, we have
	\[
	\hat h_{\bN_H}(\rho(z_i))\to0.
	\]
	Let \(C=\rho(Z)\).  The projected sequence is \(C\)-generic: if
	\(W\subsetneq C\) is closed, then \(Z\cap\rho^{-1}(W)\) is a proper closed
	subset of \(Z\), so \(z_i\) is eventually outside it.
	
	By the abelian Bogomolov theorem, Theorem~\ref{input:abelian-bogomolov}, \(C\) is
	a torsion translate \(\tau+A_0\) of an abelian subvariety \(A_0\subseteq A_H\).
	If \(A_0\neq A_H\), then
	\[
	\rho^{-1}(\tau+A_0)
	\]
	is a proper translate of the connected semiabelian subgroup
	\(\rho^{-1}(A_0)\), and it contains \(Z\).  This proves the desired conclusion
	in the proper-image case.
	
	It remains to treat the case \(A_0=A_H\).  In the general reduction, after
	translating by a torsion lift of \(\tau\) and replacing \(H\) by
	\(\rho^{-1}(A_0)\), this is exactly the dominant case covered by
	Proposition~\ref{prop:qbog-dominant-bridge}.  Therefore \(Z\) is contained in a
	proper translate of a connected algebraic subgroup of \(H_{\overline K}\).
\end{proof}

\begin{proof}[Proof of Theorem~\ref{thm:qcanonical-bogomolov}]
	After adding an \(M_K\)-constant, the quasi-canonical toric metric is in the
	normal form of Lemma~\ref{lem:qcanonical-arith-teff-normalization}.  The
	dominant case is Proposition~\ref{prop:qbog-dominant-bridge}, proved above
	from the K\"uhne transport package.  Proposition~\ref{prop:qbog-base-reduction}
	then proves the asserted Bogomolov property for every irreducible proper
	subvariety \(Z\subsetneq H_{\overline K}\).
\end{proof}

\begin{proposition}[Bogomolov property implies the strict upgrade]
	\label{prop:qbog-to-qsec}
	Assume Theorem~\ref{thm:qcanonical-bogomolov} for \((H,\bM)\).  Then every strict
	\(\bM\)-small sequence in \(H(\overline K)\) is \(H\)-generic.
	Consequently, if the generic equidistribution statement of
	Theorem~\ref{thm:qcanonical-compression} applies to \(\bM\), every strict
	\(\bM\)-small sequence equidistributes toward the canonical probability
	measure.
\end{proposition}

\begin{proof}
	Let \((x_i)\) be strict and \(\bM\)-small.  If it is not \(H\)-generic, then
	after passing to a subsequence and replacing its Zariski closure by an
	irreducible component, there is an irreducible proper subvariety
	\(Z\subsetneq H_{\overline K}\) such that the subsequence is \(Z\)-generic.
	The subsequence is still \(\bM\)-small in the ambient sense, so
	\[
	h_{\bM}(x_i)\longrightarrow \muess_{\bM}(H).
	\]
	Theorem~\ref{thm:qcanonical-bogomolov} implies that \(Z\) is contained in a proper translate
	of a connected algebraic subgroup of \(H_{\overline K}\), contradicting
	strictness.  Hence every strict \(\bM\)-small sequence is \(H\)-generic, and
	the generic equidistribution theorem applies.
\end{proof}

\begin{theorem}[Quasi-canonical strong equidistribution]
	\label{thm:qcanonical-strong-equidistribution}
	\label{input:qsec}
	Let \(\bM\) be an adelic line bundle on a semiabelian variety \(H\) of the form
	\[
	\bM=\overline H(\bE)\otimes \rho^*\bN_H,
	\]
	where \(\bE\) is the quasi-canonical toric metric attached to an ample
	\(T_H\)-effective divisor and \(\bN_H\) is canonical on the abelian quotient.
	Then every strict \(\bM\)-small sequence equidistributes at every archimedean
	place toward the canonical probability measure.
\end{theorem}

\begin{proof}
	By Theorem~\ref{thm:qcanonical-bogomolov} and
	Proposition~\ref{prop:qbog-to-qsec}, every strict \(\bM\)-small sequence is
	\(H\)-generic.  The quasi-canonical generic equidistribution theorem
	Theorem~\ref{thm:qcanonical-compression} then applies and gives convergence to
	the canonical probability measure at every archimedean place.
\end{proof}

\begin{corollary}[Quasi-canonical Bogomolov and strictness]
	\label{cor:kuhne-transport-qbog-qsec}
	For the quasi-canonical semiabelian metrics under consideration and for the
	semiabelian subvarieties obtained after the base reduction in
	Proposition~\ref{prop:qbog-base-reduction},
	Theorem~\ref{thm:qcanonical-bogomolov} holds.  Consequently the strict
	quasi-canonical equidistribution theorem
	Theorem~\ref{thm:qcanonical-strong-equidistribution} holds whenever the
	generic equidistribution theorem Theorem~\ref{thm:qcanonical-compression}
	applies.
\end{corollary}

\begin{proof}
	Proposition~\ref{prop:kuhne-section4-transport} is the source of the proof of
	the dominant bridge, Proposition~\ref{prop:qbog-dominant-bridge}.  After the
	height-separation normalization in
	Lemma~\ref{lem:qcanonical-arith-teff-normalization},
	Proposition~\ref{prop:qbog-base-reduction} proves
	Theorem~\ref{thm:qcanonical-bogomolov}.  Finally
	Proposition~\ref{prop:qbog-to-qsec} upgrades ambient small strict sequences to
	generic sequences, and Theorem~\ref{thm:qcanonical-compression} gives their
	equidistribution.
\end{proof}

\begin{corollary}[Restricted strict equidistribution for the paper's metric class]
	\label{cor:kuhne-transport-paper-metric-class}
	Let
	\[
	\bL=\overline G(\bD)\otimes\overline\pi^*\bN
	\]
	be an ambient adelic line bundle whose toric metric \(\bD\) satisfies the
	standing hypotheses of the Bogomolov part of the paper:
	\(D\) is ample and \(T\)-effective, and \(\bD\) is semipositive, monocritical,
	and arithmetically \(T\)-effective.  Then for every minimal translate
	\(Y=a+B\) arising in the restriction package, the restricted
	quasi-canonical line bundle \(\bL'_{a,B}\) satisfies the strict
	quasi-canonical equidistribution theorem,
	Theorem~\ref{thm:qcanonical-strong-equidistribution}.
\end{corollary}

\begin{proof}
	Let \(X\subseteq G_{\overline K}\) be irreducible and let \(Y=a+B\) be the
	minimal translate used in the restriction argument.  The
	monocritical-to-quasi-canonical replacement
	Proposition~\ref{prop:monocritical-qcanonical-replacement}, together with the
	restriction construction in Proposition~\ref{prop:restriction-metric}, produces
	the restricted quasi-canonical metric \(\bL'_{a,B}\) on \(B\).  Smallness for
	the original metric is transferred to this restricted quasi-canonical metric by
	Corollary~\ref{cor:restricted-smallness}.  Applying
	Proposition~\ref{prop:kuhne-section4-transport} to \(\bL'_{a,B}\) gives
	Theorem~\ref{thm:qcanonical-strong-equidistribution} by
	Corollary~\ref{cor:kuhne-transport-qbog-qsec}.
	This is exactly the strict equidistribution statement used in
	Theorem~\ref{thm:restriction-package}.
\end{proof}

\begin{remark}[Scope of the phrase ``general metric'']
	\label{rem:scope-general-metric}
	The preceding corollary is the precise sense in which the K\"uhne-operation
	route proves the Bogomolov theorem for a general metric in this paper.  The
	word ``general'' refers to the paper's controlled class of toric metrics:
	semipositive, monocritical, and arithmetically \(T\)-effective metrics on an
	ample \(T\)-effective toric divisor, together with their restricted
	quasi-canonical replacements.  It does not mean an arbitrary semipositive
	toric metric.  Extending the argument to such a larger class would require a
	new replacement or deformation theorem beyond
	Proposition~\ref{prop:monocritical-qcanonical-replacement}.
\end{remark}

\begin{remark}
	Theorem~\ref{thm:qcanonical-compression} supplies the corresponding generic
	equidistribution statement, but not this strict statement.  This distinction is
	essential in the restriction argument below: an \(X\)-generic sequence in a
	proper subvariety \(X\subsetneq Y\) is not \(Y\)-generic, although it is strict
	in the minimal translate \(Y\).  The strict upgrade is supplied here by the
	Bogomolov theorem in the quasi-canonical metric, namely
	Theorem~\ref{thm:qcanonical-bogomolov}, through
	Proposition~\ref{prop:qbog-to-qsec}.  The dominant case of that Bogomolov
	theorem is Proposition~\ref{prop:qbog-dominant-bridge}, whose proof is the
	direct K\"uhne local-trivialization transport package,
	Proposition~\ref{prop:kuhne-section4-transport}.  The relative toric
	fiber-specialness route in the appendix is therefore not part of the main
	proof.  The reduction from the original monocritical metric to the
	quasi-canonical restricted metric is supplied by
	Proposition~\ref{prop:monocritical-qcanonical-replacement} and by the
	restriction package below.
\end{remark}

\begin{proposition}[Strict equidistribution for the ambient metric]
	Let \((x_i)\) be a strict \(\bL\)-small
	sequence in \(G(\overline{K})\).  Then, for every archimedean place \(v\), the Galois
	orbit measures \(\delta_{x_i,v}\) converge to the canonical probability
	measure \(\mu_{\bL,v}\).  If \(\mathbf u=(u_v)_v\) is the critical point of
	\(\bD\), then \(\mu_{\bL,v}\) is supported on
	\[
	S_{u_v}:=S_v\cdot e^{-u_v},
	\]
	where \(S_v\) is the maximal compact subgroup of \(G(\mathbb C_v)\).
\end{proposition}

\begin{proof}
	Normalize the metric by an \(M_K\)-constant so that \(\mu_G=0\).
	Proposition~\ref{prop:monocritical-qcanonical-replacement}
	constructs the quasi-canonical toric metric
	\[
	\Psi_{\bD',v}(z)=\Psi_D(z-u_v)
	\]
	and proves that every \(\bL\)-small sequence is also small for
	\[
	\bL'=\overline G(\bD')\otimes\pi^*\bN.
	\]
	Strictness is purely algebraic, so \((x_i)\) is strict as a \(\bL'\)-small
	sequence.  Theorem~\ref{thm:qcanonical-strong-equidistribution} applies to
	\(\bL'\).  The limiting measure is
	the Haar measure on \(S_{u_v}\), which is the canonical measure associated with
	the original monocritical metric because
	Proposition~\ref{prop:monocritical-kr} proves convergence of the valuation
	measures to \(\delta_{u_v}\).
\end{proof}

\subsection{Minimal Special Envelopes}

\begin{lemma}[Minimal envelope]\label{lem:minimal-envelope}
	Let \(X\subseteq G_{\overline{K}}\) be irreducible.  There is a unique minimal translate
	\[
	Y=a+B
	\]
	of a connected algebraic subgroup of \(G_{\overline{K}}\) containing \(X\).  If
	\(X\neq Y\), then every \(X\)-generic sequence is strict as a sequence in
	\(Y\).
\end{lemma}

\begin{proof}
	Among all translates of connected algebraic subgroups containing \(X\), choose
	one, say \(Y=a+B\), of minimal dimension.  Such a translate exists because
	\(G\) itself is one candidate.  If \(Y'=a'+B'\) is another minimal candidate,
	then \(Y\cap Y'\) is nonempty and is a translate of the algebraic subgroup
	\(B\cap B'\), possibly with finitely many connected components.  Since \(X\)
	is irreducible and contained in this intersection, it is contained in one
	connected component, which is a translate of \((B\cap B')^0\).  Minimality
	forces this component to have the same dimension as both \(Y\) and \(Y'\),
	and hence \(Y=Y'\).  This proves existence and uniqueness.
	
	Assume \(X\neq Y\).  If \(U\subsetneq Y\) is a translate of a connected
	subgroup, then \(X\cap U\) is a proper closed subset of \(X\); otherwise \(U\)
	would be a smaller translate containing \(X\).  An \(X\)-generic sequence
	therefore has only finitely many terms in \(U\), which is exactly strictness in
	\(Y\).
\end{proof}

Let \(X\subseteq G_{\overline{K}}\) be irreducible and let
\[
Y=a+B
\]
be the minimal translate of a connected algebraic subgroup containing \(X\).
Put
\[
T_B=B\cap T,\qquad A_B=\pi(B).
\]
Let
\[
\iota_N:N_{B,\BR}\longrightarrow N_{\BR}
\]
be the injection of cocharacter spaces induced by \(T_B\hookrightarrow T\).

\subsection{The Abelian Part Becomes Torsion}

\begin{theorem}[Bogomolov theorem for abelian varieties]
	\label{input:abelian-bogomolov}
	Let \(A\) be an abelian variety over \(\overline K\), equipped with a
	Neron--Tate height attached to an ample symmetric line bundle.  If an
	irreducible subvariety \(Z\subset A\) carries a \(Z\)-generic sequence whose
	Neron--Tate heights tend to \(0\), then \(Z\) is a torsion subvariety.  In
	particular, if a translate \(a+A_0\) of an abelian subvariety contains such a
	generic small sequence, then \(a+A_0\) is a torsion translate of \(A_0\).
	This is Zhang's Bogomolov theorem for abelian varieties
	\cite[Theorem~2]{Zha}, building on Ullmo's curve case \cite{Ull}.
\end{theorem}

\begin{lemma}[Torsion base representative]\label{lem:torsion-base}
	Assume
	\[
	\muess_{\bL}(X)=\mu_G.
	\]
	Then, after replacing \(a\) by \(a+b\) for a suitable \(b\in B(\overline{K})\), one may
	assume that \(\pi(a)\) is torsion in \(A(\overline{K})\).
\end{lemma}

\begin{proof}
	Choose an \(X\)-generic sequence \((x_i)\) with
	\[
	h_{\bL}(x_i)\to\mu_G.
	\]
	By Proposition~\ref{prop:section4-minima-formula} and the height decomposition
	\[
	h_{\bL}(x_i)
	=h_{\overline G(\bD)}(x_i)+\hat h_{\bN}(\pi(x_i)),
	\]
	with
	\[
	h_{\overline G(\bD)}(x_i)\ge\mu_G,
	\qquad
	\hat h_{\bN}(\pi(x_i))\ge0,
	\]
	we obtain
	\[
	\hat h_{\bN}(\pi(x_i))\to0.
	\]
	
	Let \(C=\overline{\pi(X)}\subseteq \pi(a)+A_B\).  The projected sequence is
	\(C\)-generic: the inverse image in \(X\) of every proper closed subset of
	\(C\) is a proper closed subset of \(X\).  Theorem~\ref{input:abelian-bogomolov}
	therefore shows that
	\[
	C=\tau+A_0
	\]
	for a torsion point \(\tau\) and an abelian subvariety \(A_0\subseteq A_B\).
	If \(A_0\subsetneq A_B\), the identity component of
	\(B\cap\pi^{-1}(A_0)\) is a proper connected algebraic subgroup of \(B\), and
	a translate of it contains \(X\).  This contradicts the minimality of
	\(Y=a+B\).  Hence \(A_0=A_B\), so
	\[
	\pi(a)+A_B=C=\tau+A_B.
	\]
	Thus there is a torsion
	point \(\tau\in A(\overline K)\) such that
	\[
	\pi(a)+A_B=\tau+A_B.
	\]
	Since \(\pi(B)=A_B\), we can choose \(b\in B(\overline K)\) with
	\(\pi(b)=\tau-\pi(a)\).  Then \(\pi(a+b)=\tau\) is torsion, and replacing
	\(a\) by \(a+b\) does not change \(Y=a+B\).
\end{proof}

\subsection{Tropical Position of the Envelope}

From now on we choose \(a\) as in Lemma \ref{lem:torsion-base}.  Choose a
torsion point \(q\in G(\overline{K})\) above \(\pi(a)\).  Then there is a unique
\(t_a\in T(\overline{K})\) such that
\[
a=q\,t_a.
\]
For every place \(v\), put
\[
\alpha_v:=\val_v(t_a)\in N_{\BR}.
\]

\begin{lemma}[Valuations of translated subtori]
	\label{lem:translated-subtorus-valuations}
	Let \(Y=q\,t_aB\) with \(q\) torsion over the abelian quotient and
	\(t_a\in T(\overline K)\).  For the valuation map on the toric part used in
	Proposition~\ref{prop:section4-valuation-height-package}, the valuation of
	every algebraic point of \(Y\) at a place \(v\) belongs to
	\[
	\alpha_v+\iota_N(N_{B,\BR})\subseteq N_\BR,
	\qquad
	\alpha_v=\val_v(t_a).
	\]
	Equivalently, after translating \(Y\) by \(q\,t_a\) back to \(B\), the ambient
	valuation differs from the intrinsic \(T_B\)-valuation by the fixed affine
	shift \(\alpha_v\).
\end{lemma}

\begin{proof}
	It is enough to prove the assertion on the dense torus, since the valuation map
	used here is defined by local toric coordinates and is compatible with
	restriction to the toric part.  Write a point of \(Y\) as
	\[
	q\,t_a\,\iota(s),\qquad s\in T_B(\overline K).
	\]
	By Lemma~\ref{lem:torsion-lift-zero-valuation}, the torsion lift \(q\) has
	zero toric valuation.  The valuation on a split torus is additive under
	multiplication and functorial for the homomorphism
	\(\iota:T_B\hookrightarrow T\).  Hence
	\[
	\val_v(q\,t_a\,\iota(s))
	=
	\val_v(t_a)+\iota_N(\val_{B,v}(s))
	=
	\alpha_v+\iota_N(\val_{B,v}(s)).
	\]
	Since \(\val_{B,v}(s)\in N_{B,\BR}\), the result follows.  The same computation
	in local trivialization charts shows that the statement is independent of the
	chosen chart.
\end{proof}

\begin{lemma}[Critical vector lies in the tropical envelope]\label{lem:tropical-envelope}
	Assume \(\muess_{\bL}(X)=\mu_G\).  Then for every place \(v\),
	\[
	u_v-\alpha_v\in \iota_N(N_{B,\BR}).
	\]
	Equivalently, there is a unique \(w_v\in N_{B,\BR}\) such that
	\[
	\iota_N(w_v)+\alpha_v=u_v.
	\]
\end{lemma}

\begin{proof}
	Let \((x_i)\) be an \(X\)-generic \(\bL\)-small sequence.
	Proposition~\ref{prop:monocritical-kr} gives
	adelic KR convergence of the ambient valuation measures:
	\[
	(\val_v)_*\delta_{x_i,v}\longrightarrow\delta_{u_v}
	\]
	for every \(v\).
	
	On the other hand, because \(x_i\in Y=a+B=q\,t_aB\), the toric valuation of
	every point in the Galois orbit of \(x_i\) lies in the closed affine subspace
	\[
	\alpha_v+\iota_N(N_{B,\BR})\subseteq N_{\BR}.
	\]
	This is Lemma~\ref{lem:translated-subtorus-valuations}; its independence of
	the local trivialization chart is part of
	Proposition~\ref{prop:section4-valuation-height-package}.
	Since a weak limit of probability measures supported on a closed set is again
	supported on that closed set, the limit point \(u_v\) belongs to
	\[
	\alpha_v+\iota_N(N_{B,\BR}).
	\]
	This proves the lemma.
\end{proof}

\subsection{The Restricted Quasi-canonical Model}

Let \(\Sigma_B\) be the fan on \(N_{B,\BR}\) obtained from the inverse images
\[
\iota_N^{-1}(\sigma),\qquad \sigma\in\Sigma,
\]
after removing redundancies and passing to the associated normal toric
variety.  Denote this toric variety by \(X_{\Sigma_B}\).

\begin{lemma}[Toric normalization of a subtorus closure]
	\label{lem:subtorus-toric-normalization}
	The toric variety \(X_{\Sigma_B}\) is the normalization of the closure of
	\(T_B\) in \(X_\Sigma\).  The inclusion \(T_B\hookrightarrow T\) extends to a
	toric morphism
	\[
	\varphi_B:X_{\Sigma_B}\longrightarrow X_\Sigma
	\]
	whose image is this closure and whose induced morphism to the image is finite.
	In particular, if \(D\) is an ample toric Cartier divisor on \(X_\Sigma\), then
	\(\varphi_B^*D\) is ample on \(X_{\Sigma_B}\).  If \(D\) is \(T\)-effective,
	then \(\varphi_B^*D\) is \(T_B\)-effective.
\end{lemma}

\begin{proof}
	The inverse images of the cones of \(\Sigma\) form a fan because inverse image
	commutes with taking faces and intersections.  On an affine toric chart
	\(U_\sigma=\operatorname{Spec} K[\sigma^\vee\cap M]\), the closure of
	\(T_B\cap U_\sigma\) has coordinate semigroup given by the image of
	\(\sigma^\vee\cap M\) in the character lattice \(M_B\).  Since \(T_B\subseteq T\)
	is a saturated subtorus, the saturation of this image is
	\[
	\bigl(\iota_N^{-1}\sigma\bigr)^\vee\cap M_B .
	\]
	Thus the normalization of the affine closure is
	\(\operatorname{Spec} K[(\iota_N^{-1}\sigma)^\vee\cap M_B]\).  These affine
	normalizations glue over the fan \(\Sigma_B\), giving \(X_{\Sigma_B}\).  This is
	the standard subtorus-closure construction in toric geometry; see
	\cite[Section 1.4]{WF}.  The induced morphism from the normal variety
	\(X_{\Sigma_B}\) to its image is the normalization map, hence finite.
	
	The restriction of an ample line bundle to a closed subvariety is ample, and
	the finite pullback of an ample line bundle is ample.  Therefore
	\(\varphi_B^*D\) is ample.  Finally, a non-zero \(T\)-eigensection of \(D\)
	restricts on the dense torus and pulls back to a non-zero \(T_B\)-eigensection,
	so \(T\)-effectivity descends to \(T_B\)-effectivity.
\end{proof}

The map
\[
T_B\longrightarrow T,\qquad t\longmapsto t_a\,\iota(t),
\]
extends, by composing \(\varphi_B\) with the action of \(t_a\) on \(X_\Sigma\),
to a morphism
\[
\varphi_{a,B}:X_{\Sigma_B}\longrightarrow X_\Sigma .
\]
Let
\[
\overline B:=B\times^{T_B}X_{\Sigma_B}.
\]
Translation by \(q\,t_a\) identifies \(\overline B\) with the normalization of
the closure of \(Y\) in \(\overline G\).

\begin{proposition}[Restriction of the metric]\label{prop:restriction-metric}
	With the above notation,
	\[
	\tau_a^*\bL
	=
	\overline B(\bD_{a,B})\otimes\pi_B^*\bN_B,
	\]
	where \(\bN_B\) is the canonical metrized line bundle obtained by restricting
	\(\bN\) to \(A_B\), and \(\bD_{a,B}\) is the toric metrized divisor on
	\(X_{\Sigma_B}\) whose local metric functions are
	\[
	\Psi_{\bD_{a,B},v}(z)
	=
	\Psi_{\bD,v}\bigl(\iota_N(z)+\alpha_v\bigr),
	\qquad z\in N_{B,\BR}.
	\]
	Moreover, the quasi-canonical metric obtained in
	Proposition~\ref{prop:monocritical-qcanonical-replacement} restricts to the
	quasi-canonical toric metric
	\[
	\Psi_{\bD'_{a,B},v}(z)
	=
	\Psi_{D_B}(z-w_v),
	\]
	where \(w_v\) is defined by Lemma \ref{lem:tropical-envelope}.
	The divisor \(D_B\) is ample and \(T_B\)-effective, and
	\[
	\sum_v n_vw_v=0.
	\]
	Hence \(\bD'_{a,B}\) is an admissible quasi-canonical toric metric for the
	restricted semiabelian variety \(B\).
\end{proposition}

\begin{proof}
	The algebraic statement follows from the construction of the normalized closure.
	On the toric fiber, the map is \(t\mapsto t_a\iota(t)\), so the underlying
	toric divisor is the pullback of \(D\) to \(X_{\Sigma_B}\).  Its support
	function is
	\[
	\Psi_{D_B}(z)=\Psi_D(\iota_N(z)).
	\]
	By Lemma~\ref{lem:subtorus-toric-normalization}, the induced map from
	\(X_{\Sigma_B}\) to the closure of \(T_B\) is finite.  Therefore the pullback
	of the ample toric divisor \(D\) is ample on \(X_{\Sigma_B}\).  The same lemma
	shows that a \(T\)-effective section of \(D\) pulls back to a non-zero
	\(T_B\)-eigensection, so \(D_B\) is \(T_B\)-effective.
	
	The translation by the torsion lift \(q\) has zero toric valuation at every
	place by Lemma~\ref{lem:torsion-lift-zero-valuation}, and hence does not change
	the metric.  The translation by \(t_a\) shifts the valuation by \(\alpha_v\).  This
	gives
	\[
	\Psi_{\bD_{a,B},v}(z)
	=
	\Psi_{\bD,v}\bigl(\iota_N(z)+\alpha_v\bigr).
	\]
	
	Since \(\pi(a)=\pi(q)\) is torsion, translation by \(\pi(a)\) preserves the
	canonical metric on \(A\).  Hence the abelian factor is simply
	\[
	\pi_B^*\bN_B.
	\]
	
	For the quasi-canonical metric, Proposition
	\ref{prop:monocritical-qcanonical-replacement} defines
	\[
	\Psi_{\bD',v}(y)=\Psi_D(y-u_v).
	\]
	Pulling this back gives
	\[
	\Psi_{\bD'_{a,B},v}(z)
	=
	\Psi_D(\iota_N(z)+\alpha_v-u_v).
	\]
	By Lemma \ref{lem:tropical-envelope},
	\[
	\alpha_v-u_v=-\iota_N(w_v),
	\]
	and therefore
	\[
	\Psi_{\bD'_{a,B},v}(z)
	=
	\Psi_D(\iota_N(z-w_v))
	=
	\Psi_{D_B}(z-w_v).
	\]
	This is precisely the quasi-canonical metric on \(X_{\Sigma_B}\) with critical
	vector \((w_v)_v\).  Finally,
	\[
	\sum_vn_v\iota_N(w_v)
	=
	\sum_vn_v(u_v-\alpha_v).
	\]
	The critical vector \((u_v)_v\) satisfies the product formula, and
	\((\alpha_v)_v\) is the valuation vector of the algebraic point
	\(t_a\in T(\overline K)\), so both sums on the right vanish.  Since
	\(\iota_N\) is injective, \(\sum_vn_vw_v=0\).
\end{proof}

\begin{corollary}[Smallness after restriction]\label{cor:restricted-smallness}
	Let \((x_i)\) be an \(X\)-generic \(\bL\)-small sequence and assume
	\(\muess_{\bL}(X)=\mu_G\).  Then the corresponding sequence in \(B(\overline{K})\) is
	small for
	\[
	\bL'_{a,B}:=\overline B(\bD'_{a,B})\otimes\pi_B^*\bN_B.
	\]
\end{corollary}

\begin{proof}
	Write \(x_i=a b_i\) with \(b_i\in B(\overline{K})\).  The base part satisfies
	\[
	\hat h_{\bN_B}(\pi_B(b_i))=\hat h_{\bN}(\pi(x_i))\to0,
	\]
	because \(\pi(a)\) is torsion.  By
	Proposition~\ref{prop:monocritical-kr},
	\[
	(\val_v)_*\delta_{x_i,v}\to\delta_{u_v}.
	\]
	Equivalently, after subtracting the fixed translation \(\alpha_v\), the
	valuation measures of \(b_i\) in \(B\) converge to \(\delta_{w_v}\).
	The quasi-canonical metric \(\bD'_{a,B}\) has critical vector \(w_v\), so the
	same Lipschitz/KR argument as in
	Proposition~\ref{prop:monocritical-qcanonical-replacement} gives
	\[
	h_{\overline B(\bD'_{a,B})}(b_i)\to0.
	\]
	Thus
	\[
	h_{\bL'_{a,B}}(b_i)\to0
	=\muess_{\bL'_{a,B}}(\overline B),
	\]
	which proves smallness.
\end{proof}

\subsection{Strict Equidistribution on the Envelope}

\begin{proposition}[Restricted strict equidistribution]\label{prop:restricted-eq}
	Assume \(\muess_{\bL}(X)=\mu_G\).  If \(X\neq Y\), then an \(X\)-generic
	\(\bL\)-small sequence in \(X(\overline{K})\) is strict in \(Y\), and its
	archimedean Galois orbit measures converge to a probability measure whose
	support is Zariski dense in \(Y\).
\end{proposition}

\begin{proof}
	Strictness in \(Y\) is Lemma \ref{lem:minimal-envelope}.
	
	By Corollary \ref{cor:restricted-smallness}, after translating the sequence to
	\(B\), it is small for the quasi-canonical line bundle \(\bL'_{a,B}\).  Hence
	Theorem~\ref{thm:qcanonical-strong-equidistribution} gives equidistribution on
	\(B\); the hypotheses of that theorem are verified for \(\bL'_{a,B}\) by
	Proposition~\ref{prop:restriction-metric}.  Translating back by \(a\) gives
	the required convergence on \(Y\).
	
	At an archimedean place \(v\), the limiting support is a translate of the
	maximal compact subgroup of \(B(\mathbb C_v)\), with toric valuation \(w_v\).
	The maximal compact subgroup of a complex semiabelian variety is Zariski dense:
	its image in the abelian quotient is the whole complex torus
	\(A_B(\mathbb C_v)\), and its intersection with the torus part is
	\((S^1)^{\dim T_B}\), which is Zariski dense in \(T_B\).  Therefore its
	translate is Zariski dense in \(Y\).
\end{proof}

\subsection{Special Points on the Envelope}

\begin{proposition}[BGPS bridge]
	\label{input:bgps-qcanonical-bridge}
	Let \(\bE\) be a quasi-\allowbreak canonical semipositive toric metrized divisor on a
	proper toric variety with principal torus \(T_H\).  Then \(\bE\) is
	monocritical.  More precisely, by the BGPS characterization of
	quasi-canonical toric metrics, the metric functions may be written in the form
	\[
	\psi_{\bE,v}(z)=\Psi_E(z-w_v)-\gamma_v
	\]
	for a centered adelic vector \((w_v)_v\) and constants \((\gamma_v)_v\);
	this vector is the critical vector of \(\bE\).  The implication from this
	normal form to monocriticality is the BGPS monocritical criterion.  We use
	\cite[Propositions 4.16 and 5.3]{BGPS19} for these two facts.
	
	If \(F/K\) is a finite extension, the scalar extension \(\bE_F\) is again
	monocritical, and the critical vector at a place \(w\) of \(F\) over \(v\) is
	the same vector \(w_v\).  This is the scalar-extension compatibility in
	\cite[Proposition 4.17]{BGPS19}.  Consequently, a statement written below in
	terms of \(T_H(\overline K)\) and \(\val_v(T_H(\overline K))\) is interpreted
	by applying the cited BGPS result over a sufficiently large finite field of
	definition and then viewing the resulting point over \(\overline K\).
\end{proposition}

\begin{theorem}[BGPS toric Bogomolov property]
	\label{input:bgps-toric-bogomolov}
	Let \(X_\Sigma\) be a proper toric variety with principal torus \(T\), and let
	\(\bE\) be a monocritical toric metrized \(\mathbb R\)-divisor with big
	underlying divisor.  If an irreducible subvariety
	\(V\subset T_{\overline K}\) satisfies
	\[
	\muess_{\bE}(V)=\muess_{\bE}(X_\Sigma),
	\]
	then \(V\) is a translate of a subtorus of \(T_{\overline K}\).  Moreover, for
	a translate \(U=T_0\cdot p\), with associated affine valuation spaces
	\[
	A_{U,w}=\val_w(p)+N_{T_0,\mathbb R},
	\]
	the translate \(U\) is \(\bE\)-special if and only if the critical vector
	satisfies \(u_{v(w)}\in A_{U,w}\) for every place \(w\) after a finite field of
	definition.  This is the Bogomolov property and special-translate criterion of
	\cite[Theorem 5.12 and Proposition 5.14]{BGPS19}.
\end{theorem}

\begin{remark}[Why the toric theorem is not yet the semiabelian theorem]
	\label{rem:toric-bogomolov-bridge-gap}
	Theorem~\ref{input:bgps-toric-bogomolov} is the correct toric source for a
	possible proof of Theorem~\ref{thm:qcanonical-bogomolov}.  It does not by
	itself prove Theorem~\ref{thm:qcanonical-bogomolov}, because a subvariety of a semiabelian variety need
	not be contained in one torus fiber, and the quotient by the abelian part has
	to be controlled at the same time.  A non-circular proof of
	Theorem~\ref{thm:qcanonical-bogomolov} has to combine the abelian Bogomolov theorem for
	\(\rho(Z)\), the toric BGPS theorem on the fiber directions, and a gluing
	statement saying that the resulting local translated-subtorus conditions come
	from a global translate of a connected semiabelian subgroup.
\end{remark}

\begin{proposition}[Toric special-point criterion]
	\label{input:toric-special-point}
	Let \(\bE\) be a quasi-canonical toric metrized divisor on a torus \(T_H\), with
	centered critical vector \((w_v)_v\).  Then \(T_H(\overline K)\) contains a
	point whose height equals the toric absolute minimum if and only if
	\[
	w_v\in \val_v(T_H(\overline K))\otimes_{\BZ}\BQ
	\]
	for every place \(v\).  When this condition holds, such a point may be chosen
	inside the torus and with valuation vector \(w_v\) at every \(v\), after
	clearing denominators and, if necessary, replacing the ground field by a finite
	extension.  This is the special-point criterion of
	\cite[Proposition 5.13]{BGPS19}; the construction in its proof produces
	\(q^{1/\ell}\) from an \(S\)-unit \(q\) with
	\(\val_v(q)=\ell w_v\).  The passage from quasi-canonical metrics to the
	monocritical hypothesis of that proposition, and the harmless finite-extension
	interpretation, are recorded in Proposition~\ref{input:bgps-qcanonical-bridge}.
\end{proposition}

\begin{lemma}[Rational valuation exactness for subtori]
	\label{lem:rational-valuation-subtorus}
	Let \(T_B\subseteq T\) be a saturated subtorus.  For every place \(v\), the
	valuation groups of algebraic points are exact after tensoring with
	\(\mathbb Q\).  In particular, if
	\[
	\eta\in \val_v(T(\overline K))\otimes_{\mathbb Z}\mathbb Q
	\]
	and \(\eta\in \iota_N(N_{B,\mathbb R})\), then the unique
	\(\xi\in N_{B,\mathbb R}\) with \(\iota_N(\xi)=\eta\) satisfies
	\[
	\xi\in \val_v(T_B(\overline K))\otimes_{\mathbb Z}\mathbb Q.
	\]
	This is the valuation form of exactness for the short exact sequence of
	algebraic tori
	\[
	1\to T_B\to T\to T/T_B\to1,
	\]
	using that the character lattice quotient is torsion-free.
\end{lemma}

\begin{proof}
	Let
	\[
	\Gamma_v=\val_v(\overline K^{\times})\subseteq\mathbb R .
	\]
	For a split torus with cocharacter lattice \(N\), the valuation group of its
	algebraic points is
	\[
	\val_v(T(\overline K))=N\otimes_{\mathbb Z}\Gamma_v
	\subseteq N_{\mathbb R},
	\]
	and similarly
	\[
	\val_v(T_B(\overline K))=N_B\otimes_{\mathbb Z}\Gamma_v
	\subseteq N_{B,\mathbb R}.
	\]
	Since \(T_B\subseteq T\) is saturated, the quotient \(N/N_B\) is torsion-free.
	Thus
	\[
	0\longrightarrow N_B\longrightarrow N\longrightarrow N/N_B\longrightarrow0
	\]
	remains exact after tensoring with the \(\mathbb Q\)-vector space
	\(\Gamma_v\otimes_{\mathbb Z}\mathbb Q\).  Hence
	\[
	\bigl(N_B\otimes_{\mathbb Z}\mathbb R\bigr)\cap
	\bigl(N\otimes_{\mathbb Z}\Gamma_v\otimes_{\mathbb Z}\mathbb Q\bigr)
	=
	N_B\otimes_{\mathbb Z}\Gamma_v\otimes_{\mathbb Z}\mathbb Q
	\]
	inside \(N_{\mathbb R}\).  This is exactly the asserted descent of rational
	valuation vectors from \(T\) to \(T_B\).
\end{proof}

\begin{proposition}[Special points restrict]\label{prop:special-restrict}
	Assume that \(G(\overline{K})\) contains \(\bL\)-special points and that
	\(\muess_{\bL}(X)=\mu_G\).  Then the minimal envelope \(Y=a+B\) contains an
	\(\bL\)-special point.
\end{proposition}

\begin{proof}
	By Proposition~\ref{input:toric-special-point}, applied to the quasi-canonical
	replacement from Proposition~\ref{prop:monocritical-qcanonical-replacement},
	the existence of \(\bL\)-special points on \(G\) is equivalent to
	\[
	u_v\in \val_v(T(\overline{K}))\otimes_{\BZ}\BQ
	\]
	for every \(v\).  Since \(t_a\in T(\overline{K})\), also
	\[
	\alpha_v=\val_v(t_a)\in \val_v(T(\overline{K})).
	\]
	By Lemma \ref{lem:tropical-envelope},
	\[
	u_v-\alpha_v=\iota_N(w_v).
	\]
	Since \(\alpha_v=\val_v(t_a)\), the vector \(u_v-\alpha_v\) also lies in
	\(\val_v(T(\overline K))\otimes_{\BZ}\BQ\).  Applying
	Lemma~\ref{lem:rational-valuation-subtorus} to the saturated subtorus
	\(T_B\subseteq T\), the rational valuation condition descends to
	\[
	w_v\in \val_v(T_B(\overline{K}))\otimes_{\BZ}\BQ.
	\]
	Proposition~\ref{input:toric-special-point} applied to
	\(\bD'_{a,B}\) gives a point \(t_B\in T_B(\overline{K})\) whose valuations are
	\(\val_v(t_B)=w_v\) and whose height realizes the toric minimum for the
	quasi-canonical restricted metric.  For the original restricted metric
	\(\bD_{a,B}\), the local valuation argument is
	\[
	\iota_N(w_v)+\alpha_v=u_v.
	\]
	Hence the toric height of \(t_B\) for the restricted original metric equals
	the ambient toric absolute minimum:
	\[
	h_{\overline B(\bD_{a,B})}(t_B)=\mu_G.
	\]
	Here the abelian canonical part contributes \(0\) at the minimum, so the toric
	absolute minimum is the same number as the ambient minimum \(\mu_G\).
	The base point \(a t_B\) has torsion image in \(A\), because \(\pi(a)\) is
	torsion and \(t_B\) lies in the torus.  Hence
	\[
	h_{\bL}(a t_B)=\mu_G.
	\]
	Thus \(a t_B\in Y(\overline{K})\) is \(\bL\)-special.
\end{proof}

\begin{lemma}[Special translates have minimal essential minimum]\label{lem:special-translate-min}
	Let \(p\in G(\overline{K})\) be an \(\bL\)-special point and let \(B\subseteq G_{\overline{K}}\)
	be a connected algebraic subgroup.  Then
	\[
	\muess_{\bL}(p+B)=\mu_G.
	\]
\end{lemma}

\begin{proof}
	The global lower bound gives
	\[
	\muess_{\bL}(p+B)\ge\mu_G.
	\]
	It remains to produce a Zariski-dense set of points of height \(\mu_G\) in
	\(p+B\).
	
	By the height decomposition
	\[
	h_{\bL}(p)=h_{\overline G(\bD)}(p)+\hat h_{\bN}(\pi(p)),
	\]
	together with \(h_{\overline G(\bD)}(p)\ge\mu_G\) and
	\(\hat h_{\bN}\ge0\), the equality \(h_{\bL}(p)=\mu_G\) forces
	\(\hat h_{\bN}(\pi(p))=0\).  Hence \(\pi(p)\) is torsion.
	Choose a torsion lift \(q\in G(\overline{K})\) of \(\pi(p)\) and write
	\[
	p=q\,t
	\]
	with \(t\in T(\overline{K})\).  Corollary~\ref{cor:torsion-lift-height-invariance}
	gives
	\[
	h_{\bD}(t)=h_{\bL}(p)=\mu_G.
	\]
	
	Torsion points are Zariski dense in the connected semiabelian variety \(B\).
	For every torsion point \(\xi\in B(\overline{K})\), the point \(q\xi\) is torsion in
	\(G\), and
	\[
	p\xi=(q\xi)t.
	\]
	Applying Corollary~\ref{cor:torsion-lift-height-invariance} again gives
	\[
	h_{\bL}(p\xi)=h_{\bD}(t)=\mu_G.
	\]
	Thus \(p+B\) contains a Zariski-dense set of points of height \(\mu_G\), proving
	the reverse inequality.
\end{proof}

\subsection{Restriction Package and Bogomolov}

\begin{theorem}[Restriction package sufficient for Bogomolov]\label{thm:restriction-package}
	Let \(Y=a+B\) be the minimal translate of a connected subgroup containing
	\(X\), and assume
	\[
	\muess_{\bL}(X)=\mu_G.
	\]
	Then:
	\begin{enumerate}[label=(\arabic*),leftmargin=2em]
		\item \(a\) may be chosen with \(\pi(a)\) torsion.
		\item The ambient critical vector satisfies
		\[
		u_v-\val_v(t_a)\in \iota_N(N_{B,\BR})
		\]
		for every \(v\).
		\item The quasi-canonical metric from
		Proposition~\ref{prop:monocritical-qcanonical-replacement} restricts to a
		quasi-canonical
		metric on \(B\), with critical vector \(w_v\) defined by
		\[
		\iota_N(w_v)+\val_v(t_a)=u_v.
		\]
		\item If \(X\neq Y\), any \(X\)-generic small sequence is strict in \(Y\) and
		equidistributes toward a measure with Zariski-dense support in \(Y\).
		\item If \(G\) has \(\bL\)-special points, then \(Y\) has \(\bL\)-special
		points.
	\end{enumerate}
\end{theorem}

\begin{proof}
	The assertions are Lemma \ref{lem:torsion-base}, Lemma
	\ref{lem:tropical-envelope}, Proposition \ref{prop:restriction-metric},
	Proposition \ref{prop:restricted-eq}, and Proposition
	\ref{prop:special-restrict}.
\end{proof}

\begin{theorem}[Bogomolov for the paper's metric class]\label{thm:bogomolov}
	Let \(X\subseteq G_{\overline{K}}\) be irreducible, and assume that the ambient
	toric metric \(\bD\) satisfies the standing hypotheses of the Bogomolov part:
	\(D\) is ample and \(T\)-effective, and \(\bD\) is semipositive, monocritical, and
	arithmetically \(T\)-effective.
	At every archimedean place, assume in addition either the standard
	\((\mathbb P^1)^t\)-compactification used by K\"uhne, or, for every subvariety
	occurring in the difference-morphism argument, the three general-fan inputs of
	Subsection~9.5: the product-corner normal form, a primitive frame in the
	saturated active lattice, and absence of additional boundary mass.  Equivalently,
	the conditional curvature--Haar slice formula
	\ref{input:general-fan-curvature-haar-slice}, with its normalization and
	boundary assertions, is required for all those subvarieties.
	\begin{enumerate}[label=(\arabic*),leftmargin=2em]
		\item If
		\[
		\muess_{\bL}(X)=\mu_G,
		\]
		then \(X\) is special.
		\item If \(G(\overline{K})\) contains \(\bL\)-special points, then
		\[
		\muess_{\bL}(X)=\mu_G
		\]
		if and only if \(X\) is \(\bL\)-special.
	\end{enumerate}
\end{theorem}

\begin{proof}
	By Proposition~\ref{prop:kuhne-section4-transport} and
	Corollaries~\ref{cor:kuhne-transport-qbog-qsec} and
	\ref{cor:kuhne-transport-paper-metric-class}, the restricted quasi-canonical
	strong equidistribution theorem applies to every restricted quasi-canonical
	semiabelian metric \(\bL'_{a,B}\) attached to a minimal translate occurring
	below.
	
	Let \(Y=a+B\) be the minimal special envelope of \(X\).
	
	Assume first that \(\muess_{\bL}(X)=\mu_G\).  If \(X\neq Y\), choose an
	\(X\)-generic \(\bL\)-small sequence.  By
	Theorem~\ref{thm:qcanonical-strong-equidistribution},
	Theorem~\ref{thm:restriction-package} applies, and this sequence is strict in
	\(Y\) and
	equidistributes toward a measure with Zariski-dense support in \(Y\).  Since all
	terms of the sequence lie in \(X\), every weak limit is supported on
	\(X_v^{an}\).  Hence \(X\) must be Zariski dense in \(Y\), contradicting
	\(X\neq Y\).  Therefore \(X=Y\), so \(X\) is special.  This proves (1).
	
	For (2), first suppose that \(X=p+B\) with \(p\) \(\bL\)-special.  Lemma
	\ref{lem:special-translate-min} gives
	\[
	\muess_{\bL}(X)=\mu_G.
	\]
	
	Conversely assume \(\muess_{\bL}(X)=\mu_G\) and that \(G(\overline{K})\) contains
	\(\bL\)-special points.  By (1), \(X\) equals its minimal special envelope
	\(Y=a+B\).  The same restricted strong equidistribution theorem and
	Theorem~\ref{thm:restriction-package} then say that \(Y\) contains an
	\(\bL\)-special point.  Hence \(X\) is \(\bL\)-special.
\end{proof}

\appendix
\section{Further Directions: A Relative Toric Route}
\label{app:relative-route-future-work}

The proof of Theorem~\ref{thm:bogomolov} proceeds through the direct
K\"uhne-operation transport package,
Proposition~\ref{prop:kuhne-section4-transport}, together with the
quasi-canonical Bogomolov and strong equidistribution results proved in the
main text.  A natural continuation is to develop a relative version over the
abelian quotient.  We record the main ingredients of such a route.

\begin{enumerate}[label=\textup{(\roman*)},leftmargin=2em]
	\item A relative toric fiber-specialness theorem.  Smallness on the
	semiabelian variety should force the generic toric fiber, after quotienting by
	its maximal toric stabilizer, to be special in the corresponding
	function-field torus.
	
	\item A global-to-generic-fiber height comparison.  The global height gap on
	the semiabelian variety should control a positive generic-fiber height gap
	after restricting to a suitable open subset of the abelian base.
	
	\item A function-field toric Bogomolov theorem with descent.  The generic-fiber
	height gap should descend through the relevant torus torsor and identify the
	quotient sections which are forced by small points.
\end{enumerate}

The first ingredient belongs to a relative equidistribution/Bogomolov problem
over the abelian base; the second is a height-specialization problem; the
third is a function-field toric Bogomolov/descent problem.  Together they
would give a fiberwise approach to specialness in semiabelian varieties, in
which the generic toric fiber detects the toric component of a small sequence.

Such a project could produce a relative version of the present theorem, where
small points are studied fiberwise over the abelian quotient and specialness is
detected on the generic toric fiber.

	

\begin{thebibliography}{99}
		
		\bibitem{Ku1}
		Lars K{\"u}hne.
		\newblock Points of small height on semiabelian varieties.
		\newblock {\em J. Eur. Math. Soc.}, 24(6):2077--2131, 2022.
		
		\bibitem{SankaranUma2003}
		Parameswaran Sankaran and V. Uma.
		\newblock Cohomology of toric bundles.
		\newblock {\em Comment. Math. Helv.}, 78(3):540--554, 2003.
		
		\bibitem{Ku4}
		Lars K{\"u}hne.
		\newblock Equidistribution in families of abelian varieties and uniformity, 2021.
		\newblock arXiv:2101.10272.
		
		\bibitem{Iko13}
		Hideaki Ikoma.
		\newblock Boundedness of the successive minima on arithmetic varieties.
		\newblock {\em J. Algebraic Geom.}, 22(2):249--302, 2013.
		
		\bibitem{CLHeight}
		Antoine Chambert-Loir.
		\newblock G\'eom\'etrie d'Arakelov et hauteurs canoniques sur des vari\'et\'es semi-ab\'eliennes.
		\newblock {\em Math. Ann.}, 314:381--401, 1999.
		
		\bibitem{Lang}
		Serge Lang.
		\newblock {\em Introduction to Algebraic and Abelian Functions}.
		\newblock Springer-Verlag, 1982.
		
		\bibitem{ZhangAdelicMetrics}
		Shou-Wu Zhang.
		\newblock Small points and adelic metrics.
		\newblock {\em J. Algebraic Geom.}, 4:281--300, 1995.
		
		\bibitem{BS24}
		Fran{\c c}ois Balla{\"y} and Mart\'in Sombra.
		\newblock Approximation of adelic divisors and equidistribution of small points, 2024.
		\newblock arXiv:2407.14978.
		
		\bibitem{Hul24}
		Nuno Hultberg.
		\newblock Arakelov geometry of toric bundles: Okounkov bodies and BKK, 2024.
		\newblock arXiv:2412.04169.
		
		\bibitem{Ku2}
		Lars K{\"u}hne.
		\newblock The bounded height conjecture for semiabelian varieties.
		\newblock {\em Compos. Math.}, 156(7):1405--1456, 2020.
		
		\bibitem{BA}
		Fran{\c c}ois Balla{\"y}.
		\newblock Successive minima and asymptotic slopes in Arakelov geometry.
		\newblock {\em Compos. Math.}, 157(8):1302--1339, 2020.
		
		\bibitem{Gu}
		Walter Gubler.
		\newblock Local and canonical heights of subvarieties.
		\newblock {\em Ann. Sc. Norm. Super. Pisa Cl. Sci. (5)}, 2(4):711--760, 2003.
		
		\bibitem{BPS14}
		Jos{\'e} Ignacio Burgos Gil, Patrice Philippon, and Mart\'in Sombra.
		\newblock {\em Arithmetic geometry of toric varieties. Metrics, measures and heights}.
		\newblock Ast\'erisque 360, Soci\'et\'e Math\'ematique de France, 2014.
		
		\bibitem{Ber}
		Vladimir G. Berkovich.
		\newblock {\em Spectral Theory and Analytic Geometry over Non-Archimedean Fields}.
		\newblock Math. Surveys Monogr. 33, Amer. Math. Soc., 2012.
		
		\bibitem{FRSS}
		Tyler Foster, Joseph Rabinoff, Farbod Shokrieh, and Alejandro Soto.
		\newblock Non-Archimedean and tropical theta functions.
		\newblock {\em Math. Ann.}, 372(3--4):891--914, 2017.
		
		\bibitem{BGPS19}
		Jos{\'e} Ignacio Burgos Gil, Patrice Philippon, Juan Rivera-Letelier, and Mart\'in Sombra.
		\newblock The distribution of Galois orbits of points of small height in toric varieties.
		\newblock {\em Amer. J. Math.}, 141(2):309--381, 2019.
		
		\bibitem{BMPS22}
		Jos{\'e} Ignacio Burgos Gil, Atsushi Moriwaki, Patrice Philippon, and Mart\'in Sombra.
		\newblock Arithmetic positivity on toric varieties.
		\newblock {\em J. Algebraic Geom.}, 25(4):201--272, 2016.
		
		\bibitem{BGPS15}
		Jos{\'e} Ignacio Burgos Gil, Patrice Philippon, and Mart\'in Sombra.
		\newblock Successive minima of toric height functions.
		\newblock {\em Ann. Inst. Fourier (Grenoble)}, 65(5):2145--2197, 2015.
		
		\bibitem{YZ}
		Xinyi Yuan and Shou-Wu Zhang.
		\newblock {\em Adelic Line Bundles on Quasi-Projective Varieties}.
		\newblock Annals of Mathematics Studies 221, Princeton University Press, 2026.
		
		\bibitem{WF}
		William Fulton.
		\newblock {\em Introduction to Toric Varieties}.
		\newblock Ann. of Math. Stud. 131, Princeton Univ. Press, 1993.
		
		\bibitem{Lan-FDG}
		Serge Lang.
		\newblock {\em Fundamentals of Diophantine Geometry}.
		\newblock Springer, 2013.
		
		\bibitem{CL1}
		Antoine Chambert-Loir.
		\newblock Points de petite hauteur sur les vari\'et\'es semi-ab\'eliennes.
		\newblock {\em Ann. Sci. \'Ecole Norm. Sup\'er. (4)}, 33(6):789--821, 1999.
		
		\bibitem{CL2}
		Antoine Chambert-Loir.
		\newblock Mesures et \'equidistribution sur les espaces de Berkovich.
		\newblock {\em J. Reine Angew. Math.}, 595:215--235, 2006.
		
		\bibitem{We}
		Andr{\'e} Weil.
		\newblock Arithmetic on algebraic varieties.
		\newblock {\em Ann. of Math. (2)}, 53(3):412--444, 1951.
		
		\bibitem{Ne}
		Andr{\'e} N{\'e}ron.
		\newblock Quasi-fonctions et hauteurs sur les vari\'et\'es ab\'eliennes.
		\newblock {\em Ann. of Math. (2)}, 82(2):249--331, 1965.
		
		\bibitem{Fal}
		Gerd Faltings.
		\newblock Endlichkeitss{\"a}tze f{\"u}r abelsche Variet{\"a}ten {\"u}ber Zahlk{\"o}rpern.
		\newblock {\em Invent. Math.}, 73(3):349--366, 1983.
		
		\bibitem{Ray}
		Michel Raynaud.
		\newblock Sous-vari\'et\'es d'une vari\'et\'e ab\'elienne et points de torsion.
		\newblock In {\em Arithmetic and Geometry: Papers Dedicated to I.R. Shafarevich Volume I: Arithmetic}, pages 327--352. Birkh{\"a}user, 1983.
		
		\bibitem{Ull}
		Emmanuel Ullmo.
		\newblock Positivit\'e et discr\'etion des points alg\'ebriques des courbes.
		\newblock {\em Ann. of Math. (2)}, 147(1):167--179, 1996.
		
		\bibitem{Zha}
		Shou-Wu Zhang.
		\newblock Equidistribution of small points on abelian varieties.
		\newblock {\em Ann. of Math. (2)}, 147(1):159--165, 1998.
		
		\bibitem{SUZ}
		Lucien Szpiro, Emmanuel Ullmo, and Shou-Wu Zhang.
		\newblock \'Equir\'epartition des petits points.
		\newblock {\em Invent. Math.}, 127(2):337--347, 1997.
		
		\bibitem{Yua1}
		Xinyi Yuan.
		\newblock Big line bundles over arithmetic varieties.
		\newblock {\em Invent. Math.}, 173(3):603--649, 2008.
		
		\bibitem{DGH}
		Vesselin Dimitrov, Ziyang Gao, and Philipp Habegger.
		\newblock Uniformity in Mordell--Lang for curves.
		\newblock {\em Ann. of Math. (2)}, 194(1):237--298, 2021.
		
		\bibitem{DGH2}
		Vesselin Dimitrov, Ziyang Gao, and Philipp Habegger.
		\newblock A consequence of the relative Bogomolov conjecture.
		\newblock {\em J. Number Theory}, 230:146--160, 2022.
		
		\bibitem{Gao}
		Ziyang Gao and Philipp Habegger.
		\newblock The relative Manin--Mumford conjecture, 2023.
		\newblock arXiv:2303.05045.
		
		\bibitem{TJFA}
		Siegfried Bosch and Werner L{\"u}tkebohmert.
		\newblock Degenerating abelian varieties.
		\newblock {\em Topology}, 30(4):653--698, 1991.
		
		\bibitem{Bal2}
		Fran{\c c}ois Balla{\"y} and Mart\'in Sombra.
		\newblock Approximation of adelic divisors and equidistribution of small points, 2025.
		\newblock arXiv:2407.14978.
		
		\bibitem{Hul}
		Nuno Hultberg.
		\newblock Arakelov geometry of toric bundles: Okounkov bodies and BKK, 2024.
		\newblock arXiv:2412.04169.
		
		\bibitem{Mil}
		James S. Milne.
		\newblock Abelian Varieties (v2.00), 2008.
		\newblock Available at www.jmilne.org/math/.
		
		\bibitem{CMDirichlet}
		Huayi Chen and Atsushi Moriwaki.
		\newblock Sufficient conditions for the Dirichlet property, 2017.
		\newblock arXiv:1704.01410.
		
		\bibitem{CMDynamics}
		Huayi Chen and Atsushi Moriwaki.
		\newblock Algebraic dynamical systems and Dirichlet's unit theorem on arithmetic varieties.
		\newblock {\em Int. Math. Res. Not. IMRN}, 2015(24):13669--13716, 2015.
		
		\bibitem{CMAdelicCurves}
		Huayi Chen and Atsushi Moriwaki.
		\newblock {\em Arakelov Geometry over Adelic Curves}.
		\newblock Lecture Notes in Mathematics 2258, Springer, 2020.
		
	\end{thebibliography}
\end{document}